\documentclass{amsart}
\usepackage[utf8]{inputenc}
\usepackage{mathrsfs}

\usepackage{amsmath,amssymb,amsthm,tikz}
\usepackage{fullpage}
\usetikzlibrary{decorations.pathreplacing}
\usepackage[hidelinks]{hyperref}

\newtheorem{introthm}{Theorem}
\newtheorem{theorem}{Theorem}[section]
\newtheorem{lemma}[theorem]{Lemma}
\newtheorem{proposition}[theorem]{Proposition}
\newtheorem{definition}[theorem]{Definition}
\newtheorem{corollary}[theorem]{Corollary}
\newtheorem{claim}[theorem]{Claim}

\newtheorem*{remark*}{Remark}

\newcommand{\NE}{\operatorname{NE}}
\newcommand{\Sing}{\mathrm{Sing}}
\newcommand{\Reg}{\mathrm{Reg}}
\newcommand{\GGG}{\mathcal{G}}

\newcommand{\SSS}{S}

\newcommand{\GS}{G\SSS}

\newcommand{\CCC}{\mathcal{C}}

\newcommand{\eps}{\varepsilon}
\def\diam{\operatorname{diam}}

\title{Local product structure for equilibrium states \\ of geodesic flows and applications}

\author{Benjamin Call, David Constantine, Alena Erchenko,\\ Noelle Sawyer, and Grace Work}

\newcommand{\Addresses}{{
\bigskip
  \footnotesize

  \textsc{Department of Mathematics, Statistics, and Computer Science, University of Illinois at Chicago, Chicago, IL}
	\par\nopagebreak
  \textit{E-mail address:} 
  \texttt{bcall@uic.edu}
  \bigskip
  \footnotesize

   \textsc{Mathematics and Computer Science Department, Wesleyan University, Middletown, CT}
	\par\nopagebreak
  \textit{E-mail address:} 
  \texttt{dconstantine@wesleyan.edu}
  \bigskip
  \footnotesize

        \textsc{Department of Mathematics, Dartmouth College, Hanover, NH}
	\par\nopagebreak
  \textit{E-mail address:} 
  \texttt{alena.erchenko@dartmouth.edu}
  \bigskip
  \footnotesize

    \textsc{Mathematics and Computer Science Department, Southwestern University, Georgetown, TX}
	\par\nopagebreak
  \textit{E-mail address:} 
  \texttt{sawyern@southwestern.edu}
  \bigskip
  \footnotesize

    \textsc{Department of Mathematics, University of Wisconsin-Madison, Madison, WI}
	\par\nopagebreak
  \textit{E-mail address:}
  \texttt{work2@wisc.edu}
  \bigskip
  \footnotesize

}}

\begin{document}

\begin{abstract}
Let $\SSS$ be a compact surface of genus $\geq 2$ equipped with a metric that is flat everywhere except at finitely many cone points with angles greater than $2\pi$. We examine the geodesic flow on $\SSS$ and prove local product structure for a wide class of equilibrium states. Using this, we establish the Bernoulli property for these systems. We also establish local product structure for a similar class of equilibrium states for geodesic flows on rank 1, nonpositively curved manifolds.
\end{abstract}

\maketitle

%
\section{Introduction}\label{sec:intro}

In this paper we develop a technique for obtaining local product structure of equilibrium states using the non-uniform Gibbs property, following an idea of Climenhaga for Anosov diffeomorphisms with the uniform Gibbs property \cite{ClimenhagaGibbs}. This is applied for a wide class of equilibrium states for the geodesic flow on translation surfaces (Section~\ref{subsec:lps}) and rank 1, nonpositively curved manifolds (Appendix~\ref{section: Appendix}). More precisely, we consider equilibrium states for potential functions which are locally constant on a neighborhood of a `singular' set for the flow -- a set of geodesics which do not experience any of the hyperbolic dynamics of the flow. We expect that the techniques used can be extended beyond the discussed settings to some other applications of the Climenhaga-Thompson decomposition machinery \cite{Climenhaga-Thompson}.

Using the classical machinery \cite{OW}, the local product structure allows us to improve the K-property to Bernoulli. We recall that the K-property for these potentials on translation surfaces was shown previously in \cite{CCESW} and for rank 1, nonpositively curved manifolds in \cite{CT19}. Moreover, \cite{ALP} established the Bernoulli property in the latter setting bypassing the establishment of the local product structure of the equilibrium states. Therefore, we provide the details for obtaining the Bernoulli property only for the translation surface case. However, it can be also adapted for the rank 1, nonpositively curved manifolds case. More history on the development of the subject can be found in Section~\ref{subsec:history}.

Towards our main results, we prove that the considered potentials have the global Bowen property, an improvement from the non-uniform Bowen property previously established in \cite{CCESW} and \cite{BCFT} for the respective cases, which may be of independent interest.

%
\subsection{Previous results}\label{subsec:history}

Local product structure of equilibrium states associated to any H\"older continuous potential for Axiom A flows was obtained by Haydn \cite{Haydn} using Markov partitions and symbolic dynamics, by Leplaideur \cite{Leplaideur} using a geometric method that relates the considered equilibrium states to other equilibrium states of some special sub-systems satisfying expansiveness, and by Climenhaga \cite{ClimenhagaDimensionTheory} using dimension theory. Former results in this direction are \cite{Margulis, Hamenstadt, Hasselblatt, Kaimanovich, BowenMarcus, Series}.

The first result showing geodesic flows are Bernoulli was Ornstein and Weiss's proof that geodesic flows on compact, negatively curved surfaces are Bernoulli with respect to the Liouville measure \cite{OW}. This work builds on Ornstein's work providing the first examples of Bernoulli flows in \cite{Ornstein_imbedding}, and provides a framework for showing a flow is Bernoulli that has frequently been followed, including in the present paper. (See section \ref{subsec:outline} for an outline of this argument.) Ratner noted that the Ornstein-Weiss argument could be applied to transitive, $C^2$ Anosov flows with Gibbs measures \cite{ratner}.

One direction in which to generalize these results is to situations with non-uniform hyperbolicity. In \cite{Pesin_geodesic, Pesin_characteristic}, Pesin relaxed the negative curvature assumptions of \cite{OW} to surfaces with no focal points. Burns and Gerber showed that even in some situations with significant positive curvature, specifically for some special metrics on 2-spheres, the geodesic flow is still Bernoulli with respect to the Liouville measure \cite{Burns_Gerber}.  As a corollary of the work of Ledrappier, Lima, and Sarig \cite{LLS}, the measure of maximal entropy is Bernoulli for the geodesic flow on any closed surface whose curvature is not identically zero. Call and Thompson prove that the measure of maximal entropy is Bernoulli for closed, nonpositively curved, rank-1 manifolds \cite{CT19}. Using countable Markov partitions, Araujo, Lima, and Poletti established the Bernoulli property for equilibrium states for H\"older continuous potentials with the pressure gap as in \cite{Climenhaga-Thompson} and some multiples of the geometric potential in the same setting \cite{ALP}.

Chernov and Haskell proved that non-uniformly hyperbolic flows with the $K$-property satisfying a list of properties are Bernoulli \cite{CH}. Similarly, Alansari proves that the $K$-property implies Bernoulli in the setting of hyperbolic measures with local product structure \cite{Alansari}.

A second direction of research relevant to our work here has been geodesic flows on non-manifold spaces. Roblin originally showed the existence of the measure of maximal entropy for geodesic flow in the CAT(-1) setting \cite{roblin}. Constantine, Lafont and Thompson proved the geodesic flow on a compact, locally CAT(-1) space is Bernoulli (or constructed from a Bernoulli shift) for equilibrium states \cite{CLT-strong}. In \cite{BAPP}, Broise-Alamichel, Parkonnen, and Paulin present extensive results on this setting; they significantly weaken the compactness assumption, but do need an assumption on the types of potential function allowed. Recent work in the non-compact direction but without this restriction on the potentials is due to Dilsavor and Thompson \cite{DT}.

%
\subsection{Statement of results}\label{subsec:statements}

Let $\SSS$ be a compact surface of genus $g\geq 2$. Equip $\SSS$ with a flat Riemannian metric away from a finite collection of `cone points,' at each of which the total angle is greater than $2\pi$. (See \S\ref{sec:preliminaries} for precise definitions). Let $g_t$ be the geodesic flow on $G\SSS$, the space of geodesics on $\SSS$.

Given a potential function $\phi:G\SSS\to \mathbb{R}$, one can hope to find equilibrium states for $\phi$. In our previous paper \cite{CCESW}, we proved that if $\phi$ is H\"older and satisfies a pressure gap condition (see \cite[Theorem~A]{CCESW}), then it has a unique equilibrium state $\mu$ that satisfies the $K$-property. We also show (\cite[Theorem~B]{CCESW}) that potentials that are locally constant on some neighborhood of the singular set for the flow satisfy this pressure gap condition, providing a more easily checked condition on $\phi$.

In this paper we improve these results for locally constant potentials by proving more refined properties of the dynamics of $(G\SSS, g_t, \mu).$

We say a measure $\mu$ on $\GS$ has \emph{local product structure} if at small scales, $\mu$ is equivalent to a product measure, constructed using a (geometric) local product structure based on the unstable and weak stable foliations. (See Sections \ref{subsec:stable unstable} and \ref{subsec:lps} for precise definitions.) Our first main theorem is the following:

\begin{introthm}\label{mainthm:LPS}
    Let $(\GS, \mu, g_t)$ be as above, where $\mu$ is the equilibrium state for a H\"older potential which is locally constant on some neighborhood of the singular set for $g_t$. For any $\eps>0$ there is a subset $E$ of $\GS$ with $\mu(E)>1-\eps$ on which the measure $\mu$ has local product structure.
\end{introthm}

A measure-preserving flow $(X, \mu, \{\Phi_t\})$ is said to be \emph{Bernoulli} if for all $t\neq 0$ the discrete dynamical system $(X, \mu, \Phi_t)$ is isomorphic to a Bernoulli shift. Our second main theorem is the following:

\begin{introthm}\label{mainthm:Bernoulli}
    $(\GS, \mu, g_t)$ as in Theorem \ref{mainthm:LPS} is Bernoulli.
\end{introthm}

In the process of proving these theorems, we examine the Bowen property for the potential $\phi$ relative to our flow (see Definition~\ref{defn:global bowen}). The Bowen property for some collection $\mathcal{C}$ of orbit segments says, roughly, for any two segments from $\mathcal{C}$ of equal length which stay close together, the integrals of $\phi$ along the segments are close, uniformly over all of $\mathcal{C}$ and over all lengths of the segments. In \cite{CCESW} we proved this property for a large collection of `good' segments -- those that experience a definite amount of the hyperbolicity of the flow. (See Prop \ref{prop:non-uniform Bowen} for the details.) 

In the present paper we improve this to a \emph{global} Bowen property, where $\mathcal{C}$ can be taken to be all orbit segments. The argument uses the technology of $\lambda$-decompositions, which has been the main technical tool for applications of the Climenhaga-Thompson machinery. Hence, this argument, in particular Proposition~\ref{proposition: global Bowen property}, may be of independent interest in other settings.

\begin{introthm}
For $(\GS, \mu, g_t)$ as above, and for $\phi$ which is H\"older and is locally constant on $B(\Sing,\delta)$ for some $\delta>0$,  $\phi$ has the global Bowen property.
\end{introthm}

In an Appendix, we adapt Theorem A to geodesic flows on nonpositively curved, rank 1 manifolds.

\begin{introthm}\label{thm:rank 1}
Let $M$ be a closed, connected, Riemannian manifold which is rank 1 with nonpositive sectional curvature. Let $(T^1M, \mu, g_t)$ be the geodesic flow on its unit tangent bundle, together with an equilibrium measure $\mu$ for a H\"older continuous potential which is constant on a neighborhood of the singular set for $g_t$. Then for any $\eps>0$ there is a subset $E$ of $\GS$ with $\mu(E)>1-\eps$ on which the measure $\mu$ has local product structure.
\end{introthm}
%
\subsection{An outline of the paper}\label{subsec:paper outline}

The outline of this paper is as follows. 

In \S\ref{sec:preliminaries}, we introduce the necessary geometric background, beginning with some basic geometric lemmas (\S\ref{subsec:geometric lemmas}) as well as developing properties of the stable and unstable sets in our setting (\S \ref{subsec:stable unstable}). In \S \ref{subsec:Bowen brackets} and \S\ref{sec:rectangles}, we introduce the Bowen bracket operations and exploit the local (topological) product structure of $\GS$.

In \S\ref{sec:dynamics} we prove some of our main results on the dynamics of the geodesic flow. In \S\ref{subsec:global Bowen}, we show that the potentials considered in this paper satisfy the Bowen property globally, rather than the non-uniform version. Then, in \S\ref{subsec:Gibbs}, we recall the Gibbs property developed for the Climenhaga-Thompson machinery, and show how it interacts with Bowen balls in our setting. Finally, in \S\ref{subsec:lps}, we establish the local product structure of our unique equilibrium states.

\S\ref{sec:Bernoulli} is devoted to the proof of the Bernoulli property. After establishing some preliminary definitions in \S\ref{subsec:defs2}, we give an outline of the steps of an argument for the Bernoulli property due to Ornstein and Weiss \cite{OW} in \S\ref{subsec:outline}. \S\ref{subsec:1-2} covers the first two steps of this outline, which mainly involve recalling Lemmas from previous papers and the local product structure proved at the end of \S\ref{sec:dynamics}. In \S\ref{subsec:3} we prove a ``layerwise'' intersection property for a partition which is needed in the proof. \S\ref{subsec:4} and \S\ref{subsec:5} conclude the proof.

Appendix \ref{section: Appendix} turns to the smooth setting and presents the proof of Theorem \ref{thm:rank 1}. The proof follows the arguments of \S\ref{sec:dynamics}. In \S\ref{sec:geometric preliminaries} through \S\ref{sec: NPC Bowen} we discuss the tools needed for the proof, and in \S\ref{sec:proof of LPS rank 1} we prove Theorem \ref{thm:rank 1}. Our exposition primarily highlights the adjustments needed to the proof of Theorem \ref{mainthm:LPS} in the rank 1 setting.

%
\section*{Acknowledgements}

We would like to thank the American Institute of Mathematics for their hospitality and support during the SQuaRE ``Thermodynamic formalism for CAT(0) spaces". We would also like to thank Vaughn Climenhaga for helpful comments regarding the presentation of the paper. AE was supported by the Simons Foundation (Grant Number 814268 via the Mathematical Sciences Research Institute, MSRI) and NSF grant DMS-2247230. BC was partially supported by NSF grant DMS-2303333.

%
\section{Preliminaries}\label{sec:preliminaries}

In this section we collect a few preliminaries which will be used throughout the proof.

%
\subsection{Definitions}\label{subsec:defs}

First, we define the main objects of study.

Throughout, $\SSS$ is a flat surface with finitely many large cone angle singularities. That is, $\SSS$ is a compact, connected surface of genus $g\geq 2$ equipped with a metric which is flat except at finitely many points, called \emph{cone points} and denoted by $Con$. The total angle around $p\in Con$ is denoted $\mathcal{L}(p)$ and satisfies $\mathcal{L}(p)>2\pi.$ We denote by $\theta_0$ the minimum of $\mathcal{L}(p)-2\pi$ over all $p\in Con.$

$\tilde \SSS$ is the universal cover of $\SSS$, and is a CAT(0) space. (See \cite[Part II]{bh} for definitions and basic results on CAT(0) spaces.) $d_\SSS$ and $d_{\tilde \SSS}$ are the metrics on $\SSS$ and $\tilde \SSS$, respectively. The unique geodesic segment joining points $\tilde p, \tilde q\in \tilde \SSS$ is denoted $[\tilde p, \tilde q].$

The geodesic flow takes place on the space of geodesics:
\[ GS = \{ \gamma: \mathbb{R}\to \SSS: \gamma \mbox{ is a local isometry}\}\]
\[ G\tilde S = \{ \tilde \gamma: \mathbb{R}\to \tilde \SSS: \tilde \gamma \mbox{ is a local isometry}\}.\]
We endow these spaces with the following metrics:
\[ d_{G\tilde \SSS}(\tilde\gamma_1,\tilde\gamma_2)= \int_{-\infty}^\infty d_{\tilde \SSS}(\tilde \gamma_1(s),\tilde \gamma_2(s))e^{-2|s|}ds, \]
\[ d_{G\SSS}(\gamma_1, \gamma_2) = \inf_{\tilde\gamma_1,\tilde \gamma_2} d_{G\tilde \SSS}(\tilde\gamma_1, \tilde\gamma_2)\]
where the infimum is taken over all lifts of $\gamma_1$ and $\gamma_2$ to $G\tilde \SSS$. The geodesic flow is defined by $g_t(\gamma(s))=\gamma(s+t)$.

The geometric feature that drives the dynamical properties we study is the ability of geodesics to turn with various angles when they encounter cone points. 

\begin{definition}\label{defn:turning angle}
The \emph{turning angle} of $\gamma$ at time $t$ is $\theta(\gamma,t)\in (-\frac{1}{2}\mathcal L(\gamma(t)),\frac{1}{2}\mathcal L(\gamma(t))]$ and is the signed angle between the segments $[\gamma(t-\delta),\gamma(t)]$ and $[\gamma(t),\gamma(t+\delta)]$ (for sufficiently small $\delta>0$). A positive (resp. negative) sign for $\theta$ corresponds to a  counterclockwise (resp. clockwise) rotation with respect to the orientation of  $[\gamma(t-\delta),\gamma(t)]$.
\end{definition}

Geodesics that do not encounter cone points do not experience the hyperbolic dynamics caused by those points and hence are (from a dynamical perspective) singular:

\begin{definition}\label{defn:Sing}
The \emph{singular geodesics} are
\begin{equation*}
\Sing =\{\gamma \in \GS : |\theta(\gamma,t)|=\pi \quad \forall t\in\mathbb R\}.
\end{equation*}
\end{definition}

$\Sing$ is $g_t$-invariant, closed, and, hence, compact. Some, but not all, singular geodesics lie in flat strips:
\begin{definition}\label{defn:flat strip}
A \emph{flat half-strip} in $\tilde\SSS$ is an isometric embedding of $[0,\infty)\times[a,b]$ into $\tilde\SSS$ for some $[a,b]$. A \emph{flat strip} is an isometric embedding of $\mathbb{R}\times[a,b]$ into $\tilde\SSS$ for some $[a,b]$.
\end{definition}

In our previous paper \cite{CCESW}, we defined a function $\lambda: G\SSS \to [0,\infty)$ which played an essential role in studying the dynamics of the geodesic flow. Roughly speaking, $\lambda(\gamma)$ measures the hyperbolicity of the dynamics experienced by $\gamma$: a large value of $\lambda$ means that at some point near time zero, $\gamma$ turns with an angle significantly different from $\pi$ at a cone point. For the details of the definition of $\lambda$, see \cite[\S 3]{CCESW}. Here we note the key properties of $\lambda$ we will need:
\begin{itemize}
    \item $\bigcap_{t\in\mathbb{R}}g_t\lambda^{-1}(0) = \Sing$ (\cite[Cor 3.5]{CCESW}).
    \item $\lambda$ is lower semicontinuous (\cite[Lemma 3.8]{CCESW}).
    \item Let $s_0$ be a positive parameter in the definition of $\lambda$. We recall that it depends only on the geometry of $\SSS$, specifically the minimum distance between cone points in $\SSS$. Then, the following holds.
    \begin{lemma}\label{lemma: properties of lambda}(\cite{CCESW}, Proposition 3.9)
    If $\lambda(x) > c>0$, then either:
    \begin{enumerate}
        \item There exists $t \in (-s_0,s_0)$ such that x turns at $t$ with angle at least $cs_0$, or
        \item There exist $t_1\in (-\frac{\theta_0}{2c},-s_0]$ and $t_2\in [s_0,\frac{\theta_0}{2c})$ such that $x$ first turns at $t_1$ and $t_2$ with angle at least $cs_0$, where $\theta_0$ depends only on the geometry of $\SSS$.
    \end{enumerate}
    \end{lemma}
\end{itemize}

Geodesics not in $\Sing$ are \emph{regular}, and the value of $\lambda$ measures their regularity:

\begin{definition}\label{defn: regular}
For all $c > 0$, let $$\operatorname{Reg}(c) := \{\gamma\in GS\mid \lambda(\gamma) \geq c\}.$$    
\end{definition}

As we will see below, for many of our arguments about geodesic segments, only their level of regularity at the beginning and end are important.

\begin{definition}\label{defn: good start and end}
Let $\CCC(c) = \{(x,t) \in \GS \times \mathbb{R} \mid \lambda(x) \geq c \text{ and } \lambda(g_tx) \geq c \}$ be the set of orbit segments with large values of $\lambda$ at their beginning and end.
\end{definition}

$\partial_\infty\tilde\SSS$ is the boundary at infinity of $\tilde \SSS$ (see \cite[\S II.8]{bh}). If $\tilde\gamma\in G\tilde \SSS$, then $\tilde \gamma(\pm\infty)$ are its forward and backwards endpoints at infinity. Geodesics that share an endpoint at infinity either bound a flat half-strip or are asymptotic in the following sense:

\begin{definition}\label{def: cone point asymptotic}
We say that geodesics $\tilde\gamma_1,\tilde\gamma_2\in\tilde\SSS$ are \emph{$*$-cone point asymptotic}, where $*\in\{+,-\}$, if $\tilde\gamma_1(*\infty)=\tilde\gamma_2(*\infty)=\xi$ and there is a cone point $p$ such that the ray defined by $p$ and $\xi$ is contained in both $\tilde\gamma_1$ and $\tilde\gamma_2$.
\end{definition}

%
\subsection{Geometric lemmas}\label{subsec:geometric lemmas}

It is useful to relate $d_{\GS}$ to $d_\SSS$, the metric on the surface itself. First, if two geodesics are close in $\GS$, then they are close in $\SSS$ at time zero.

\begin{lemma}[\cite{CLT}, Lemma 2.8]\label{lem:closeness in S and GS}
For all $\gamma_1, \gamma_2\in \GS$, 
\[ d_\SSS(\gamma_1(0), \gamma_2(0)) \leq 2d_{\GS}(\gamma_1,\gamma_2).\]
Furthermore, for $s, t \in \mathbb R$, $d_\SSS(\gamma_1(s),\gamma_2(t)) \leq 2 d_{\GS}(g_s\gamma_1,g_t \gamma_2).$
\end{lemma}

Conversely, if two geodesics are close in $\SSS$ for a significant interval of time surrounding zero, then they are close in $\GS$:

\begin{lemma}[\cite{CLT}, Lemma 2.11]\label{lem:shadow in S}
Let $\eps$ be given and $a<b$ arbitrary. There exists $T=T(\eps)>0$ such that if $d_\SSS(\gamma_1(t),\gamma_2(t))<\eps/2$ for all $t\in[a-T,b+T]$, then $d_{\GS}(g_t\gamma_1,g_t\gamma_2)<\eps$ for all $t\in[a,b]$. For small $\eps$, we can take $T(\eps)= -  \log(\eps)$.
\end{lemma}

We will need a few refinements of Lemma \ref{lem:shadow in S} throughout the paper, which we record here.

\begin{lemma}[\cite{CCESW}, Lemma 2.13]\label{lem:exponentially close}
Suppose that $d_\SSS(\gamma_1(t),\gamma_2(t))=0$ for all $t\in[a,b]$. Then, for all $t\in[a,b]$, $d_{\GS}(g_t\gamma_1, g_t\gamma_2)\leq e^{-2\min\{ |t-a|,|t-b| \}}$.
\end{lemma}

For $t>0$, the time-$t$ geodesic flow is $e^{2t}$-Lipschitz (see, e.g., \cite[Lemma 2.14]{CCESW}). The next two lemmas give similar results in specific geometric situations.

\begin{lemma}\label{lemma: expansion u}
Suppose that $\gamma(t) = \eta(t)$ for all $t \leq r$ for some $r\in \mathbb{R}$. Then there exists $C := C(r) > 0$ such that $d_{GS}(g_t\gamma,g_t\eta) \leq Ce^{2t}d_{GS}(\gamma,\eta)$ for $t \leq 0$. If $r\geq 0$, then $C=1$.
\end{lemma}

\begin{proof}
By definition, for appropriately chosen lifts $\tilde \gamma$ and $\tilde \eta$,
\begin{align*}
d_{GS}(g_t\gamma,g_t\eta) &= \int_{-\infty}^{\infty}e^{-2|s|}d_{\tilde{S}}(\tilde{\gamma}(s + t), \tilde{\eta}(s + t))\,ds.
\\
&= \int_{-\infty}^{\infty}e^{-2|s-t|}d_{\tilde{S}}(\tilde{\gamma}(s), \tilde{\eta}(s))\,ds
\end{align*}
By our assumption on $\gamma$ and $\eta$, 
\begin{align*}
\phantom{d_{GS}(g_t\gamma,g_t\eta)} \qquad \ \  &= \int_{r}^{\infty} e^{-2|s-t|}d_{\tilde{S}}(\tilde{\gamma}(s),\tilde{\eta}(s))\,ds
\\
&\leq e^{\max\{-4r,0\}}e^{2t} \int_{r}^{\infty} e^{-2|s|}d_{\tilde{S}}(\tilde{\gamma}(s),\tilde{\eta}(s))\,ds
\\
&= e^{\max\{-4r,0\}}e^{2t}d_{GS}(\gamma,\eta).\qedhere
\end{align*}
\end{proof}

\begin{lemma}\label{lemma: boundedness cs}
Suppose that $\gamma(t) = \eta(t+\delta)$ for all $t \geq r$ for some $r\in \mathbb{R}$. Then there exists $D := D(r,\delta) > 0$ such that $d_{GS}(g_t\gamma,g_t\eta) \leq De^{-2t}d_{GS}(\gamma,\eta)+2\delta$ for $t \geq 0$. 
\end{lemma}

\begin{proof}
Observe that there exist lifts $\tilde{\gamma},\tilde{\eta}$ of $\gamma$ and $\eta$ such that 

\begin{align*}
d_{GS}(g_t\tilde{\gamma},g_t\tilde{\eta}) &= \int_{-\infty}^{\infty}e^{-2|s|}d_{\tilde{S}}(\tilde{\gamma}(s + t), \tilde{\eta}(s + t))\,ds
\\
&= \int_{-\infty}^{\infty}e^{-2|s-t|}d_{\tilde{S}}(\tilde{\gamma}(s), \tilde{\eta}(s))\,ds
\\
&\leq \int_{-\infty}^{\infty}e^{-2|s-t|}d_{\tilde{S}}(\tilde{\gamma}(s), \tilde{\eta}(s+\delta))\,ds+\int_{-\infty}^{\infty}e^{-2|s-t|}d_{\tilde{S}}(\tilde{\eta}(s+\delta), \tilde{\eta}(s))\,ds
\\
&=\int_{-\infty}^{\infty}e^{-2|s-t|}d_{\tilde{S}}(\tilde{\gamma}(s), \tilde{\eta}(s+\delta))\,ds+\int_{-\infty}^{\infty}e^{-2|s-t|}\delta\,ds
\\
&= \int_{-\infty}^{r} e^{-2|s-t|}d_{\tilde{S}}(\tilde{\gamma}(s),\tilde{\eta}(s+\delta))\,ds+\delta
\\
&\leq \int_{-\infty}^{r} e^{-2|s-t|}d_{\tilde{S}}(\tilde{\gamma}(s+\delta),\tilde{\eta}(s+\delta))\,ds+2\delta
\\
&=  \int_{-\infty}^{r+\delta} e^{-2|s-t-\delta|}d_{\tilde{S}}(\tilde{\gamma}(s),\tilde{\eta}(s))\,ds+2\delta
\\
&\leq e^{\max\{4r,0\}}e^{2|\delta|}e^{-2t}\int_{-\infty}^{r+\delta} e^{-2|s|}d_{\tilde{S}}(\tilde{\gamma}(s),\tilde{\eta}(s))\,ds+2\delta
\\
&= e^{\max\{4r,0\}}e^{2|\delta|}e^{-2t}d_{G\tilde{S}}(\tilde{\gamma},\tilde{\eta})+2\delta.\qedhere
\end{align*}
\end{proof}

Note that in the sense of Definition \ref{def: cone point asymptotic}, the geodesics in Lemma \ref{lemma: expansion u} are (a special case of) $-$-cone point asymptotic, and those in Lemma \ref{lemma: boundedness cs} are $+$-cone point asymptotic.

%
\subsection{Local stable and unstable sets}\label{subsec:stable unstable}

The local stable and unstable sets, analogues of the local stable and unstable manifolds from the smooth setting, play an essential role in our arguments. We define them here using Busemann functions; the definition given agrees with the usual one in the smooth setting.

\begin{definition}\label{defn: busemann}
For a geodesic $\tilde{\gamma}\in G\tilde{S}$, the \emph{Busemann function} determined by $\tilde{\gamma}$ is
$$B_{\tilde{\gamma}}(\cdot) := B_{\tilde{\gamma}(0)}(\cdot, \tilde{\gamma}(+\infty)),$$
where $B_{\tilde{\gamma}(0)}(\cdot, \tilde{\gamma}(+\infty)) : \tilde{S}\to \mathbb{R}$ is the function
$$x\mapsto \lim_{t\to\infty}(d_{\tilde{S}}(x,\tilde{\gamma}(t)) - t).$$
\end{definition}

Using this, we can define natural analogues of the local stable and unstable manifolds.

\begin{definition}\label{def: stable unstable}
We define the \emph{local strong unstable set} at $\gamma\in GS$ by 
$$W^{u}(\gamma,\delta) = \pi\{\tilde{\eta}\in G\tilde{S} \mid \tilde{\eta}(-\infty) = \tilde{\gamma}(-\infty), B_{-\tilde{\gamma}}(\tilde{\eta}(0)) = 0, \text{ and }d_{G\tilde{S}}(\tilde{\eta},\tilde{\gamma}) < \delta\},$$
where $\pi : G\tilde{S}\to GS$ is projection.

The \emph{strong unstable set} at $\gamma \in G\SSS$ is
$$W^{u}(\gamma) = \pi\{\tilde{\eta}\in G\tilde{S} \mid \tilde{\eta}(-\infty) = \tilde{\gamma}(-\infty), B_{-\tilde{\gamma}}(\tilde{\eta}(0)) = 0\}.$$
\end{definition}

The (local) strong stable set is defined similarly, replacing $-\infty$ with $+\infty$ and $-\tilde \gamma$ with $\tilde \gamma$. We will also need the local center stable set, which is defined to allow some ``mismatch of times''.

\begin{definition}\label{defn: center stable}
The \emph{local center stable set} at $\gamma$ is defined by
$$W^{cs}(\gamma,\delta) = \pi\{\tilde{\eta}\in G\tilde{S} \mid \tilde{\eta}(+\infty) = \tilde{\gamma}(+\infty) \text{ and }d_{G\tilde{S}}(\tilde{\eta},\tilde{\gamma}) < \delta\}.$$
\end{definition}

The stable and unstable sets are invariant under the geodesic flow. We prove this here for the 
local strong unstable, as we will need it later. Similar proofs for the stable and center stable can easily be obtained in the same way. We begin with some comparison geometry lemmas.

\begin{lemma}\label{lemma: distance as t tends to infinity}
Let $\tilde\gamma,\tilde\eta\in G\tilde {S}$ with $\tilde\eta(+\infty)=\tilde\gamma(+\infty)=\xi$. Fix $\tau>0$. Then,
$$\left|d_{\tilde S}(\tilde\eta(0),\tilde\gamma(t))-(d_{\tilde S}(\tilde\eta(0), \tilde\eta(\tau))+d_{\tilde S}(\tilde\eta(\tau),\tilde\gamma(t)))\right|\rightarrow 0 \quad\text{as}\quad t\rightarrow\infty.$$
\end{lemma}

\begin{proof}
Since $\tilde\eta(+\infty)=\tilde\gamma(+\infty)$, there exists $C>0$ such that $d_{\tilde S}(\tilde\gamma(t),\tilde\eta(t))\leq C$ for all $t\geq 0$. Consider a geodesic triangle $\Delta$ with vertices $\tilde\eta(0), \tilde\eta(t),\tilde\gamma(t)$ for $t\gg \tau$. Pick a point $p$ on the side $[\tilde\eta(0),\tilde\gamma(t)]$ such that $$\frac{d_{\tilde S}(\tilde\eta(0),p)}{d_{\tilde S}(\tilde\eta(0),\tilde\gamma(t))}=\frac{\tau}{t}.$$ 
By the properties of similar triangles for the comparison Euclidean triangle for $\Delta$ and the fact that $\tilde S$ is CAT(0), 
\begin{equation}\label{eq:distance to a segment}
d_{\tilde S}(\tilde\eta(\tau),p)\leq C\frac{\tau}{t}.
\end{equation}
Then, using the triangle inequality for various triangles,
$$\left[d_{\tilde S}(\tilde\eta(0), \tilde\eta(\tau))-d_{\tilde S}(\tilde\eta(\tau),p)\right]+[d_{\tilde S}(\tilde\eta(\tau),\tilde\gamma(t))-d_{\tilde S}(\tilde\eta(\tau),p)]\leq d_{\tilde S}(\tilde\eta(0),\tilde\gamma(t))\leq d_{\tilde S}(\tilde\eta(0), \tilde\eta(\tau))+d_{\tilde S}(\tilde\eta(\tau),\tilde\gamma(t))$$
Combining the above inequality with \eqref{eq:distance to a segment} gives the statement.
\end{proof}

\begin{lemma}
Let $\tilde\gamma,\tilde\eta\in G\tilde {S}$ with $\tilde\eta(+\infty)=\tilde\gamma(+\infty)$. If $B_{\tilde\gamma}(\tilde\eta(0))=0$ then $B_{g_\tau\tilde\gamma}(g_\tau\tilde\eta(0))=0$
for all $\tau$.
\end{lemma}

\begin{proof}
Using Definition \ref{defn: busemann},
\begin{align*}
0&=B_{\tilde\gamma}(\tilde\eta(0))=\lim\limits_{t\rightarrow\infty}(d_{\tilde S}(\tilde\eta(0),\tilde\gamma(t))-t)\\&=\lim\limits_{t\rightarrow\infty}(d_{\tilde S}(\tilde\eta(0), \tilde\eta(\tau))+d_{\tilde S}(\tilde\eta(\tau),\tilde\gamma(t))-t)\qquad\text{ by Lemma~\ref{lemma: distance as t tends to infinity}}\\&=\lim\limits_{t\rightarrow\infty}(\tau+d_{\tilde S}(g_\tau\tilde\eta(0),g_{\tau}\tilde\gamma(t-\tau))-t)\\
&=\lim\limits_{u\rightarrow\infty}(d_{\tilde S}(g_\tau\tilde\eta(0),g_\tau\tilde\gamma(u))-u)=B_{g_\tau\tilde\gamma}(g_\tau\tilde\eta(0)).\qedhere
\end{align*}
\end{proof}

\begin{corollary}\label{corollary: busemann under flow}
    Let $\tilde\gamma,\tilde\eta\in G\tilde {S}$ with $\tilde\eta(-\infty)=\tilde\gamma(-\infty)$. If $B_{-\tilde\gamma}(\tilde\eta(0))=0$ then $B_{g_\tau(-\tilde\gamma)}(g_\tau\tilde\eta(0))=0$
for all $\tau$.
\end{corollary}

\begin{corollary}\label{cor: flow preserves unstable}
Let $\gamma \in \GS$. Then
    $$g_\tau(W^u(\gamma,\delta))\subset W^u(g_\tau\gamma, e^{2\tau}\delta)\qquad \text{for}\quad \tau>0$$
    and
    $$g_\tau(W^u(\gamma,\delta))\subset W^u(g_\tau\gamma, \delta) \qquad \text{for}\quad \tau\leq 0.$$
\end{corollary}

\begin{proof}
If $\tilde\eta(-\infty)=\tilde\gamma(-\infty)$ then $g_\tau\tilde\eta(-\infty)=g_\tau\tilde\gamma(-\infty)$. Moreover, by Corollary~\ref{corollary: busemann under flow}, if $B_{-\tilde\gamma}(\tilde\eta(0))=0$ then $B_{g_\tau(-\tilde\gamma)}(g_\tau\tilde\eta(0))=0$. To complete the proof, we just need to bound the distance from $g_\tau \tilde \gamma$.

Let $\gamma\in \GS$, and let $\eta\in W^u(\gamma,\delta)$. Since $\tilde{\gamma}(-\infty) = \tilde{\eta}(-\infty)$, we know that $d_{G\tilde{S}}(\tilde{\gamma}(t),\tilde{\eta}(t))$ remains bounded as $t\to -\infty$. By convexity of the distance function, it is decreasing as $t\to -\infty$ and bounded above by $\delta$ for $t\leq 0$. Therefore, for $\tau \leq 0$, we have

\begin{align*}
d_{G\tilde{S}}(g_{\tau}\tilde{\eta},g_{\tau}\tilde{\gamma}) &= \int_{-\infty}^{\infty}e^{-2|t|}d_{\tilde{S}}(\tilde{\eta}(t + \tau),\tilde{\gamma}(t + \tau))\,dt
\\
&\leq \int_{-\infty}^{\infty}e^{-2|t|}d_{\tilde{S}}(\tilde{\eta}(t),\tilde{\gamma}(t))\,dt
\\
&= d_{G\tilde{S}}(\tilde{\eta},\tilde{\gamma}).
\end{align*}
Hence, 
$$g_{\tau}(W^u(\gamma,\delta))\subset W^u(g_{\tau}\gamma,\delta) \quad \text{for}\quad \tau\leq 0.$$

Finally, if $\tau \geq 0$, using that $g_\tau$ is $e^{2\tau}$-Lipschitz,
$$d_{GS}(g_{\tau}\eta,g_{\tau}\gamma) \leq e^{2\tau}d_{GS}(\eta,\gamma) < e^{2\tau}\delta.$$

Thus, we obtain the desired statement.
\end{proof}

The following lemmas and corollaries make the relationship between our systems and nonuniformly hyperbolic systems evident.

\begin{lemma}\label{lemma: geometry of u and cs}
Assume $\gamma\in\operatorname{Reg}(c)$. Then there exists $\delta$ such that if $\tilde\eta\in W^u(\tilde{\gamma},\delta)$, then $\tilde \gamma$ and $\tilde \eta$ are $-$-cone point asymptotic, and if $\tilde{\eta}\in W^{cs}(\tilde{\gamma},\delta)$, then $\tilde \gamma$ and $\tilde \eta$ are  $+$-cone point asymptotic.
\end{lemma}

\begin{proof}
Since $\gamma\in\operatorname{Reg}(c)$, by \cite[Proposition 3.9]{CCESW}, $\tilde{\gamma}$ turns with angle greater than $\pi$ at some $\tilde{p}\in \tilde{S}$. There exists $\delta > 0$ small enough so that $\tilde{\eta}$ passes through the cone point $\tilde{p}$ \cite[Figure 2]{CCESW}. Now, \cite[Proposition 8.2]{bh} tells us that there are unique geodesic rays from $\tilde{p}$ that are asymptotic to $\tilde{\gamma}(+\infty)$ and $\tilde{\gamma}(-\infty)$. This completes our proof.
\end{proof}

Lemmas~\ref{lemma: geometry of u and cs} and~\ref{lemma: expansion u} combine to give us the following corollary.

\begin{corollary}\label{cor:u expansion}
For all $c > 0$, there exist $\delta, C > 0$ such that for all $\gamma\in \operatorname{Reg}(c)$, we have the following property. For all $\eta\in W^u(\gamma,\delta)$, for all $t\leq 0$,
$$d_{GS}(g_t\gamma,g_t\eta) \leq Ce^{2t}d_{GS}(\gamma,\eta).$$
\end{corollary}

Similarly, Lemmas~\ref{lemma: geometry of u and cs} and ~\ref{lemma: boundedness cs} combine to give us an analogous corollary on center stable sets.

\begin{corollary}\label{cor: cs bound}
For all $c > 0$, there exists $\delta_0, D > 0$ such that for all $\gamma\in \operatorname{Reg}(c)$, we have the following property. For all $\eta\in W^{cs}(\gamma,\delta)$ with $0<\delta<\delta_0$, for all $t\geq 0$,
$$d_{GS}(g_t\gamma,g_t\eta) \leq D\delta.$$
\end{corollary}

We noted above that the function $\lambda:\GS \to [0,\infty)$ is lower semi-continuous. On local strong unstable sets we in fact have full continuity.

\begin{proposition}\label{prop:continuity of lambda on unstable}
Let $\lambda\colon GS\rightarrow [0,+\infty)$ be as in \cite[Definition 3.3]{CCESW} with parameter $s_0$ and $\gamma\in GS$. Then for all $\eps>0$ there exists $\delta>0$ such that for all $\xi\in W^u(\gamma,\delta)$, we have $|\lambda(\gamma)-\lambda(\xi)|<\eps.$
\end{proposition}

\begin{proof}
Assume that $\gamma\in\operatorname{Reg}(c)$. Then, we have two possibilities:
\begin{enumerate}
    \item There exists $u\in(-s_0,s_0)$ so that $\gamma$ turns with angle larger than $\pi$ at the cone point $\gamma(u)$. 
    
    In this case, $\lambda$ is continuous at $\gamma$ by Case 1 in \cite[Proof of Lemma 3.8]{CCESW} so we have the statement if we only consider $\xi\in W^u(\gamma,\delta)$.
    
    \item There exist $u_1\in\left[-\frac{\theta_0}{2c},-s_0\right]$ and $u_2\in\left[s_0,\frac{\theta_0}{2c}\right]$ such that $\gamma$ has turning angles at least $cs_0$ away from $\pm\pi$ at $\gamma(u_1)$ and $\gamma(u_2)$. 
    
    Following Case 2 in \cite[Proof of Lemma 3.8]{CCESW}, for sufficiently small $\delta$ (depending on $\eps$ and $S$), $\gamma$ and $\xi$ share the segment $\gamma[u_1,u_2]$. Moreover, since $\xi\in W^u(\gamma,\delta)$, they actually share $\gamma[-\infty,u_2]$ and $\gamma(0)=\xi(0)$. Thus, for such a choice of $\delta$, $\gamma$ and $\xi$ meet cone points with turning angle larger than $\pi$ at the same time and turn there with close angles as can be seen in Case 2 of \cite[Proof of Lemma 3.8]{CCESW}. Thus, $|\lambda(\gamma)-\lambda(\xi)|<\eps$.
\end{enumerate}

Assume that $\lambda(\gamma)=0$. Then, by \cite[Proposition 3.4]{CCESW}, either $\lambda(g_t\gamma)=0$ for all $t\leq 0$ or for all $t\geq 0$. Therefore, we have the following possibilities:
\begin{enumerate}
    \item $\lambda(g_t\gamma)=0$ for all $t\leq 0$ (so $\gamma$ does not turn with angle larger than $\pi$ on $\gamma[-\infty,0]$), and there exists $u>0$ such that $\gamma$ turns with angle larger than $\pi$ at $\gamma(u)$ and has turning angles $\pm\pi$ on $\gamma[0,u)$. 
    
    In this case, for sufficiently small $\delta$ (depending on $\gamma$ and $\eps$), $\xi\in W^u(\gamma,\delta)$ has to pass through $\gamma(u)$ and turn with angle close to the turning angle of $\gamma$ at $\gamma(u)$ by \cite[Lemmas 2.11 and 2.14]{CCESW}. Moreover, since $\xi(-\infty)=\gamma(-\infty)$ and $B_{-\tilde\gamma}(\tilde\xi(0))=0$, we have $\xi(t)=\gamma(t)$ for $t\leq u$. Thus, $\gamma$ and $\xi$ see the same first cone point at the same time and turn with close angles so $\lambda(\xi)$ and $\lambda(\gamma)$ are close.

    \item $\lambda(g_t\gamma)=0$ for all $t\geq 0$ (so $\gamma$ does not turn with angle larger than $\pi$ on $\gamma[0,\infty]$), and there exists $u<0$ such that $\gamma$ turns with angle larger than $\pi$ at $\gamma(u)$ and has turning angles $\pm\pi$ on $\gamma(u,0]$. 
    
    By choosing $\delta$ small enough (depending on $\gamma$ and $\eps$), we can guarantee that for $\xi\in W^u(\gamma,\delta)$ we have that $\xi(t)=\gamma(t)$ for $t\leq u$ and $\xi$ does not turn with angle larger than $\pi$ on $\xi\left(u,2\max\left\{\frac{\theta_0}{\lambda^{ss}(\gamma)},\frac{\theta_0}{\eps}\right\}\right]$. Then, $|\lambda(\xi)-\lambda(\gamma)|<\eps$.
    
    \item $\lambda(g_t\gamma)=0$ for all $t$ (so $\gamma$ does not turn with angle larger than $\pi$ at any point). 
    
    If $\tilde\gamma$ is in the interior of a flat strip of $\tilde S$, then for sufficiently small $\delta$, $\tilde\xi$ is also in the interior of a flat strip of $\tilde S$ for $\xi\in W^u(\gamma,\delta)$. Thus, $\lambda(\xi)=0$. 
    
    If $\tilde\gamma$ is not in the interior of a flat strip, then $\tilde\gamma$ passes through cone points but turns with angle $\pm\pi$ at those points. If $\lambda(\xi)=0$ there is nothing to show. If there exists $u> 0$ such that $\gamma(u)$ is a cone point, then for sufficiently small $\delta$, $\xi(t)=\gamma(t)$ for $t\leq u$ so $\lambda(\xi)=0$. Assume that $\lambda(\xi)>0$ and $\gamma(0,+\infty)$ does not contain cone points. Let $u\leq 0$ be the closest number to $0$ such that $\gamma(u)$ is a cone point. Then, by taking sufficiently small $\delta$, we can ensure that $\xi(t)=\gamma(t)$ for $t\leq u$ and $\xi$ turns with a small angle at $\xi(u)$. Specifically, we can ensure that $\frac{\theta(\xi,u)}{\max\{s_0,|u|\}}$ is less than $\eps$. Then, $|\lambda(\xi)-\lambda(\gamma)|<\eps$. \qedhere
\end{enumerate}
\end{proof}

\begin{corollary}
Let $\lambda\colon GS\rightarrow [0,+\infty)$ be as in \cite[Definition 3.3]{CCESW} and $\gamma\in GS$. Then, for any $a>0$ and $t\in\mathbb{R}$, $\lambda\circ g_t|_{W^u(\gamma,a)}$ is continuous.
\end{corollary}

\begin{proof}
Let $a > 0$ and $t\in\mathbb{R}$. Then $g_tW^u(\gamma, a) \subset W^u(g_t\gamma, e^{2t}a)$ by Corollary \ref{cor: flow preserves unstable}. Consequently, $\lambda|_{g_tW^u(\gamma,a)}$ is continuous by Proposition \ref{prop:continuity of lambda on unstable}.
\end{proof}

%
\subsection{Bowen brackets and their behavior}\label{subsec:Bowen brackets}

\begin{definition}\label{def: Bowen bracket}
Let $\delta>0$.  Consider $\gamma,\eta\in \Reg$ such that $d_{GS}(\gamma,\eta) \leq \delta$. We define the Bowen bracket as
	\[ [ \gamma, \eta] := W^{cs}(\gamma,\delta) \cap W^{u}(\eta,\delta).\]
\end{definition}

\begin{definition}\label{def: su-Bowen bracket}
Let $\delta > 0$. Consider $\gamma,\eta\in \Reg$ such that $d_{GS}(\gamma,\eta) \leq \delta$. We define the $su$-Bowen bracket as
	\[ \langle \gamma, \eta\rangle := W^{s}(\gamma,\delta) \cap W^{u}(\eta,\delta).\]
\end{definition}

\begin{remark*}
Note that if either of these brackets contains a regular geodesic, then the bracket consists of at most one point, as otherwise, we would have two geodesics $\xi_1,\xi_2$ with $\xi_1(\pm\infty) = \xi_2(\pm\infty)$. This is only possible if $\xi_1 = g_t\xi_2$ for some $t$; however, this would contradict both $\xi_1$ and $\xi_2$ belonging to $W^u(\eta,\delta)$.
\end{remark*}

A priori, $[\gamma, \eta]$ and $\langle \gamma, \eta\rangle$ might be empty. But using our $\lambda$, we can precisely describe situations where these brackets exist and consist of a single, easily described geodesic.

\begin{lemma}\label{lemma: bracket_su geometry}
Let $c > 0$. Suppose $x$ is chosen with $\lambda(g_{-s_0}x),\lambda(g_{s_0}x) > c$. Then there exists $\delta := \delta(c) > 0$ such that if $\gamma\in W^{u}(x,\delta),\eta\in W^{s}(x,\delta)$, then $\langle\gamma,\eta\rangle = \xi$, where $\xi(r) = \gamma(r)$ for $r\geq 0$ and $\xi(r) = \eta(r)$ for $r\leq 0$.
\end{lemma}

\begin{proof}
First, if $\lambda(g_{s_0}x),\lambda(g_{-s_0}x) > c$, then the first cone points $x$ turns at are at times $-2s_0-\frac{\theta_0}{2c}<t_1<0$ and $0<t_2<2s_0+\frac{\theta_0}{2c}$ and with angle at least $s_0c$ by Lemma~\ref{lemma: properties of lambda}. Then choose $\delta$ small enough so that given $y$ with $d_{GS}(y,x) \leq \delta$, then $y$ turns at $x(t_1)$ and $x(t_2)$ with angle at least $\frac{s_0c}{2}$. In particular, if $\gamma\in W^u(x,\delta)$ and $\eta\in W^s(x,\delta)$, this implies that $\gamma(r)=\eta(r)=x(r)$ for $r\in[t_1,t_2]$, which in turn guarantees that $\xi$ as constructed is a geodesic. Furthermore, we have $d_{GS}(\xi,\gamma) = d_{GS}(\eta,x)$, and similarly $d_{GS}(\xi,\eta) = d_{GS}(\gamma,x)$. Finally, by our choice of $\delta$, $\langle\gamma,\eta\rangle$ consists of a unique geodesic, because $\xi$ is bounded away from a flat strip.
\end{proof}

\begin{corollary}\label{cor: bracket distance}
Assume we are in the setting of Lemma~\ref{lemma: bracket_su geometry}. Then, $d_{GS}(\xi,x)=d_{GS}(x,\gamma)+d_{GS}(x,\eta)$.
\end{corollary}

\begin{proof}
    Straightforward from the definition of $d_{GS}$ and $\xi$.
\end{proof}

\begin{corollary}\label{cor: lambda bounded below}
Assume we are in the setting of Lemma \ref{lemma: bracket_su geometry}. Given $\xi = \langle \gamma,\eta\rangle$, then $\lambda(g_{r}\xi) > \frac{cs_0}{2\left(3s_0 + \frac{\theta_0}{2c}\right)}$ for $-s_0 \leq r \leq s_0$
\end{corollary}

\begin{proof}
Recall that by our choice of $\delta$, $\xi$ first turns at a cone point at times $-\left(2s_0 + \frac{\theta_0}{2c}\right) < t_1 < 0 < t_2 < 2s_0 + \frac{\theta_0}{2c}$, with turning angle at least $\frac{cs_0}{2}$. Now, observe that since $\xi$ turns with angle at least $\frac{cs_0}{2}$ at both $t_1$ and $t_2$, we have
$$\lambda(\xi) \geq \frac{\frac{cs_0}{2}}{\max\{|t_1|,|t_2|,s_0\}} \geq \frac{cs_0}{2\left(2s_0 + \frac{\theta_0}{2c}\right)}.$$
If we let $r\in [-s_0,s_0]$ and consider $g_r\xi$, there are three possibilities. We could have $t_1 - r < 0 < t_2 - r$, in which case $g_r\xi$ first turns at a cone point at times $t_1 - r < 0 < t_2 - r$, and it turns with angle at least
$\frac{cs_0}{2}$, or $t_1 - r \geq 0$ or $t_2 - r \leq 0$. In the first case, then
$$\lambda(g_r\xi) \geq \frac{cs_0}{2\max\{|t_1 - r|,|t_2 - r|,s_0\}} \geq \frac{cs_0}{2\left(2s_0 + \frac{\theta_0}{2c} + r\right)}.$$
In the second case, $t_1 - r \in [0,s_0)$, and so $g_r\xi$ turns at a cone point with angle at least $\frac{cs_0}{2}$ at some time less than $s_0$ away from $0$. Thus, $\lambda(g_r\xi) > \frac{c}{2}$. The analysis for the third case is similar.
\end{proof}

\subsection{Rectangles and flow boxes}\label{sec:rectangles}

We now exploit the local \emph{topological} product structure of $\GS$ to build rectangles and flow boxes. A rectangle is a transversal to the flow direction with `rectangular' structure given by stable and unstable directions, and a flow box is a rectangle thickened in the flow direction by a small amount. We now construct these sets at small scale around regular geodesics and prove a few properties of their geometry.

\begin{definition}\label{definition: su-rectangle}
Let $x$ and $\delta$ be as in Lemma~\ref{lemma: bracket_su geometry}. We then define an $su$-rectangle centered at $x$ with $\diam(R_x(\eps))<4\eps<\delta$ in the following way:
\begin{equation}\label{def rectangle}
R_x(\eps)=\{\langle\gamma,\eta\rangle\,|\,\gamma\in W^{u}(x,\eps),\eta\in W^{s}(x,\eps)\}.
\end{equation}

Notice that for $\eps<s_0$ (in particular, smaller than the injectivity radius of $S$) we have $R_x(\eps)$ is a transversal to the flow. That is, for any $\gamma \in GS$ within $\epsilon$ of $R_x(\eps)$, $g_{[-5\epsilon,5\epsilon]}\gamma \cap R_x(\eps)$ is either empty or a single point.    
\end{definition}

\begin{lemma}
For $\eps < \delta/4$, the su-Bowen bracket is defined and equal to a single geodesic on any pair of geodesics in $R_x(\eps)$.
\end{lemma}

\begin{proof}
Let $y,z\in R_x(\eps)$. Then, $y=\langle\gamma_y,\eta_y\rangle$ and $z=\langle\gamma_z,\eta_z\rangle$ where $\gamma_y,\gamma_z\in W^u(x,\eps)$ and $\eta_y,\eta_z\in W^s(x,\eps)$. By Corollary~\ref{cor: bracket distance}, $d_{GS}(x,y),d_{GS}(x,z)<2\eps<\delta$ and so $y(r)=z(r)=x(r)$ for $r\in[t_1,t_2]$ for $t_1,t_2$ as in Lemma~\ref{lemma: bracket_su geometry}. Thus, $\langle y,z\rangle=\langle\gamma_y,\eta_z\rangle\in R_x(\eps)$.  
\end{proof}

\begin{lemma}\label{lem: bracket and flow interaction}
Given a rectangle $R_x(\eps)$, there exists $t_R := t_R(c)$ such that for all $\langle\gamma,\eta\rangle\in R_x(\eps)$ and $t\geq t_R$,
$$\lambda(g_t\langle \gamma,\eta\rangle) = \lambda(g_t\gamma) \text{ and } \lambda(g_{-t}\langle \gamma,\eta\rangle) = \lambda(g_{-t}\eta).$$
\end{lemma}

\begin{proof}
Choose $t_R > 2s_0 + \frac{\theta_0}{2\min\{\lambda(g_{s_0}x),\lambda(g_{-{s_0}}x\}}$, where, again, $x$ is the point at which $R_x(\eps)$ is centered. The proof of Lemma \ref{lemma: bracket_su geometry} gives us that for any $\langle \gamma,\eta\rangle\in R_x(\eps)$, then $\gamma$ and $\langle\gamma,\eta\rangle$ turn at the same first cone point with the same angle, and that that cone point occurs in the interval $(0,t_R)$. Furthermore, by Lemma \ref{lemma: bracket_su geometry}, for any $\langle \gamma,\eta\rangle\in R_x(\eps)$, $g_{t_R}\langle \gamma,\eta\rangle(t) = g_{t_R}\gamma(t)$ for $t\geq -t_R$. Therefore, $g_{t_R}\langle\gamma,\eta\rangle$ and $g_{t_R}\gamma$ agree at least up to the first cone point in negative time and for all cone points in positive time, and so $\lambda(g_{t + t_R}\langle \gamma,\eta\rangle) = \lambda(g_{t + t_R}\gamma)$ for all $t\geq 0$. A similar proof holds for $\eta$ by reversing the time direction.
\end{proof}

\begin{definition}\label{definition: flow box}
    Let $R_x(\eps)$ be as in Definition~\ref{definition: su-rectangle}.
    For $n\in\mathbb N$ such that $n>\frac{4}{\delta}$, we call $\mathcal{B}(n,\eps)=g_{\left[-\frac{1}{n},\frac{1}{n}\right]}(R_x(\eps))$ the $(n,\eps)$-flow box centered at $x$.
\end{definition}

\begin{lemma}\label{lem: bowen bracket in box}
There exists $n_0=\max\left\{\frac{8}{\delta},5\right\}$, such that for all $n\geq n_0$, the Bowen bracket is defined on any pair of geodesics in $\mathcal{B} (n,\eps)$ as in Definition~\ref{definition: flow box}.
\end{lemma}

\begin{proof}
Let $y,z\in \mathcal{B}(n,\eps)$ then $y=g_{t_y}\langle\gamma_y,\eta_y\rangle$ and $z=g_{t_z}\langle\gamma_z,\eta_z\rangle$ for $\gamma_y, \gamma_z\in W^u(x,\eps)$, $\eta_y,\eta_z\in W^s(x,\eps)$, and $t_y,t_z\in\left[-\frac{1}{n},\frac{1}{n}\right]$. We claim that $[y,z]=g_{t_z}\langle\gamma_y,\eta_z\rangle.$

We have
\begin{align*}
d_{\GS}(g_{t_z}\langle\gamma_y,\eta_z\rangle,y)&=d_{\GS}(g_{t_z}\langle\gamma_y,\eta_z\rangle,g_{t_y}\langle\gamma_y,\eta_y\rangle)\\&\leq\int_{-\infty}^{+\infty}d_{\tilde S}\left(\langle\gamma_y,\eta_z\rangle(t),\langle\gamma_y,\eta_y\rangle(t)\right)e^{-2|t-t_z|}\,dt+|t_z-t_y|\\&=\int_{-\infty}^0d_{\tilde S}\left(\eta_z(t),\eta_y(t)\right)e^{-2|t-t_z|}\,dt+|t_z-t_y|\\&\leq e^{\frac{2}{n}}\left(d_{\GS}(\eta_z,x)+d_{\GS}(\eta_y,x)\right)+\frac{2}{n}\\&\leq e^{\frac{2}{n}}2\eps+\frac{2}{n}<\delta.
\end{align*}

Similarly,
\begin{align*}
 d_{\GS}\left(g_{t_z}\langle\gamma_y,\eta_z\rangle,z\right)&=\int_0^{+\infty}d_{\tilde S}\left(\gamma_y(t),\gamma_z(t)\right)e^{-2|t-t_z|}\,dt\\&\leq e^{2|t_z|}\left(d_{\GS}(\gamma_y,x)+d_{\GS}(\gamma_z,x)\right)=e^{\frac{2}{n}}2\eps<\delta
\end{align*}

Thus, by the construction, we have that $g_{t_z}\langle\gamma_y,\eta_z\rangle\in W^{cs}(y,\delta)\cap W^{u}(z,\delta)$. Moreover, it is regular, so  $[y,z]=g_{t_z}\langle\gamma_y,\eta_z\rangle$.
\end{proof}

We record a Lemma here which we will need later.

\begin{lemma}\label{lem:zero msr bdry}
For almost every small $\eps$, $\mu(\mathcal{B}(n,\eps)\setminus \operatorname{Int}\mathcal{B}(n,\eps))=0$ for all $n$ large enough (depending on the injectivity radius).
\end{lemma}

\begin{proof}
The boundary of $\mathcal{B}(n, \eps)$ is the union of three pieces. The first two are $g_{\pm\frac{1}{n}}R_x(\eps)$. Since $\mu$ is $g_t$-invariant and for all small $t$, $g_{\pm\frac{1}{n}+t}R(\eps)$ are disjoint, these boundary portions must have zero measure, or else $\mu$ will have infinite total mass.

The third piece is $g_{\left[-\frac{1}{n},\frac{1}{n}\right]}\partial R(\eps)$. If for a positive measure set of small $\eps$ (indeed, even if for an uncountable set of small $\eps$ these have positive measure) again $\mu$ will have infinite total mass.
\end{proof}

%
\section{Dynamical results}\label{sec:dynamics}

\begin{definition}\label{def: Bowen balls}
Given $\gamma \in GS$, 
\begin{itemize}
    \item The \emph{stable Bowen ball of length $n$ and radius $\eps$} is 
$$B_{n}^{s}(\gamma,\epsilon) := \{\eta \in W^{s}(\gamma,\eps) \mid d_{GS}(g_t\gamma,g_t\eta) \leq \epsilon \text{ for } -n\leq t \leq 0\}.$$
    \item The \emph{center stable Bowen ball} is
$$B_{n}^{cs}(\gamma,\epsilon) := \{\eta \in W^{cs}(\gamma,\eps) \mid d_{GS}(g_t\gamma,g_t\eta) \leq \epsilon \text{ for } -n\leq t \leq 0\}.$$
    \item The \emph{unstable Bowen ball} $B^u_{m}(\gamma,\epsilon)$ is defined similarly:
$$B_m^u(\gamma,\eps) := \{\eta \in W^{u}(\gamma,\eps) \mid d_{GS}(g_t\gamma,g_t\eta) \leq \eps \text{ for } 0\leq t \leq m\}$$
except the shadowing occurs over the interval $0\leq t \leq m$.
    \item The \emph{two-sided Bowen ball} is
$$B_{[-n,m]}(\gamma,\eps) := \{\eta \in GS\mid d_{GS}(g_t\gamma,g_t\eta) \leq \eps \text{ for } -n\leq t \leq m\}.$$ We will also use the following notation $B_t(\gamma,\eps):=B_{[0,t]}(\gamma,\eps).$
\end{itemize} 
\end{definition}

%
\subsection{The global Bowen property}\label{subsec:global Bowen}

The Bowen property for a potential plays a key role in many of the techniques used for studying thermodynamic formalism of systems with some hyperbolicity.

\begin{definition}\label{defn:global bowen}
A potential $\phi : GS \to \mathbb{R}$ has the \emph{global Bowen property} if there is some $\varepsilon > 0$ for which there exists a constant $K > 0$ such that for all $x,y \in GS$ and $t > 0$
$$\sup\left\{\left|\int_0^t\phi(g_rx) - \phi(g_ry)\, dr\right| : d_{GS}(g_rx,g_ry) \leq \varepsilon \text{ for } 0 \leq r \leq t\right\} \leq K.$$
\end{definition}

In \cite{CCESW}, we proved a non-uniform version of the Bowen property for any H\"older potential $\phi$, which we state explicitly below (Proposition \ref{prop:non-uniform Bowen}). In this section we provide an argument which upgrades this non-uniform Bowen property to the global property for H\"older $\phi$ that are locally constant on a neighborhood of the singular set, using properties of $\lambda$. Using the $\lambda$-decomposition machinery we have developed already, the proof here is relatively short, and may be useful in other settings that use a $\lambda$-decomposition.

Before providing that proof, however, we note that it is possible to leverage the particularly nice geometry of $\SSS$ to prove the global Bowen property without invoking the $\lambda$-decomposition machinery. The heart of the argument lies in studying the geometry of Bowen balls and how they interact with cone points. The figures below provide two illustrative examples of the possible geometry.

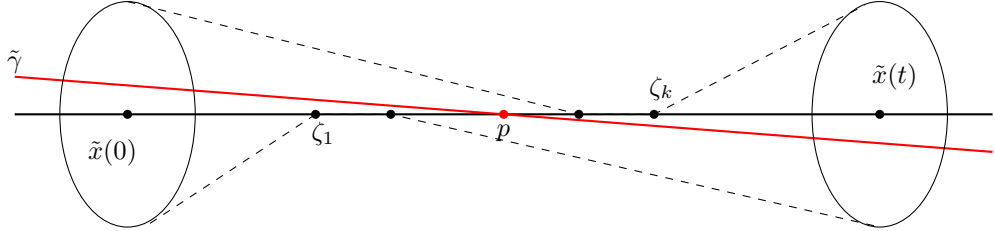
\begin{figure}[ht]
\centering
\begin{tikzpicture}[scale=1]

\draw[thick] (-6.5,0) -- (6.5,0);
\draw[thick, red] (-6.5,.5) -- (6.5,-.5);


\fill (-5,0) circle(.06);
\fill (5,0) circle(.06);

\node at (-5.2,-.5) {$\tilde x(0)$} ;
\node at (5.2,.5) {$\tilde x(t)$} ;

\draw (-5,0) ellipse (.9cm and 1.5cm);
\draw (5,0) ellipse (.9cm and 1.5cm);

\draw [dashed] (-4.7,-1.45) -- (-2.5,0) -- (-1.5,0) -- (5,-1.5);
\draw [dashed] (-5,1.5) -- (1,0) -- (2,0) -- (4.8,1.45);




\node at (-2.4,-.25) {$\zeta_1$} ;
\node at (2.1,.35) {$\zeta_k$} ;

\fill[red] (0,0) circle(.06);

\node at (0,-.25) {$p$};
\node at (-6.5, .7) {$\tilde \gamma$};

\fill (-2.5,0) circle(.06);
\fill (-1.5,0) circle(.06);

\fill (2,0) circle(.06);
\fill (1,0) circle(.06);

\end{tikzpicture}
\caption{The boundary of $H$ intersects $x$ in two disjoint intervals.}\label{fig:CHcase2}
\end{figure}

\begin{figure}[h]
\centering
\begin{tikzpicture}[scale=1]

\draw[thick] (-6.5,0) -- (6.5,0);


\fill (-5,0) circle(.06);
\fill (5,0) circle(.06);

\node at (-5.2,-.5) {$\tilde x(0)$} ;
\node at (5.2,.5) {$\tilde x(t)$} ;

\draw (-5,0) ellipse (.9cm and 1.5cm);
\draw (5,0) ellipse (.9cm and 1.5cm);

\draw [dashed] (-4.7,-1.45) -- (-2.5,0) -- (2,0) -- (4.8,- 1.5);
\draw [dashed] (-4.7,1.45) -- (-2.5,0) -- (2,0) -- (4.7,1.4);




\node at (-2.4,-.25) {$\zeta_1$} ;
\node at (2.1,.35) {$\zeta_k$} ;

\fill (-2.5,0) circle(.06);
\fill (-1.5,0) circle(.06);

\fill (2,0) circle(.06);
\fill (1,0) circle(.06);

\end{tikzpicture}
\caption{The boundary of $H$ intersects $x$ in $[\zeta_1,\zeta_k]$.}\label{fig:CHcase3}
\end{figure}
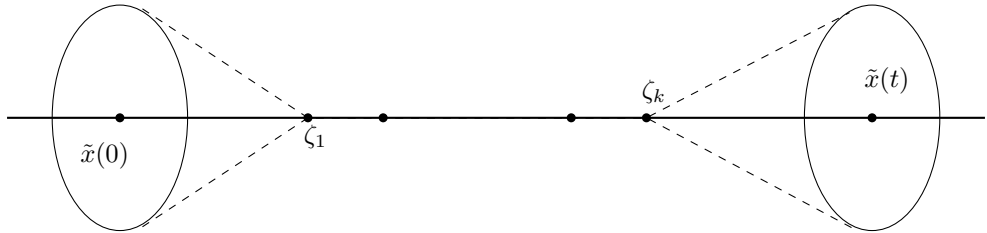

Working in the universal cover, let $H$ be the convex hull of the balls of radius $\eps$ around $\tilde x(0)$ and $\tilde x(t)$. Any geodesic $\tilde y$ in $B_t(\tilde x,\eps)$ must traverse $H$. If $[\tilde x(0), \tilde x(t)]$ contains no cone points, for $r\in [T,t-T]$ where $T$ is as in Lemma \ref{lem:shadow in S}, $g_r\tilde x$ is close to $\Sing$ and the fact that $\phi$ is constant near $\Sing$ can be used to prove the bound on $|\int_0^t \phi(g_r x)-\phi(g_r y) \ dr|.$ Otherwise, the cone points on $[\tilde x(0), \tilde x(t)]$ have a strong impact on the geometry of $H$ -- its boundary must pass through any cone points on $[\tilde x(T), \tilde x(t-T)].$ Figures \ref{fig:CHcase2} and \ref{fig:CHcase3} show two possible ways this can happen (cone points are marked by $\zeta_i$). In Figure \ref{fig:CHcase2}, one can construct a singular geodesic $\tilde \gamma$ through $H$ and very near to both $\tilde x$ and $\tilde y$ using the cone-point-free `line of sight' between the $\eps$-balls around $\tilde x(0)$ and $\tilde x(t)$. Using again that $\phi$ is locally constant near $\Sing$, the desired bound follows. In Figure \ref{fig:CHcase3}, no such line of sight exists. But as $\tilde y$ traverses $H$ between $B_\eps(\tilde x(0))$ and $\zeta_1$ and between $\zeta_k$ and $\tilde x(t)$ there are clear lines of sight and nearby singular geodesics. Between $\zeta_1$ and $\zeta_k$, $\tilde x$ and $\tilde y$ must overlap. Then Lemma \ref{lem:exponentially close} proves (after a small reparametrization) $g_r\tilde x$ and $g_r\tilde y$ are exponentially close, and the H\"older property of $\phi$ can be used to complete the result.

We now prove the global Bowen property using the $\lambda$-function.

\begin{lemma}\label{lem: general version of BCFT 3.4} (compare to \cite[Proposition 3.4]{BCFT})\label{lemma: nbhd of Sing}
If $\lambda$ is lower semicontinuous, then for all $\delta > 0$, there exist $c>0$ and $T > 0$ such that if $\lambda(g_tx) \leq c$ for $t\in [-T,T]$, then $d_{GS}(x,\Sing) \leq \delta$.
\end{lemma}

\begin{proof}
For all $c > 0$, let $A(c) = \left\{x\in GS\mid \lambda(g_tx) > c \text{ for some } t\in \left[\frac{-1}{c},\frac{1}{c}\right]\right\}$. Notice that by definition, $A(c_2)\subset A(c_1)$ if $0<c_1<c_2$.

We will show that $A(c)$ is open using lower semi-continuity of $\lambda$. Let $x\in A(c)$, and suppose that $\lambda(g_{t^*}x) > c$ for some $t^*\in [-c^{-1},c^{-1}]$. Using lower semi-continuity, let $\epsilon_1 > 0$ be chosen small enough so that $\lambda(y) > c$ whenever $y\in B(g_{t^*}x,\epsilon_1)$. By continuity of the flow, let $\epsilon_2 > 0$ be chosen such that if $y\in B(x,\epsilon_2)$, then $g_{t^*}y\in B(g_{t^*}x,\epsilon_1)$. Then $B(x,\epsilon_2)\subset A(c)$, and so $A(c)$ is open.

By \cite[Corollary 3.5]{CCESW}, $\Sing^c=\bigcup\limits_{c>0}A(c)$. Consider $K=\{x\in GS\, |\, d_{GS}(x,\Sing)\geq \delta \}$. $\{A(c)\}_{c>0}$ is an open cover for the compact set $K$. Thus, there exists a finite subcover $\{A(c_j)\}_{j=1,\ldots,k}$ for $K$. Let $c=\min\{c_1,\ldots,c_k\}$. Since $A(c_j)\subset A(c)$ for $j=1,\ldots,k$, we find that $K\subset A(c)$ and the lemma holds.
\end{proof}

Recall that $\CCC(c)$ is the set of orbit segments with $\lambda\geq c$ at their beginning and end (Definition \ref{defn: good start and end}).

\begin{definition}\label{def: Bowen property}
Given a potential $\phi : \GS\to\mathbb{R}$, we say that $\phi$ has the Bowen property on $\CCC(c)$ if there is some $\eps > 0$ for which there exists a constant $K > 0$ such that
\begin{equation*}
    \sup\left\{\left|\int_0^t \phi(g_rx) - \phi(g_ry)\,dr\right| : (x,t)\in \CCC(c) \text{ and } d_{\GS}(g_ry,g_rx) \leq \eps \text{ for } 0\leq r \leq t\right\} \leq K.
\end{equation*}
\end{definition}

\begin{proposition}\label{prop:non-uniform Bowen}
Any H\"older $\phi:G\SSS\to \mathbb{R}$ has the Bowen property on $\CCC(c)$ for all $c>0$.
\end{proposition}

\begin{proof}
This is proved in \S6 of \cite{CCESW}. In that paper, we prove the Bowen property for a collection of orbit segments $\mathcal{G}(c)$ which is smaller than $\CCC(c)$; segments in $\mathcal{G}(c)$ have average value of $\lambda$ greater than $c$ over any initial and terminal segment, not just at times 0 and $t$. However, the proof only uses the values of $\lambda$ at $x$ and $g_t x$, and hence applies to $\CCC(c).$
\end{proof}

We now update this non-uniform result to the global Bowen property under the additional assumption that $\phi$ is locally constant on a neighborhood of $\Sing$.

\begin{proposition}\label{proposition: global Bowen property}
If $\lambda$ is lower semicontinuous, $\phi$ is locally constant on $B(\Sing,\delta)$ for some $\delta>0$ and has the Bowen property on $\CCC(c)$ for all $c > 0$, then $\phi$ has the global Bowen property.
\end{proposition}

\begin{proof}
Suppose that $\phi$ is constant on $B(\Sing,\delta)$ for some $\delta > 0$. Let $c$ and $T$ be as in Lemma~\ref{lemma: nbhd of Sing} for $\frac{\delta}{2}$. Furthermore, assume $\phi$ has the Bowen property on $\CCC(c)$ for all $\epsilon < \epsilon_0$. Let $x\in G\SSS$, let $\epsilon < \min\{\delta/2, \epsilon_0\}$, and let $t > 0$. Suppose that $d_{G\SSS}(g_sy,g_sx)\leq \epsilon$ for $0\leq s\leq t$.

We have four cases to consider, beginning with a trivial one.

\

\noindent\textbf{Case 0:} $t\leq 2T$.

Trivially,
\[ \left| \int_0^t \phi(g_sx) - \phi(g_sy)\right| \leq 4T\Vert\phi\Vert,\]
where $\Vert\phi\Vert$ is the maximum value of $|\phi|.$ 

\

Henceforth we can assume $t>2T$.

\

\noindent\textbf{Case 1:} $\lambda(g_sx)<c$ for all $s\in [0,t]$.

By Lemma \ref{lemma: nbhd of Sing}, for $s\in [T, t-T]$, we have $g_sx\in B(\Sing,\delta/2)$. Since $y\in B_t(x,\epsilon)$ and $\epsilon<\delta/2$, $g_sy\in B(\Sing,\delta)$ for $s\in [T, t-T].$ Then as $\phi$ is locally constant on $B(\Sing,\delta)$,
\[ \left| \int_0^t \phi(g_sx) - \phi(g_sy) ds\right| \leq 4T\Vert\phi\Vert\]
as $\left|\int_T^{t-T}\phi(g_sx) - \phi(g_sy) ds\right|=0$. 

\

Henceforth, we can assume $\lambda$ takes on a value greater than or equal to $c$ at some point along the orbit segment $(x,t)$. Let $t_1 = \min\{s\geq 0\mid \lambda(g_sx) \geq c\}$ and $t_2 = \max\{s\leq t\mid \lambda(g_sx)\geq c\}$. If $t_1=t_2$, then the simpler version of the estimates below work, so we assume $t_2>t_1$. Then $(g_{t_1}x,t_2 - t_1) \in \CCC(c)$. Furthermore, for all $s\in [0,t_1]\cup[t_2,t]$, we have that $\lambda(g_sx) \leq c$. The proof is completed by applying our final two cases to the intervals $[0,t_1]$ and $[t_2, t]$ as appropriate.

\

\noindent\textbf{Case 2:} $t_1\leq2T$ and $t-t_2\leq 2T$.

We present the situation where both of these inequalities hold. Applying the non-uniform Bowen property (Prop \ref{prop:non-uniform Bowen}) for $(g_{t_1}x,t_2-t_1)$, we have
\begin{align*} 
    \left| \int_0^t \phi(g_sx) - \phi(g_sy) ds\right| & \leq \left| \int_0^{t_1} \phi(g_sx) - \phi(g_sy) ds\right| +\left| \int_{t_1}^{t_2} \phi(g_sx) - \phi(g_sy) ds\right|  + \left| \int_{t_2}^t \phi(g_sx) - \phi(g_sy) ds\right| \\
    &\leq 4T\Vert\phi\Vert + K + 4T\Vert\phi\Vert,
\end{align*}
where $K$ is given by the non-uniform Bowen property.

\

\noindent\textbf{Case 3:} $t_1>2T$ or $t-t_2>2T$.

We present the argument in a situation where both of these hold as, otherwise, we can use an estimate as in Case 0 for $\left| \int_0^{t_1} \phi(g_sx) - \phi(g_sy) ds\right|$ or $\left| \int_{t_2}^t \phi(g_sx) - \phi(g_sy) ds\right|$.
By Lemma~\ref{lemma: nbhd of Sing}, for $s\in [T,t_1 - T]$, we have $g_sx\in B(\Sing,\frac{\delta}{2})$, and similarly for $s\in [t_2 + T, t - T]$. Then if $y\in B_t(x,\epsilon)$, because $\epsilon < \delta/2$, it follows that $g_sy\in B(\Sing,\delta)$ for $s\in [T,t_1 - T]$ and $s\in [t_2 + T, t - T]$. Therefore, because $\phi$ is locally constant on $B(\Sing,\delta)$, we have that
$$\left|\int_{0}^{t_1}\phi(g_sx) - \phi(g_sy)\,ds\right| \leq \left|\int_{0}^{T}\phi(g_sx) - \phi(g_sy)\,ds\right| + 0 + \left|\int_{t_1 - T}^{t_1}\phi(g_sx) - \phi(g_sy)\,ds\right| \leq 4T\lVert\phi\rVert.$$
Similarly, we can bound
$$\left|\int_{t_2}^{t}\phi(g_sx) - \phi(g_sy)\,ds\right| \leq 4T\lVert\phi\rVert.$$
Now, since $y\in B_t(x,\epsilon)$, $g_{t_1}y\in B_{t_2 - t_1}(g_{t_1}x,\epsilon)$. By Prop \ref{prop:non-uniform Bowen}, there exists $K > 0$ (independent of $x,y,t_1,t_2$) such that
$$K\geq \left|\int_{0}^{t_1 - t_2}\phi(g_s(g_{t_1}x)) - \phi(g_s(g_{t_1}y))\,ds\right| = \left|\int_{t_1}^{t_2} \phi(g_sx) - \phi(g_sy)\,ds\right|.$$
Combining this with our previous inequalities, we have
\begin{align*}
\left|\int_{0}^{t}\phi(g_sx) - \phi(g_sy)\,ds\right| &\leq \left|\int_{0}^{t_1}\phi(g_sx) - \phi(g_sy)\,ds\right| + \left|\int_{t_1}^{t_2}\phi(g_sx) - \phi(g_sy)\,ds\right| + \left|\int_{t_2}^{t}\phi(g_sx) - \phi(g_sy)\,ds\right|
\\
&\leq 4T\lVert\phi\rVert + K + 4T\lVert\phi\rVert.
\end{align*}
As our choice of $T$ and $K$ were independent of $x$, $t$, and $y$, this implies the Bowen property.
\end{proof}

%
\subsection{Upper and lower Gibbs properties}\label{subsec:Gibbs}

\subsubsection{Lower Gibbs}\label{sec:lower gibbs}

One of the results of \cite{Climenhaga-Thompson} shows that for any system satisfying the conditions of \cite[Theorem A]{Climenhaga-Thompson}, the unique equilibrium state has a type of lower Gibbs property. In \cite{CCESW}, we showed that in the setting of Theorems \ref{mainthm:LPS} and \ref{mainthm:Bernoulli} with $\phi$ H\"older continuous and locally constant on $\Sing$, $(\GS,g_t,\mu,\phi)$ satisfy conditions sufficient to apply \cite{Climenhaga-Thompson}. In particular, we show that the set of orbit segments $\CCC(c) := \{(x,t)\mid \lambda(x),\lambda(g_tx)\geq c\}$ satisfies the specification property in the proof of \cite[\S 4 \text{ and }\S 5]{CCESW}. Consequently, we have the following corollary.

\begin{corollary}\label{cor: unif Lower Gibbs}
Let $(\GS, \mu, g_t)$ be as in Theorems \ref{mainthm:LPS} and \ref{mainthm:Bernoulli}, and let $\phi$ be H\"older continuous and locally constant on $\Sing$. Then $\mu$ has the following lower Gibbs property for a collection of orbit segments.

Let $\rho,\eta>0$. If $\{(x_i, t_i)\}$ is a collection of orbit segments in $\GS\times [0,\infty)$ such that for all $i$, $\lambda(x_i), \lambda(g_{t_i} x_i)\geq \eta$, then there exists a constant $Q_2(\rho, \eta)$ depending only on $\rho$ and $\eta$ such that for all $i$,
\[ \mu(B_{t_i}(x_i, \rho)) \geq Q_2(\rho,\eta) e^{-t_i P(\phi) + \int_{0}^{t_i}\phi(g_sx_i)\,ds}.\]
\end{corollary}

This gives us the following useful result.

\begin{corollary}\label{cor: zero measure Sing}
Let $(\GS, \mu, g_t)$ be as in Theorems \ref{mainthm:LPS} and \ref{mainthm:Bernoulli}. Then $\mu(\lambda^{-1}(0)) = 0$. In particular, $\mu(\Sing) = 0$.
\end{corollary}

\begin{proof}
Let $x\in \GS$ be such that $\lambda(x) > 0$. Then, because $\lambda$ is lower semi-continuous, there exists a ball $B(x,\rho)$ such that $\lambda|_{B(x,\rho)} > 0$.  As the lower Gibbs property tells us that $\mu(B(x,\rho)) > 0$, this implies that $\mu$-a.e. $y\in \GS$ visits $B(x,\rho)$ infinitely often in both forwards and backwards time. Thus, for almost every $y$, $\lambda(y) \neq 0$, as if $\lambda(y) = 0$, then $\lambda(g_ty) = 0$ either for all $t\geq 0$ or all $t \leq 0$ \cite[Proposition 3.4]{CCESW}. As $\Sing \subset \lambda^{-1}(0)$, it follows that $\mu(\Sing) = 0$ as well.
\end{proof}

%

\subsubsection{Upper Gibbs}\label{sec:upper gibbs}

The upper Gibbs bound is improved from \cite{Climenhaga-Thompson} by our proof that $\phi$ has the Bowen property globally.

\begin{proposition}\label{prop: Upper Gibbs}
Let $(\GS, \mu, g_t)$ be as in Theorems \ref{mainthm:LPS} and \ref{mainthm:Bernoulli}. Then, $\mu$ has the upper Gibbs property. That is, for all $\rho > 0$ sufficiently small, there exists $Q := Q(\rho) > 0$ such that 
$$\mu(B_t(x,\rho)) \leq Qe^{-tP(\phi) + \int_{0}^{t}\phi(g_sx)\,ds}.$$
\end{proposition}

\begin{remark*}
We note that the upper bound on $\rho$ is less than half of the injectivity radius of $S$ (see \cite[Lemmas 2.16 and 2.17]{CCESW}) and small enough that $\phi$ has the Bowen property.
\end{remark*}

\begin{proof}
Because in our setting we have specification at all scales \cite[\S5]{CCESW}, applying \cite[Proposition 4.21]{Climenhaga-Thompson}, for all sufficiently small $\rho > 0$, there exists $Q' > 0$ such that
$$\mu(B_t(x,\rho)) \leq Q'e^{-tP(\phi) + \sup\limits_{y\in B_t(x,\rho)}\int_{0}^{t}\phi(g_sy)\,ds(y,t)}.$$
Since $\phi$ has the Bowen property at scale $\rho$ by Proposition \ref{proposition: global Bowen property}, we have that
$$\sup_{y\in B_t(x,\rho)}\int_{0}^{t}\phi(g_sy)\,ds(y,t) \leq K + \int_{0}^{t}\phi(g_sx)\,ds(x,t).$$
Thus, it follows that
$$\mu(B_t(x,\rho)) \leq Q'e^{K}e^{-tP(\phi) + \int_{0}^{t}\phi(g_sx)\,ds(x,t)}.$$
\end{proof}

%
\subsection{Local product structure for equilibrium states}\label{subsec:lps}

A key tool in our proof of the Bernoulli property for $(\GS, g_t, \mu)$ is local product structure of the equilibrium measure $\mu$. Roughly speaking, the stable, unstable, and flow directions provide a \emph{topological} product structure near any regular geodesic. The equilibrium state $\mu$ has local product structure when it is absolutely continuous with respect to a product measure built using this topological local product structure (see Definition \ref{defn:lps} for the precise definition).

As usual, the non-uniformly hyperbolic nature of our system only allows us to prove this structure near regular geodesics, with the constants involved depending on the level of regularity.

Our approach to the proof follows that of \cite{ClimenhagaGibbs}, in which Climenhaga shows that the Gibbs property can be used to deduce local product structure for Anosov diffeomorphisms. Our modifications to Climenhaga's scheme revolve around adjusting the argument to flows instead of maps (replacing the stable leaves with center-stable leaves and using the invariance of $\mu$ under the flow) and (a much more significant change) accounting for the non-uniformity of our lower Gibbs bound (Corollary \ref{cor: unif Lower Gibbs}).

Sections \ref{sec: Partitions} and \ref{subsec: relations bowen} provide, respectively, a technical construction of special partitions related to the Bowen balls, and preliminary results on Bowen balls and their relation to the topological local product structure. Section \ref{sec:lps} is the heart of the matter, dedicated to proving Proposition \ref{prop: LPS} and Corollary \ref{cor: ae lps}. The local product structure of $\mu$ follows from these.

%

\subsubsection{Partition Construction}\label{sec: Partitions}

Since the construction below might be of the independent interest, we formulate it in a considerably more general setting. Let $(X,d)$ be a compact metric space with the metric $d$, and let $\mathcal F=(f_t)_{t\in\mathbb R}\colon X\rightarrow X$ be a continuous flow. In this section, we construct an ``almost adapted'' partition of a compact subset of $X$ with a special property (see Proposition~\ref{prop: partition construction}). We call this partition almost adapted as each element of the partition contains and is contained in a Bowen ball (see \cite[Definition 4.1]{ClimenhagaGibbs} for the definition of adapted partitions). However, in our case the time parameter of the Bowen balls depends on the element of the partition. We begin with a lemma that we will eventually use to construct our partitions.

\begin{lemma}\label{lemma: disjoint balls}
Suppose that a set $A$ has the following separation property: for each $x\in A$, there exists $t_x > 0$ such that for any $x\neq y\in A$, $\max\{d(f_tx,f_ty)\,|\,t\in[0,\min\{t_x,t_y\}]\}\geq\eps$. Then the collection $\{B_{t_x}(x,\frac{\eps}{2})\}_{x\in A}$ is pairwise disjoint.
\end{lemma}

\begin{proof}
Let $x\neq y\in A$. Without loss of generality, assume $t_x\leq t_y$. Towards a contradiction, suppose $z\in B_{t_x}(x,\frac{\eps}{2})\cap B_{t_y}(y,\frac{\eps}{2})$. Then, we have 
\begin{align*}
\eps\leq\max\{d(f_tx,f_ty)\,|\,t\in[0,t_x]\}&\leq \max\{d(f_tx,f_tz)\,|\,t\in[0,t_x]\}+\max\{d(f_tz,f_ty)\,|\,t\in[0,t_x]\}\\&\leq \max\{d(f_tx,f_tz)\,|\,t\in[0,t_x]\}+\max\{d(f_tz,f_ty)\,|\,t\in[0,t_y]\}\\&<\frac{\eps}{2}+\frac{\eps}{2}=\eps,
\end{align*}
which is not possible. Therefore, $B_{t_x}\left(x,\frac{\eps}{2}\right)\cap B_{t_y}\left(y,\frac{\eps}{2}\right)=\emptyset$.
\end{proof}

\begin{definition}
    Let $A\subset X$. For $T,\eps>0$ and $E\subset A$, we say that $E$ is a \emph{$(T,\eps)$-separated subset of $A$} if $\max\{d(f_tx,f_ty)\,|\,t\in[0,T]\}\geq \eps$ for all distinct $x,y\in E$.
\end{definition}

\begin{proposition}\label{prop: partition construction}
	Consider a compact subset $K$ of $X$. Let $\lambda : X\to [0,\infty)$ be a non-negative function and assume $\lambda \circ f_t: K\to [0,\infty)$ is a continuous function for all $t\in \mathbb{R}$. Let $T_0,c > 0$. Further, suppose there exists $\kappa > 0$ such that for all $x\in K$, $\lambda(f_{T_0 + n\kappa + r}x)\geq c$ for all $|r|\leq \kappa$ for infinitely many $n\in\mathbb{N}$. Then for all $\eps>0$, $m\in\mathbb N$ there exists a finite partition $\xi_m$ of $K$ such that for each $A\in \xi_m$, there exist $x\in K$ and $t\geq T_0+m\kappa$ with $\lambda(f_{t+r}x)\geq c$ for $r\in [-\kappa,\kappa]$ and
	$$K\cap B_{t}\left(x,\frac{\eps}{2}\right) \subset A\subset K\cap B_t(x,\eps).$$
	Further, for all $n\in\mathbb{N}$, we can choose $(x,t)$ so that $\lambda(f_{T_0 + r_i\kappa} x)\geq c$ for $r_1 < r_2 < \cdots < r_n$ with $r_{i+1}-r_i\geq 2$ and $t \geq T_0 + r_n\kappa$.
\end{proposition}

We call the partitions  $\xi_m$ given by Proposition \ref{prop: partition construction} \textit{almost adapted}.

\begin{remark*}
    The essence of what we want from the points $x$ in this proposition is that their orbits continually return to $\{x\in X\,|\, \lambda(x) \geq c\}$ and hence experience the hyperbolicity of the flow. The precise conditions on $x$ in this proposition (in particular, remaining in $\{x\in X\,|\, \lambda(x) \geq c\}$  over the interval $[t-\kappa, t+\kappa]$) are imposed to make later arguments slightly simpler (in particular, allowing us to remove some cases from a future argument). However, as we note below in Lemma \ref{lemma: recurring lambda}, while the condition we impose on $x$ may seem complicated, it is not too difficult to check.
\end{remark*}

\begin{proof}
	Let $T=T_0+m\kappa$. First, we define $K_0 := \{x\in K\mid \lambda(f_{T + r}x)\geq c \text{ for } |r|\leq \kappa\}$. We note that $K_0$ is compact using the facts that $\lambda\circ f_T$ is continuous and $K$ is compact. Additionally, for each $x\in K$ and $i\in \mathbb{N}$, let $\Psi_i(x)$ be the number of $p\in\mathbb{N}$ such that both $T_0 + p\kappa \leq T + i\kappa$ and $\lambda(f_{T_0 + p\kappa}x)\geq c$. Then, defining
	$$F_0 := \{x\in K_0\mid \Psi_0(x)\geq n\},$$
	we have that $F_0$ is compact, as the set $\{x\in K \mid \Psi_0(x)\geq n\}$ is compact (it is a finite union of compact sets). Let $A_0$ be a maximal $(T,\eps)$-separated subset of $F_0$ that exists because $F_0$ is compact. Then, for all $i\in\mathbb N$, we define $K_i$ and $F_i$ inductively in the following way. We set
	$$K_i := \{x\in K\mid \lambda(f_{T+i\kappa + r}x)\geq c \text{ for } |r|\leq \kappa\}\setminus \bigcup_{j=0}^{i-1}\bigcup_{x\in A_j}(K\cap B_{T+j\kappa}(x,\eps)),$$
	let
	$$F_i := \{x\in K_i\mid \Psi_i(x)\geq n\},$$
	and let $A_i$ be a maximal $(T + i\kappa,\eps)$-separated subset of $F_i$. We note that $F_i$ is compact for all $i\in\mathbb N$ so $A_i$ exists. Observe from this construction that $\{B_{T+i\kappa}(x,\eps)\mid 0\leq i \leq N, x\in A_i\}$ is an open cover of $\bigcup_{i=0}^{N} F_i$, as otherwise, there would be some $0\leq i \leq N$ and $y\in F_i$ for which $A_i\cup\{y\}$ is $(T + i\kappa,\eps)$-separated. Since for all $x\in K$, $\lambda(f_{T+i\kappa}x)\geq c$ for infinitely many $i\in\mathbb{N}$, there exists some $N$ for which $g_N(x)\geq n$ and so either $x\in F_N$ or $x\in \bigcup_{i=0}^{N-1}\bigcup_{x\in A_i}(K\cap B_{T+i\kappa}(x,\eps))$. Thus, $\{B_{T+n\kappa}(x,\eps)\mid n\in\mathbb{N},x\in A_n\}$ is an open cover of $K$.
	
	Since $K$ is compact, we can choose a minimal, finite subcover $\{B_{T+n_i\kappa}(x_i,\eps)\}_{i=1}^k$. The minimality guarantees that $x_i\neq x_j$ for $i\neq j$.

	\begin{claim}\label{claim: separation}
		$\max\{d(f_tx_i,f_tx_j)\,|\, t\in [0,T+\min\{n_i,n_j\}\kappa]\} \geq \eps$ for $i\neq j$.
	\end{claim}
	
	\begin{proof}
		Without loss of generality, we can assume $n_i\leq n_j$. If $n_i=n_j$, the claim follows from the fact that $A_{n_i}$ is $(T+n_i\kappa, \eps)$-separated. If $n_i<n_j$, then $x_j\notin B_{T+n_i\kappa}(x_i,\eps)$ by construction (as $x_j\in K_{n_j}$), completing our proof. 
	\end{proof}

	By Claim~\ref{claim: separation} and Lemma \ref{lemma: disjoint balls}, $\{B_{T+n_i\kappa}(x_i,\frac{\eps}{2})\}_{i=1}^k$ are pairwise disjoint. 
	Finally, we construct the partition $\xi_m=\{P_1,\ldots, P_k\}$ in the following way by induction.  Let $P_1=B_{T+n_1\kappa}(x_1,\eps)\cap K$. Then, for $i=2,\ldots, k$, we have
	$P_i=K\cap\left(B_{T+n_i\kappa}(x_i,\eps)\setminus\bigcup_{j=1}^{i-1}P_j\right)$.
\end{proof}

\begin{lemma}\label{lemma: recurring lambda}
Let $\kappa > 0$, $x\in X$, and suppose there exists a sequence $t_k\to \infty$ for which $\lambda(f_{t_k + s}x)\geq c$ for all $|s|\leq 2\kappa$. Then for all $T_0 > 0$, $\lambda(f_{T_0 + n\kappa + s}x) \geq c$ for all $|s|\leq \kappa$ and infinitely many $n$.
\end{lemma}

\begin{proof}
Let $T_0 > 0$. Then for any $t_k \geq T_0$, write $t_k = T_0 + n_k\kappa + r$ for some $0 \leq r < \kappa$. Then, observe that for all $|s| \leq \kappa$, we have that $T_0 + n_k\kappa + s = t_k + s - r$ and $|s - r| \leq 2\kappa$. By assumption on $t_k$, we have completed our proof.
\end{proof}

\begin{corollary}\label{cor: small diameter partition}
Assume we are in the setting of Proposition~\ref{prop: partition construction}. If $K$ has an additional property that for all $x\in K$, $\lim\limits_{t\rightarrow\infty}\diam \left(B_t(x,\eps)\right)=0$, then there exists an almost adapted partition $\xi$ of $K$ with arbitrarily small diameter.
\end{corollary}

\begin{proof}
Let $\alpha > 0$ be arbitrary. By Dini's Theorem, $\diam \left(B_t(x,\eps)\right)\to 0$ as $t\rightarrow\infty$ uniformly, because $x\mapsto \diam B_t(x,\eps)$ is continuous. Then, choose $m\in\mathbb N$ in Proposition \ref{prop: partition construction} large enough so that $\sup\limits_{x\in K}\diam B_t(x,\eps) \leq \alpha$ for all $t\geq T_0+m\kappa$.
\end{proof}

We have created this partition to have this ``almost adapted'' property so that we can later obtain good measure estimates on the partition elements using the Gibbs property. However, because the compact set $K$ ``cuts into'' the Bowen balls, estimates can sometimes be more difficult. However, after construction, this partition can be extended to fill out the inner Bowen balls in the following manner.

\begin{corollary}\label{cor: fill out partition}
	Let $\xi = \{A_1,\cdots, A_n\}$ be a partition constructed as in Proposition~\ref{prop: partition construction}. Then there exists a partition $\eta = \{B_0, B_1,\cdots, B_n\}$ of $X$ such that for $1\leq i \leq n$, each partition element $B_i\in \eta$ satisfies
	$$B_{t_i}\left(x_i,\frac{\eps}{2}\right) \subset B_i\subset B_{t_i}(x_i,\eps).$$
	Further, in the setting of Corollary \ref{cor: small diameter partition}, $\eta$ can be constructed so that $B_1,\cdots, B_n$ have arbitrarily small diameter.
\end{corollary}

\begin{proof}
	By construction, $\{B_{t_i}(x_i,\frac{\eps}{2})\}_{i=1}^n$ are pairwise disjoint, and so we can simply construct $B_1 := B_{t_1}(x_1,\eps)$ and $B_i = B_{t_i}(x_i,\eps)\setminus \bigcup_{j=1}^{i-1}B_i$. We then take $B_0 = X\setminus \bigcup_{i=1}^n B_i$. (See Figure \ref{fig:partitions} part (a) for a schematic illustration of this construction.)
\end{proof}

%

\subsubsection{A refining sequence of almost adapted partitions}

We continue in the same general framework as the previous subsection. In order to estimate conditional measures, we will need to work with a refining sequence of partitions. However, we need each of the partitions in this sequence to have the adapted property discussed above, or else we won't be able to obtain measure estimates on the partition elements. While Proposition \ref{prop: partition construction} lets us obtain a sequence of almost adapted partitions with diameter going to $0$, they are not refining. Consequently, constructing such a sequence with both is nontrivial. However, by omitting a subset of arbitrarily small measure, we are able to construct such a refining sequence.

\begin{proposition}\label{prop: refining sequence}
Assume we are in the setting of Proposition \ref{prop: partition construction} and Corollary \ref{cor: small diameter partition}, and that $\mu$ is a Radon probability measure on $X$. Then for all $\eps \in (0,\frac{1}{2})$, there exists a refining sequence $\{\xi_i\}_{i\in\mathbb{N}}$ of partitions of $X$ with $\operatorname{diam}\xi_i \to 0$ such that there exists $E\subset X$ with the following properties:
	\begin{itemize}
	\item $\mu(E) > \mu(K) - \eps$,
	\item For all $x\in E$, for all $i\in\mathbb{N}$, there exists $y\in X$ and $t\geq 0$ such that $\lambda(f_ty)\geq c$ and
	$$B_t\left(y,\frac{\eps}{2}\right) \subset \xi_i(x)\subset B_t(y,\eps),$$
	where $\xi_i(x)$ is the partition element of $\xi_i$ containing $x$.
\end{itemize}
\end{proposition}

\begin{proof}
Let $\eps \in (0,\frac{1}{2})$ be arbitrary. Then let $\zeta_1 = \{A_0,A_1,\cdots, A_n\}$ be a partition constructed as in Corollary \ref{cor: fill out partition} with $\operatorname{diam}A_i \leq \frac{1}{2}$ for $1\leq i\leq n$. (See Figure \ref{fig:partitions} for a schematic illustration of this whole construction.)

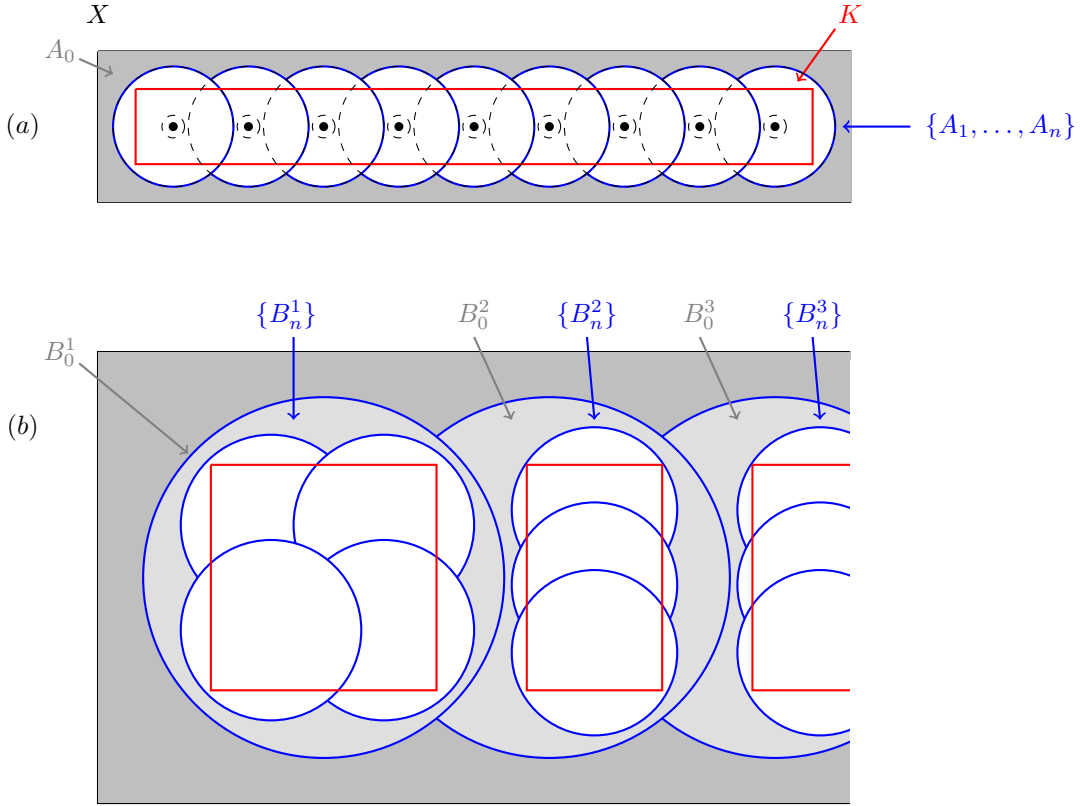
\begin{figure}[h]
\centering
\begin{tikzpicture}[scale=1]

\draw[thick] (-5,1) -- (-5,-1) -- (5,-1) -- (5,1) -- (-5,1) ;

\fill [gray!50!white] (-5,1) -- (-5,-1) -- (5,-1) -- (5,1) -- (-5,1) ;
\fill [white] (4,0) circle(.8) ;
\draw[thick, blue] (4,0) circle(.8) ;
\fill [white] (3,0) circle(.8) ;
\draw[thick, blue] (3,0) circle(.8) ;
\fill [white] (2,0) circle(.8) ;
\draw[thick, blue] (2,0) circle(.8) ;
\fill [white] (1,0) circle(.8) ;
\draw[thick, blue] (1,0) circle(.8) ;
\fill [white] (0,0) circle(.8) ;
\draw[thick, blue] (0,0) circle(.8) ;
\fill [white] (-1,0) circle(.8) ;
\draw[thick, blue] (-1,0) circle(.8) ;
\fill [white] (-2,0) circle(.8) ;
\draw[thick, blue] (-2,0) circle(.8) ;
\fill [white] (-3,0) circle(.8) ;
\draw[thick, blue] (-3,0) circle(.8) ;
\fill [white] (-4,0) circle(.8) ;
\draw[thick, blue] (-4,0) circle(.8) ;

\node at (-5,1.5) {$X$} ;
\node[red] at (5,1.5) {$K$} ;
\draw[thick, red, ->] (4.8,1.3) -- (4.3,.6) ;
\node at (-6,0) {$(a)$} ;
\node[blue] at (7,0) {$\{A_1, \ldots, A_n\}$} ;
\draw[thick, blue, ->] (5.8,0) -- (4.9,0) ;
\node[gray] at (-5.5,1) {$A_0$} ;
\draw[thick, gray, ->] (-5.25,.9) -- (-4.8,.7) ;

\draw[thick,red] (-4.5,.5) -- (-4.5,-.5) -- (4.5,-.5) -- (4.5,.5) -- (-4.5,.5) ;

\fill (-4,0) circle(.06);
\fill (-3,0) circle(.06);
\fill (-2,0) circle(.06);
\fill (-1,0) circle(.06);
\fill (0,0) circle(.06);
\fill (1,0) circle(.06);
\fill (2,0) circle(.06);
\fill (3,0) circle(.06);
\fill (4,0) circle(.06);

\draw[dashed] (-4,0) circle(.8) ;
\draw[dashed] (-4,0) circle(.15) ;
\draw[dashed] (-3,0) circle(.8) ;
\draw[dashed] (-3,0) circle(.15) ;
\draw[dashed] (-2,0) circle(.8) ;
\draw[dashed] (-2,0) circle(.15) ;
\draw[dashed] (-1,0) circle(.8) ;
\draw[dashed] (-1,0) circle(.15) ;
\draw[dashed] (0,0) circle(.8) ;
\draw[dashed] (0,0) circle(.15) ;
\draw[dashed] (1,0) circle(.8) ;
\draw[dashed] (1,0) circle(.15) ;
\draw[dashed] (2,0) circle(.8) ;
\draw[dashed] (2,0) circle(.15) ;
\draw[dashed] (3,0) circle(.8) ;
\draw[dashed] (3,0) circle(.15) ;
\draw[dashed] (4,0) circle(.8) ;
\draw[dashed] (4,0) circle(.15) ;


\node at (-6,-4) {$(b)$} ;

\draw[thick] (5,-3) -- (-5,-3) -- (-5,-9) -- (5,-9) ;
\fill [gray!50!white] (5,-3) -- (-5,-3) -- (-5,-9) -- (5,-9) -- (5,-3) ;

\fill [gray!25!white] (4,-6) circle(2.4) ;
\draw[thick, blue] (4,-6) circle(2.4) ;
\fill [gray!25!white] (1,-6) circle(2.4) ;
\draw[thick, blue] (1,-6) circle(2.4) ;
\fill [gray!25!white] (-2,-6) circle(2.4) ;
\draw[thick, blue] (-2,-6) circle(2.4) ;

\fill [white] (-2.7,-5.3) circle(1.2) ;
\draw [thick, blue] (-2.7,-5.3) circle(1.2) ;
\fill [white] (-1.2,-5.3) circle(1.2) ;
\draw [thick, blue] (-1.2,-5.3) circle(1.2) ;
\fill [white] (-1.2,-6.7) circle(1.2) ;
\draw [thick, blue] (-1.2,-6.7) circle(1.2) ;
\fill [white] (-2.7,-6.7) circle(1.2) ;
\draw [thick, blue] (-2.7,-6.7) circle(1.2) ;

\fill [white] (1.6,-5.1) circle(1.1) ;
\draw [thick, blue] (1.6,-5.1) circle(1.1) ;
\fill [white] (1.6,-6.1) circle(1.1) ;
\draw [thick, blue] (1.6,-6.1) circle(1.1) ;
\fill [white] (1.6,-7) circle(1.1) ;
\draw [thick, blue] (1.6,-7) circle(1.1) ;

\fill [white] (4.6,-5.1) circle(1.1) ;
\draw [thick, blue] (4.6,-5.1) circle(1.1) ;
\fill [white] (4.6,-6.1) circle(1.1) ;
\draw [thick, blue] (4.6,-6.1) circle(1.1) ;
\fill [white] (4.6,-7) circle(1.1) ;
\draw [thick, blue] (4.6,-7) circle(1.1) ;

\draw [thick, red] (-3.5,-4.5) -- (-.5,-4.5) -- (-.5,-7.5) -- (-3.5,-7.5) -- (-3.5,-4.5) ;
\draw [thick, red] (.7,-4.5) -- (2.5,-4.5) -- (2.5,-7.5) -- (.7,-7.5) -- (.7,-4.5) ;
\draw [thick, red] (3.7,-4.5) -- (5.1,-4.5) -- (5.1,-7.5) -- (3.7,-7.5) -- (3.7,-4.5) ;

\fill [white] (5,-3.1) -- (5,-9.1) -- (8.1,-9.1) -- (8.1,-3.1) -- (5,-3.1);

\node[gray] at (-5.5,-3) {$B^1_0$} ;
\draw [thick, gray, ->] (-5.25, -3.15) -- (-3.8,-4.35) ;
\node[blue] at(-2.5,-2.5) {$\{ B^1_n \}$} ;
\draw [thick, blue, ->] (-2.4, -2.8) -- (-2.4,-3.9) ;
\node[gray] at (0,-2.5) {$B_0^2$} ;
\draw [thick, gray, ->] (0, -2.8) -- (.5,-3.9) ;
\node[blue] at(1.5,-2.5) {$\{ B^2_n \}$} ;
\draw [thick, blue, ->] (1.5, -2.8) -- (1.6,-3.9) ;
\node[gray] at (3,-2.5) {$B_0^3$} ;
\draw [thick, gray, ->] (3, -2.8) -- (3.5,-3.9) ;
\node[blue] at(4.5,-2.5) {$\{ B^3_n \}$} ;
\draw [thick, blue, ->] (4.5, -2.8) -- (4.6,-3.9) ;

\end{tikzpicture}
\caption{Construction of the refining and nearly-everywhere, almost adapted sequence of partitions $\xi_i$. In (a), the partition $\zeta_1$ is constructed, following Corollary \ref{cor: fill out partition}. The compact set $K$ is chosen to provide room between for the Bowen balls in the construction to fit entirely inside $X$. The Bowen balls are drawn in dashed lines, and show the almost adapted nature of (most of) the partition elements (i.e., those outlined in blue). Note that, in contrast to the simplified case of this schematic, the radii of these Bowen balls may vary.
In (b), we zoom in and see the next iteration of the construction. Within each $A_i$ ($i>0$) Corollary \ref{cor: fill out partition} is applied again to build $\{B^i_n\}.$ The gray shaded regions are the index-zero elements of the partitions, and collectively have small measure. In the final step of the construction of $\xi_i$, these elements are cut up into small diameter pieces.}\label{fig:partitions}
\end{figure}

By construction, $\mu(\bigcup_{i=1}^nA_i) \geq \mu(K)$. Further, using lower semi-continuity of $\lambda$, this partition can be constructed so that $\mu(\partial A_i) = 0$ for $1\leq i \leq n$, as $\partial A_i$ is made up of the boundaries of Bowen balls. Now, for $1\leq i \leq n$, take a compact subset $K_i\subset \operatorname{int}(A_i)\cap K$ with measure $\mu(K_i)\geq (1-\eps^2)\mu(A_i\cap K)$. Applying results of the previous section, we construct partitions $\{B^i_0,B^i_1,\ldots, B^i_{m_i}\}$ of each $A_i$, for which each $B^i_j$ has the desired adapted property, $\operatorname{diam}(B^i_j) < \max\{\frac{1}{4},d(\partial A_i,K_i)\}$, and $\mu(B^i_0\cap K) \leq \eps^2\mu(A_i\cap K)$. We then define a partition $\xi_2$ of $X$ by
$$\zeta_2 := \{A_0,B^1_0,\ldots,B^1_{m_1},\ldots,B^n_0,\ldots, B^n_{m_n}\}.$$
By our assumption on the diameter, we have that $\zeta_2$ is a refinement of $\zeta_1$, as $B^i_j\subset A_i$. Further, the only partition elements without the desired adapted property have measure
$$\mu(A_0\cup B^1_0\cup\cdots \cup B^n_0) \leq \mu(X\setminus K) + \sum_{i=1}^n\mu(B^i_0\cap K) \leq 1-\mu(K) + \sum_{i=1}^n\eps^2\mu(A_i\cap K) \leq 1-\mu(K) + \eps^2.$$
We can again ensure that $\mu(\partial B^i_j) = 0$ for all $i,j$, and to obtain a refinement $\zeta_3$, we repeat this process to obtain partitions of $B^i_j$ for $1\leq i \leq n$ and $j\geq 1$, for which all elements but a set of measure $\eps^3\mu(B^i_j\cap K)$ have the desired adapted property with diameter at most $\frac{1}{8}$. Consequently, at this stage, the partition elements of $\zeta_3$ which do not have the adapted property have measure at most
$$1-\mu(K) + \eps^2 + \sum_{i=1}^n\sum_{j=1}^{m_i}\eps^3\mu(B^i_j\cap K) \leq 1-\mu(K) + \eps^2 + \eps^3.$$
Continuing this process for all $i\in\mathbb{N}$, we see that the union of elements in any partition $\zeta_i$ without the adapted property has measure at most
$$1-\mu(K) + \sum_{i=2}^{\infty}\eps^i = 1-\mu(K) + \eps\frac{\eps}{1-\eps} < 1-\mu(K) + \eps.$$

This proves our result, except for the fact that the diameter of the partitions does not go to $0$ on these ``bad'' sets. To remedy this, note that because $X$ is compact, we can construct an arbitrary refining sequence of finite partitions $\{\eta_i\}_{i\in\mathbb{N}}$ whose diameters go to $0$. Let $\mathscr{G}_i\subset \zeta_i$ be the set of good adapted elements of our partitions $\zeta_i$ (i.e., $\{B^i_j\mid 1\leq i \leq n, j\geq 1\}$ in the case of $\zeta_2$). Then, writing
$$\xi_i := \{A \mid A\in\mathscr{G}_i\}\cup \{A\cap B\mid A\in \zeta_i\setminus \mathscr{G}_i, B\in \eta_i\}.$$
Then $\lim\limits_{i\to\infty}\operatorname{diam}\xi_i = 0$ and $\{\xi_i\}_{i\in \mathbb{N}}$ satisfies the rest of our desired properties.
\end{proof}

%
\subsubsection{Interaction of Bowen balls and Bowen brackets}\label{subsec: relations bowen}

Throughout this section we use the following notations. Let $z_0\in \GS$ be such that $\lambda(z_0) \geq \frac{3}{4}\lVert\lambda\rVert_{\infty}$, and, using lower semicontinuity of $\lambda$, let $\kappa > 0$ be chosen so that $\lambda|_{B(z_0,3\kappa)} \geq \frac{2}{3}\lambda(z_0)$. Consider a parametrized geodesic $x_0\in \Reg$ and small enough $\delta$ as in Definition~\ref{def: su-Bowen bracket}, $\eps < \frac{1}{4}\{s_0,\delta\}$, and $n_0=\max\{\frac{8}{\delta}, 5, \frac{1}{s_0},\frac{4}{\kappa},\frac{1}{\eps}\}$. Let $\mathcal B(n_0,\eps)$ be the corresponding $(n_0,\eps)$-flow box centered at $x_0$ (see Definition~\ref{definition: flow box}). 

The following notation will prove useful throughout the remainder of the paper.

\begin{definition}\label{defn: local leaves}
Given a flow box $\mathcal{B}(n_0,\eps)$ and a point $x\in \mathcal{B}(n_0,\eps)$, let 
\[V_x^{cs} = W^{cs}(x,\delta) \cap \mathcal{B}(n_0,\eps) \ \mbox{ and } \ V_x^u = W^u(x,\delta) \cap \mathcal{B}(n_0,\eps)\]
be the local center stable and unstable leaves induced on $\mathcal{B}(n_0,\eps)$.
\end{definition}

\begin{lemma}\label{lem:ratio of flow boxes}
For all $-2\eps \leq a<b\leq 2\eps$ and $n \geq 4s_0^{-1}$, and for any measurable $A\subset V^u_x$ and $B\subset V^s_x$,
$$\frac{\mu\left(g_{\left[-\frac{1}{n},\frac{1}{n}\right]}\langle A,B\rangle\right)}{\mu\left(g_{[a,b]}\langle A,B\rangle\right)} = \frac{2}{n(b-a)}.$$
\end{lemma}

\begin{proof}
We prove the more general statement, where $D\subset G\SSS$ is such that for all $t\neq s$ with $|t|,|s|$ sufficiently small, $g_tD\cap g_sD = \emptyset$. By flow invariance of $\mu$, for all small $t_1,t_2,s$, we have $\mu\left(g_{[t_1,t_2]}D\right) = \mu\left(g_{[s + t_1,s + t_2]} D\right)$. 
For $-2\eps \leq a < b\leq 2\eps$ and $n \geq 4s_0^{-1}$, we have the necessary disjointness property above for $|t_i|\leq \max\left\{|a|,|b|,\frac{1}{n}\right\}$, as $|a|,|b|,$ and $\frac{1}{n}$ are less than one quarter the injectivity radius of $S$. Thus, consider $g_{\left[\frac{a}{n},\frac{b}{n}\right]}D$, and observe that 
$$ \frac{1}{n}\mu\left(g_{[a,b]}D\right) = \frac{b-a}{2}\mu\left(g_{\left[\frac{-1}{n},\frac{1}{n}\right]}D\right) = \mu\left(g_{\left[\frac{a}{n},\frac{b}{n}\right]}D\right).$$
\end{proof}

\begin{lemma}\label{lemma: bowen ball contains bracket}
Suppose $w,z\in g_{\tau}R(\eps)$ for some $\tau\in\left[-\frac{1}{n_0},\frac{1}{n_0}\right]$. Then, for all $n,m\geq0$,
$$g_{\left[-\frac{1}{n_0},\frac{1}{n_0}\right]}\langle B^u_m(w,\varepsilon),B_n^{s}(z,\varepsilon)\rangle \subset B_{[-n,m]}(\langle w,z\rangle,4\eps).$$ 
\end{lemma}

\begin{proof}
Let $w'\in B^u_m(w,\varepsilon)$, $z'\in B_n^{s}(z,\varepsilon)$. Then, write $\xi := \langle w',z\rangle$, $\eta := \langle w',z'\rangle$, and $\zeta := \langle w,z\rangle$. Then, for $-n\leq k \leq 0$, we have
\begin{align*}
d_{GS}(g_k\eta,g_k\xi) &= \int_{-\infty}^{+\infty}e^{-2|t|}d_{\tilde{S}}(g_k\tilde{\xi}(t),g_k\tilde{\eta}(t))\,dt
\\
&= \int_{-\infty}^{+\infty}e^{-2|t|}d_{\tilde{S}}(\tilde{\xi}(t+k),\tilde{\eta}(t+k))\,dt
\\
&= \int_{-\infty}^{-(k+\tau)}e^{-2|t|}d_{\tilde{S}}(\tilde{\xi}(t+k),\tilde{\eta}(t+k))\,dt+\int_{-(k+\tau)}^{+\infty}e^{-2|t|}d_{G\tilde{S}}(\tilde{\xi}(t+k),\tilde{\eta}(t+k))\,dt
\\
&= \int_{-\infty}^{-(k+\tau)}e^{-2|t|}d_{\tilde{S}}(\tilde{z}(t+k),\tilde{z'}(t+k))\,dt
\\
&\leq \int_{-\infty}^{+\infty}e^{-2|t|}d_{\tilde{S}}(\tilde{z}(t+k),\tilde{z'}(t+k))\,dt
\\
&\leq d_{G\tilde{\SSS}}(g_k\tilde{z},g_k\tilde{z'})
\\
&\leq \eps.
\end{align*}
Furthermore, as $\eta\in W^s(\xi,\eps)$, it follows that $d_{GS}(g_k\eta,g_k\xi) \leq \eps$ for all $k\geq 0$ as well.

Now, using a similar argument, we see that for $0\leq k \leq m$ (and indeed, for all $k\leq 0$), we have
\begin{align*}
d_{GS}(g_k\zeta,g_k\xi) &= \int_{-\infty}^{+\infty}e^{-2|t|}d_{\tilde{S}}(g_k\tilde{\xi}(t),g_k\tilde{\zeta}(t))\,dt
\\
&= \int_{-\infty}^{+\infty}e^{-2|t|}d_{\tilde{S}}(\tilde{\xi}(t+k),\tilde{\zeta}(t+k))\,dt
\\
&= \int_{-\infty}^{-(k+\tau)}e^{-2|t|}d_{\tilde{S}}(\tilde{\xi}(t+k),\tilde{\zeta}(t+k))\,dt+\int_{-(k+\tau)}^{+\infty}e^{-2|t|}d_{G\tilde{S}}(\tilde{\xi}(t+k),\tilde{\zeta}(t+k))\,dt
\\
&= \int_{-(k+\tau)}^{+\infty}e^{-2|t|}d_{\tilde{S}}(\tilde{w'}(t+k),\tilde{w}(t+k))\,dt
\\
&\leq \int_{-\infty}^{+\infty}e^{-2|t|}d_{\tilde{S}}(\tilde{w'}(t+k),\tilde{w}(t+k))\,dt
\\
&\leq d_{G\tilde{S}}(g_k\tilde{w'},g_k\tilde{w})
\\
&\leq \eps.
\end{align*}

Therefore, we see that for $-n\leq k \leq m$, we have
$$d_{GS}(g_k\eta,g_k\zeta)\leq d_{GS}(g_k\eta,g_k\xi)+d_{GS}(g_k\xi,g_k\zeta) \leq 2\eps.$$
Consequently, for $s\in [-\frac{1}{n_0},\frac{1}{n_0}]$ and $-n\leq k\leq m$, we have
$$d_{GS}(g_sg_k\eta,g_k\zeta) \leq d_{GS}(g_sg_k\eta,g_k\eta) + d_{GS}(g_k\eta,g_k\zeta) \leq |s| + 2\eps \leq 4\eps.$$
\end{proof}

\begin{lemma}\label{lemma: bracket contains bowen ball}
Suppose $w,z\in g_{\tau}R(\eps)$ for some $\tau\in\left[-\frac{1}{n_0},\frac{1}{n_0}\right]$. Then for all $n,m \geq 0$, we have
$$B_{[-n,m]}(\langle w,z\rangle,\eps) \subset g_{\left[-\eps + \frac{1}{n_0},\eps + \frac{1}{n_0}\right]}\langle B_m^u(w,3\eps),B_n^s(z,3\eps)\rangle.$$
\end{lemma}

\begin{proof}
First, for ease of notation, set $\xi := \langle w,z\rangle$. Now let $x\in B_{[-n,m]}(\xi,\eps)$. We will show that there exists $s\in \left[-\eps + \frac{1}{n_0},\eps + \frac{1}{n_0}\right]$ such that $g_sx\in R(\eps)$. By our choice of $\eps \leq \frac{\delta}{4}$, $x$ shares a geodesic segment around time $-\tau$ with $\xi$ (see the proof of Lemma \ref{lemma: bracket_su geometry}). Consequently, there exists $s\in\mathbb{R}$ such that $x(s) = \xi(-\tau)$, and by assumption on $x$, we know that $s\in [-\tau -\eps,-\tau + \eps]$.

With that $s$ fixed, we will show that $g_{-s}x\in \langle B_m^u(w,3\eps),B_n^s(z,3\eps)\rangle$. A key observation will be that for all $k\in\mathbb{N}$, we have

\begin{align*}
d_{GS}(g_kx,g_k\xi) &= \int_{-\infty}^{+\infty}e^{-2|t|}d_{\tilde{S}}(\tilde{x}(t + k),\tilde{\xi}(t + k))\,dt
\\
&= \int_{-\infty}^{-k}e^{-2|t|}d_{\tilde{S}}(\tilde{x}(t + k),\tilde{z}(t + k))\,dt + \int_{-k}^{+\infty}e^{-2|t|}d_{\tilde{S}}(\tilde{x}(t + k),\tilde{w}(t + k))\,dt.
\end{align*}

Now, set $\eta := \langle g_{-s}x,w\rangle$ and $\zeta := \langle z,g_{-s}x\rangle$. Let $0\leq k \leq m$. Then observe

\begin{align*}
d_{GS}(g_k\eta,g_kw) &= \int_{-\infty}^{+\infty} e^{-2|t|}d_{\tilde{S}}(\tilde{\eta}(t + k), \tilde{w}(t + k))\,dt
\\
&= \int_{-k}^{+\infty} e^{-2|t|}d_{\tilde{S}}(g_{-s}\tilde{x}(t + k),\tilde{w}(t+k))\,dt
\\
&\leq \int_{-k}^{+\infty}e^{-2|t|}s\,dt + \int_{-k}^{+\infty}e^{-2|t|}d_{\tilde{S}}(\tilde{x}(t + k),\tilde{w}(t + k))\,dt
\\
&\leq s + d_{GS}(g_kx,g_k\xi)
\\
&\leq 3\eps.
\end{align*}
A similar argument holds for $-n\leq k\leq 0$.
\end{proof}

We can use our previous lemmas to provide estimates using the Gibbs properties.

\begin{lemma}\label{lemma: upper Gibbs estimate}
Let $x,w\in g_{\tau}R(\eps)$ for some $\tau\in \left[-\frac{1}{n_0},\frac{1}{n_0}\right]$. There exists $Q_1 := Q_1(\eps)$ such that for all $a<b \in \left[-\frac{1}{n_0},\frac{1}{n_0}\right]$ and for all $n,m > 0$,
$$\mu\left(g_{[a,b]}\langle B_m^u(w,\eps),B_n^s(x,\eps)\rangle\right) \leq \frac{n_0(b-a)}{2}Q_1e^{-(n+m)P(\phi) + \int_{-n}^m\phi(g_sw)\,ds}.$$
\end{lemma}

\begin{proof}
By Lemma \ref{lem:ratio of flow boxes}, it suffices to estimate $\mu(g_{\left[-\frac{1}{n_0},\frac{1}{n_0}\right]}\langle B_m^u(w,\eps),B_n^s(x,\eps)\rangle)$. Then, appealing to Lemma \ref{lemma: bowen ball contains bracket} and the upper Gibbs property (Proposition \ref{prop: Upper Gibbs}),
\begin{align*}
\mu\left(g_{[a,b]}\langle B_m^u(w,\eps),B_n^s(x,\eps)\rangle\right) &= \frac{n_0(b-a)}{2}\mu\left(g_{\left[-\frac{1}{n_0},\frac{1}{n_0}\right]}\langle B_m^u(w,\eps),B_n^s(x,\eps)\rangle\right)
\\
&\leq \frac{n_0(b-a)}{2}\mu\left(B_{[-n,m]}(\langle w,x\rangle, 4\eps)\right)
\\
&\leq \frac{n_0(b-a)}{2}Q_1e^{-(n+m)P(\phi) + \int_{-n}^m\phi(g_s\langle w,x\rangle)\,ds}.
\end{align*}
\end{proof}

While the estimate on the upper bound was straightforward, the lower bound is slightly more complicated, due to the fact that the lower Gibbs property of $\mu$ is non-uniform.

\begin{lemma}\label{lemma: lower Gibbs estimate}
Let $x,w\in R(\eps)$. Then for all $a<b\in \left[-\frac{1}{n_0},\frac{1}{n_0}\right]$ and $n,m \geq t_R$ as in Lemma \ref{lem: bracket and flow interaction}, there exists $Q_2 := Q_2(\eps,\lambda(g_{-n}w),\lambda(g_{m}w))$ such that 
$$\mu\left(g_{[a,b]}\langle B_m^u(w,\eps),B_n^s(x,\eps)\rangle\right) \geq \frac{n_0(b-a)}{2(n_0\eps + 1)}Q_2(\eps,\lambda(g_{-n}w),\lambda(g_m w))e^{-(n+m)P(\phi) + \int_{-n}^m\phi(g_sw)\,ds}.$$
\end{lemma}

\begin{proof}
By our choice of $n,m\geq t_R$ and Lemma \ref{lem: bracket and flow interaction} we have $\lambda(g_{-n}\langle w,x\rangle) = \lambda(g_{-n}x)$ and $\lambda(g_m\langle w,x\rangle) = \lambda(g_mw)$. Then, we again use Lemma \ref{lem:ratio of flow boxes}, together with Lemma \ref{lemma: bracket contains bowen ball} and Corollary \ref{cor: unif Lower Gibbs} to obtain:
\begin{align*}
\mu\left(g_{[a,b]}\langle B_m^u(w,\eps),B_n^s(x,\eps)\rangle\right) &= \frac{n_0(b-a)}{2(n_0\eps + 1)}\mu\left(g_{\left[-\left(\eps + \frac{1}{n_0}\right),\eps + \frac{1}{n_0}\right]}\langle B_m^u(w,\eps),B_n^s(x,\eps)\rangle\right)
\\
&\geq \frac{n_0(b-a)}{2(n_0\eps + 1)}\mu\left(B_{[-n,m]}\left(\langle w,x\rangle, \frac{\eps}{3}\right)\right)
\\
&\geq \frac{n_0(b-a)}{2(n_0\eps + 1)}Q_2(\eps,\lambda(g_{-n}w),\lambda(g_m w))e^{-(n+m)P(\phi) + \int_{-n}^m\phi(g_sw)\,ds}.
\end{align*}
\end{proof}

 We now prove a few more helpful lemmas.

\begin{definition}\label{defn:pi maps}
Given $y\in \mathcal{B}(n_0,\eps)$ let $\pi_y: \mathcal{B}(n_0,\eps) \to V_y^u$ be defined by $\pi_y(z) = [z,y]$. For any $x \in \mathcal{B}(n_0,\eps)$, we define $\pi_{x,y}: V^u_x \to V^u_y$ be the restriction of $\pi_y$ to $V^u_x.$
\end{definition}

\begin{lemma}\label{lem: continuity of bracket}
The function $\pi_{x,y}$ is Lipschitz continuous with constant $e^{\frac{4}{n_0}}$. Furthermore, there exists $|\rho|\leq \frac{2}{n_0}$ such that $\pi_{x,y}(z) := \langle g_{\rho}z,y\rangle$ for all $z\in V^u_x$, and $\pi_{y,x}(z) := \langle g_{-\rho}z,x\rangle.$
\end{lemma}

\begin{proof}

We begin by showing the second part of this lemma. First, we know that there exists $r,s\in \left[-\frac{1}{n_0},\frac{1}{n_0}\right]$ such that $g_rx,g_sy\in R(\eps)$. Consequently, $\langle g_rx,g_sy\rangle$ is defined and we have that $\langle g_rx,g_sy\rangle (\pm\infty) = [x,y](\pm\infty)$. As $[x,y]$ consists of only one geodesic (because it contains a regular geodesic), it follows that there exists $\tau$ such that $g_{\tau}\langle g_rx,g_sy\rangle = [x,y]$. Because we know that $[x,y]\in W^u(y,d_{GS}(x,y))$ and $\langle g_rx,g_sy\rangle\in W^u(g_sy,d_{GS}(g_sy,\langle g_rx,g_sy\rangle)$, by Corollary \ref{corollary: busemann under flow}, it follows that $\tau = -s$, and so we have that $[x,y] = g_{-s}\langle g_rx,g_sy\rangle$. Now, it follows that $[x,y] = \langle g_{r-s}x,y\rangle$ by Corollary \ref{corollary: busemann under flow}, and the fact that $d_{GS}(\langle g_{r-s}x,y\rangle, y) \leq d_{GS}(g_{r-s}x,y)$.  Thus, taking $\rho = r - s$, we have $|\rho| \leq \frac{2}{n_0}$, and complete the proof of the second part of the lemma.

For the first, observe that
	\begin{align*}
		d_{GS}([w,y],[z,y]) &= d_{GS}(\langle g_\rho w,y\rangle, \langle g_\rho z,y\rangle)
		\\
		&\leq d_{GS}(g_\rho w,g_\rho z)
		\\
		&\leq e^{2|\rho|}d_{GS}(w,z).
	\end{align*}
The first inequality follows from the description of how $\langle-,-\rangle$ works within $\mathcal{B}(n_0,\eps)$ given by Lemma \ref{lemma: bracket_su geometry}. Specifically, $\langle g_\rho w,y\rangle(r)=\langle g_\rho w,y\rangle(r)=y(r)$ for all $r\leq 0$ but for $r\geq 0$, $\langle g_\rho w,y\rangle(r) = g_\rho w(r)$ and $\langle g_\rho z,y\rangle(r) = g_\rho z(r)$.
\end{proof}

\begin{corollary}\label{cor: continuity of bracket}
Given $x\in \mathcal{B}(n_0,\eps)$, $\pi_x$ is Lipschitz continuous, with constant $e^{\frac{2}{n_0}}$.
\end{corollary}

\begin{proof}
Let $y,z\in \mathcal{B}(n_0,\eps)$. Then, since $[y,x] = [[y,z],x]$ and $d_{GS}(z,[y,z]) \leq d_{GS}(y,z)$, we have that
$$d_{GS}([y,x],[z,x]) = d_{GS}(\pi_{z,x}([y,z]),\pi_{z,x}(z)) \leq e^{\frac{2}{n_0}}d_{GS}([y,z],z) \leq e^{\frac{4}{n_0}}d_{GS}(y,z).$$
\end{proof}

\subsubsection{Proof of the local product structure}\label{sec:lps}

In this section, we continue the same choices of $R(\eps)$ and $\mathcal{B}(n_0,\eps)$. The local product structure we want from $\mu$ is the following.

\begin{definition}\label{defn:lps}
A measure $\mu$ has \emph{local product structure} on $\mathcal{B}(n,\eps)$ if there exist $x\in \mathcal{B}(n,\eps)$ and measures $\nu^{cs,u}$ on $V_x^{cs,u}$ such that $\mu|_{\mathcal{B}(n,\eps)} \ll \nu^u \otimes \nu^{cs}$ where $\nu^u \otimes \nu^{cs}$ is the measure on $\mathcal{B}(n,\eps)$ given by pushing forward the product measure $\nu^u \times \nu^{cs}$ under $[-,-] \colon V^u_x \times V^{cs}_x \to \mathcal{B}(n,\eps)$.
\end{definition}

The measures $\nu^{cs,u}$ will be conditional measures induced by $\mu$ on $V_x^{cs,u}$, and proving this can be reduced to understanding how conditional measures on the local unstable leaves behave under the holonomy map, which maps from one unstable leaf to another along the center-stable leaves. The key result of this section is Proposition \ref{prop: main prop}. From this Proposition, we can easily prove Proposition \ref{prop: LPS} and Corollary \ref{cor: ae lps}. The local product structure is deduced from these in a straightforward fashion -- see, e.g., \S3 in \cite{ClimenhagaGibbs}.

We begin by showing how to obtain a compact subset of $V_x^u$ which satisfies the necessary properties to apply the results of \S\ref{sec: Partitions} while still allowing us to carry out the necessary estimates later.

For future ease of notation, let $$G_{\mu} := \{x\in \GS\mid g_{t_k}x\in B(z_0,\kappa) \text{ for a sequence } \{t_k\}_{k\in\mathbb{Z}} \text{ which tends to }\pm\infty \text{ as }k\to\pm\infty\},$$
and
$$G_{\mu}^{\pm} := \{x\in \GS\mid g_{t_k}x\in B(z_0,2\kappa) \text{ for a sequence } \{t_k\}_{k\in\mathbb{N}} \text{ which tends to }\pm\infty \text{ as }k\to\pm\infty\}.$$
As we will make use of various properties of these sets, we outline them below.

\begin{lemma}
For any $x\in \mathcal{B}(n_0,\eps)$, if $x\in G_{\mu}$, then $V^{cs}_x\subset G_{\mu}^+$ and $V^{cu}_x\subset G_{\mu}^{-}$. Additionally, $\mu(G_{\mu}) = 1$.
\end{lemma}

\begin{proof}
First, suppose $x\in G_{\mu}$. Then there exists $T$ such that for all $t\geq T$, $\operatorname{diam}(g_tV^s_x) \leq \kappa$. Therefore, if $g_{t_k}x\in B(z_0,\kappa)$ for some $t_k\geq T$, then $g_{t_k}V^s_x\subset B(z_0,2\kappa)$. As $G_{\mu}^{\pm}$ are flow-invariant, it follows that $V^{cs}_x\subset G_{\mu}^{+}$. A similar proof shows the desired result for $V^{cu}_x$. That $\mu(G_{\mu}) = 1$ follows from Poincar\'e recurrence and ergodicity of $\mu$, as well as the fact that $\mu(B(z_0,\kappa)) > 0$.
\end{proof}

\begin{lemma}\label{lem: generic times}
Let $c\in \left(0,\frac{\|\lambda\|_{\infty}}{4}\right)$, and recall that we've chosen $z_0\in\GS$ such that $\lambda(z_0) > 3c$. Letting $\kappa = \kappa(z_0)$ be as in Section~\ref{subsec: relations bowen} (see the first paragraph), for $z\in G_{\mu}$, there exists an increasing sequence $\{t_k\}_{k\in\mathbb{Z}}$ which tends to $\pm\infty$ as $k\to\pm\infty$, respectively, such that $\lambda(g_{t_k+s}z)\geq 2c$ for all $|s|\leq \kappa$
\end{lemma}

\begin{proof}
This follows directly from the fact that $\kappa$ is small enough so that $\lambda|_{B(z_0,3\kappa)} \geq \frac{2}{3}\lambda(z_0) > 2c$.
\end{proof}

\begin{proposition}\label{prop: Creation of compact}
Let $c \in \left(0,\frac{\|\lambda\|_{\infty}}{4}\right)$ and recall that we've chosen $z_0\in\GS$ such that $\lambda(z_0) > 3c$. Let $x\in\mathcal{B}(n_0,\eps)$. Then for all $|a|,|b| \leq \frac{1}{n_0}$ and $t\geq 0$, and for all $\bar\delta > 0$, we can find a compact subset $K\subset V^u_x\cap \operatorname{int}\mathcal{B}(n_0,\eps)$ such that:
 \begin{itemize} 
 \item[(a)] Let $\kappa = \kappa(z_0)$ be as in Section~\ref{subsec: relations bowen} (see the first paragraph). Then, for all $z\in K$, there exists a sequence $t_{k}\to\infty$ such that $\lambda(g_{t_k+s}z)\geq 2c$ for all $|s|\leq \kappa$;
 \item[(b)] $\mu(g_{[a,b]}\langle K, B_{t}^s(x,\eps)\rangle)\geq (1-\bar\delta)\mu(g_{[a,b]}\langle V_x^u, B_{t}^s(x,\eps)\rangle)$.
\end{itemize}
\end{proposition}
\begin{proof}

Let $n\geq n_0$ and $\bar\delta > 0$ be arbitrary. By inner regularity of $\mu$, let $\mathcal{K}\subset \operatorname{int}(\mathcal{B}(n_0,\eps))\cap G_\mu\cap g_{[a,b]}\langle V_x^u, B_t^s(x,\eps)\rangle$ be a compact subset of relative measure at least $1-\bar\delta$; that is,
$$\mu(\mathcal{K}) \geq (1-\bar\delta)\mu(G_\mu\cap g_{[a,b]}\langle V_x^u, B_t^s(x,\eps)\rangle) = (1-\bar\delta)\mu(g_{[a,b]}\langle V_x^u, B_t^s(x,\eps)\rangle).$$
Then, define $K_x := [\mathcal{K},x]$ to be the projection of $\mathcal{K}$ to $V_x^u$ under the bracket operation. As $\pi_x : \mathcal{B}(n_0,\eps)\to V_x^u$ is continuous by Corollary \ref{cor: continuity of bracket}, $K_x$ is compact. Now, $g_{\left[\frac{-1}{n},\frac{1}{n}\right]}\langle K_x,B_n^s(x,\eps)\rangle \supset \mathcal{K}$, and so has relative measure at least $1-\bar\delta$. Furthermore, it preserves the necessary recurrence properties. This is because any element $w\in g_{\left[\frac{-1}{n},\frac{1}{n}\right]}\langle K_x,B_n^s(x,\eps)\rangle$ is in $W^{cs}(z)$ for some $z\in G_\mu$. Therefore, using Lemma \ref{lem: generic times}, there exists some $r\in\mathbb{R}$ such that for all large enough $t$ and for all $|s|\leq \kappa$, $d(g_{r+t + s}w,g_tz) \leq 2\kappa$, and consequently, $\lambda(g_{r+t_{z,k} + s}w) \geq 2c$ for the tail of the sequence $t_{z,k}\to\infty$.
\end{proof}

The local unstable leaves $V_x^u$ form a measurable partition of $\mathcal{B}(n_0,\eps)$ as discussed in \cite{ClimenhagaGibbs} or \cite{CPZ}. (The proof of Proposition \ref{prop: integral_limit} below justifies the measurability.) Hence (see, e.g., \cite[Lemma 3.1]{ClimenhagaGibbs}), there is a unique system of conditional measures $\mu_x^u$ on the leaves $V_x^u$. The space of leaves $\{V_x^u\}_{x\in\mathcal{B}(n_0,\eps)}$ can be identified, using the (topological) local product structure, with $V^{cs}_x$ for any $x\in\mathcal{B}(n_0,\eps)$. For such an $x$ we define $\mu^{cs}_x$ to be the factor measure: $\mu^{cs}_x(E):=\mu\left(\bigcup_{y\in E}V^u_y\right)$.

\begin{remark*}
$(V^{cs}_x,\mu^{cs}_x)$ is (after renormalization) a Lebesgue space, as $V^{cs}_x$ is a complete, separable metric space. We note that from the definition of $\mu^{cs}_x$ and the flow-invariance of $\mu$, $\mu^{cs}_x$ has no atoms. Indeed, if $\mu^{cs}_x(\{y^*\})>0$ for some $y^*\in V^{cs}_x$, then for all small $t$,  $\mu(g_t V^u_{y^*}) = \mu(V^u_{y^*})=\mu^{cs}_x(\{y^*\})>0$. But for sufficiently small $t, s$, $g_t V^u_{y^*}$ and $g_s V^u_{y^*}$ are disjoint. Therefore, for small $\delta>0$, $\mu\left(\bigcup_{t\in(-\delta,\delta)} g_t V^u_{y^*}\right) = \infty$, a contradiction.
\end{remark*}

\begin{proposition}\label{prop: integral_limit}
Let $c\in \left(0,\frac{\lVert\lambda\rVert_{\infty}}{4}\right)$ and let $x_0\in R(\eps)\cap G_{\mu}$. There exists $R'\subset \mathcal B(n_0,\eps)$ with $\mu(B(n_0,\eps)\setminus R') = 0$ such that for all $x,y\in R'$, there exists a sequence of partitions of $\mathcal{B}(n_0,\eps)$, denoted $\{\xi_n\}_{n\in\mathbb{N}}$, such that for every continuous $\psi:\mathcal{B}(n_0,\eps) \to \mathbb{R}$,
$$\int_{V^u_x} \psi\,d\mu^u_x = \lim_{n\to\infty}\frac{1}{\mu(\xi_n(x))}\int_{\xi_n(x)}\psi\,d\mu,$$
and similarly for $y$. Furthermore, for all sufficiently large $n$, we have that
$$g_{\left(\frac{k}{2^n},\frac{k+1}{2^n}\right)}\left\langle V^u_{x_0},B^s_t\left(z,\frac{\eps}{2}\right)\right\rangle\subset \xi_n(x) \subset g_{\left(\frac{k}{2^n},\frac{k+1}{2^n}\right)}\langle V^u_{x_0},B^s_t(z,\eps)\rangle$$
for some $z\in V^s_{x_0}$, $\frac{-2^n}{n_0}\leq k \leq \frac{2^n}{n_0}-1$, and $t\geq 0$ with $\lambda(g_{-t}z)\geq c$, and similarly for $\xi_n(y)$.
\end{proposition}

\begin{proof}
Let $n\in\mathbb{N}$ be arbitrary. For convenience, define a probability measure $\nu$ on $V^s_{x_0}$ by $\nu(A) = \mu|_{\mathcal{B}(n_0,\eps)}\left(\{V^{cu}_z\mid z\in A\}\right)$. Observe that $\nu(G_{\mu}) = 1$, and so, $\nu(V^s_{x_0}\cap G_{\mu}^{-}) = 1$. Now, take $K_n\subset V^s_{x_0}\cap G_{\mu}^{-}\cap \operatorname{int}(\mathcal{B}(n_0,\eps))$ to be compact with $\nu(K_n)\geq 1-\frac{1}{n^3}$. This is possible because the boundary of the flow box has zero $\mu$-measure by Lemma \ref{lem:zero msr bdry}. Using Lemma \ref{lem: generic times}, Proposition \ref{prop: refining sequence} gives us a refining sequence of partitions of $V^s_{x_0}$, denoted $\{\zeta^n_i\}_{i\in\mathbb{N}}$, for which there exists a set $E^n\subset V^s_{x_0}$ with $\nu(V^s_{x_0}\setminus E^n) < \frac{1}{n^2}$ and for all $x\in E^n$ and all $i\in\mathbb{N}$, $\zeta_i(x)$ has the adapted-like properties outlined in Proposition \ref{prop: refining sequence}. Since $\sum_{n=1}^{\infty}\nu(V^s_{x_0}\setminus E^n) < \infty$, by Borel-Cantelli, we know that $\nu(\bigcap_{n=1}^{\infty}\bigcup_{k=n}^{\infty}V^s_{x_0}\setminus E^k) = 0$. Therefore, for a full $\nu$-measure subset $E'\subset V^s_{x_0}$, we have that every $x\in E'$ belongs to $E^n$ for all but finitely many $n\in\mathbb{N}$. Consequently, every $x,y\in E'$ both belong to $E^n$ for infinitely many $n\in\mathbb{N}$.

Using $\{\zeta^n_i\}_{i\in\mathbb{N}}$, we will now define a sequence of partitions of $\mathcal{B}(n_0,\eps)$. For each $n,i\in\mathbb{N}$, define a new partition by
$$\xi^n_i := \left\{g_{\left(\frac{m}{2^i},\frac{m+1}{2^i}\right)}\langle V^u_{x_0},\zeta^n_i\rangle \mid \frac{-2^i}{n_0} \leq m \leq \frac{2^i}{n_0} - 1\right\}.$$
This is a partition (up to measure zero) of $\mathcal{B}(n_0,\eps)$.
For each $n\in\mathbb{N}$, $\bigvee_{i=1}^{\infty}\xi^n_i = \{V^u_z\mid z\in \mathcal{B}(n_0,\eps)\}$. Therefore, by \cite[Proposition 8.2]{CPZ}, for all $n\in\mathbb{N}$, there exists a full measure set $R^n\subset \mathcal{B}(n_0,\eps)$ for which for all $x\in R^n$,
$$\int_{V^u_x}\psi\,d\mu^u_x = \lim\limits_{i\to\infty}\frac{1}{\mu(\xi^n_i(x))}\int_{\xi^n_i(x)}\psi\,d\mu$$
for all continuous $\psi : \mathcal{B}(n_0,\eps)\to \mathbb{R}$. Then we have that $\bigcap_{n=1}^{\infty} R^n$ is a full measure set for which this formula holds for the sequence of partitions $\{\xi^n_i\}_{i\in\mathbb{N}}$ for all $n\in\mathbb{N}$. Recalling our definition of $E'$, we define
$$R' := \bigcap_{n=1}^{\infty} R^n\cap g_{\left[-\frac{1}{n_0},\frac{1}{n_0}\right]}\langle V^u_{x_0}, E'\rangle\setminus \bigsqcup_{m=-2^n(n_0)^{-1}}^{2^n(n_0)^{-1}}g_{\frac{m}{2^n}}R(\eps),$$
which is a full measure set by construction.

Now, let $\gamma_1,\gamma_2\in R'$ be arbitrary. Then $\gamma_1 = g_{r_1}\langle x_1,y_1 \rangle$ and $\gamma_2 = g_{r_2}\langle x_2,y_2\rangle$ for $y_1,y_2\in E'$, $x_1,x_2\in V^u_{x_0}$, and $r_1,r_2\neq \frac{k}{2^n}$ for any $k\in\mathbb{N}$. Then, by our choice of $E'$, there exists some $n^*$ for which $\gamma_1,\gamma_2\in E^{n^*}$. Therefore, we see that for all $i\in\mathbb{N}$, we have that there exists $t_1,t_2\geq 0$ and $z_1,z_2\in V^s_{x_0}$ such that $\lambda(g_{-t_1}z_1)\geq c$, $\lambda(g_{-t_2}z_2)\geq c$, and
$$B^s_{t_1}\left(z_1,\frac{\eps}{2}\right)\subset \zeta^{n^*}_i(y_1)\subset B^s_{t_1}\left(z_1,\frac{\eps}{2}\right),$$
and similarly for $\zeta^{n^*}_i(y_2)$. Taking our sequence of partitions to be $\{\xi^{n^*}_i\}_{i\in\mathbb{N}}$ provides our desired properties.
\end{proof}

\begin{proposition}\label{prop: main prop}
There exists $K > 0$ such that for almost every $x,y\in \mathcal B(n_0,\eps)$ and for every uniformly continuous $\psi\colon \mathcal{B}(n_0,\eps) \to (0,\infty)$,
$$\int_{V^u_y}\psi\,d(\pi_{x,y})_*\mu_x^u \leq K\int_{V^u_y}\psi\,d\mu_y^u.$$
\end{proposition}

\begin{proof}
Let $\psi : \mathcal{B}(n_0,\eps)\to (0,\infty)$ be uniformly continuous, and let $\alpha > 0$ be small enough so that if $d_{GS}(u,u') \leq 3\alpha$, then both $\frac{1}{2}\psi(u') \leq \psi(u) \leq 2\psi(u')$, and $\frac{1}{2}\psi(\pi_yu') \leq \psi(\pi_y u) \leq 2\psi(\pi_y u')$. Now let $R'\subset \mathcal{B}(n_0,\eps)$ be as in Proposition \ref{prop: integral_limit}, and let $x,y\in R'\cap g_{\left(-\frac{1}{n_0},\frac{1}{n_0}\right)}R(\eps)\cap G_{\mu}$. Observe that this is a full measure set, from applying Proposition \ref{prop: integral_limit} and Lemma~\ref{lem:zero msr bdry}, as we chose $n_0$ so that $\frac{1}{n_0}$ is less than one quarter the injectivity radius of $S$. Now, both $\psi\circ\pi_{x,y}$ and $\psi$ are continuous functions, and so we can apply Proposition \ref{prop: integral_limit} to estimate both
	$$\int_{V^u_x}\psi\circ\pi_{x,y}\,d\mu_x^u = \int_{V^u_y}\psi\,d(\pi_{x,y})_*\mu_x^u \quad \text{ and } \quad \int_{V^u_y}\psi\,d\mu_y^u.$$
	To do so, fix $x_0\in R(\eps)$ and $c\in \left(0,\frac{\lVert\lambda\rVert_{\infty}}{4}\right)$, and let $\{\xi_n\}_{n\in\mathbb{N}}$ be the sequence of partitions given by Proposition \ref{prop: integral_limit}. We will now proceed with our proof by using Propositions \ref{prop: Creation of compact} and \ref{prop: partition construction} to partition most of $V_{x_0}^u$ (and then, in turn $V_x^u$ and $V_y^u$) into well-understood sets, for which we have nice upper and lower bounds from the Gibbs property.

Let $0 < \beta < \min\{1,\lVert\psi\rVert^{-1}\}$. For each $n\in\mathbb{N}$, recall that by Proposition \ref{prop: integral_limit}, there exists $k_{x,n},k_{y,n}\in \left[-\frac{2^n}{n_0},\frac{2^n}{n_0} - 1\right]$, $t_{x,n},t_{y,n}\geq 0$ and $z_{x,n},z_{y,n}\in V^s_{x_0}$ such that
	$$g_{\left(\frac{k_{x,n}}{2^n},\frac{k_{x,n}+1}{2^n}\right)}\langle V^u_{x_0},B^s_{t_{x,n}}(z_{x,n},\frac{\eps}{2})\rangle\subset \xi_n(x) \subset g_{\left(\frac{k_{x,n}}{2^n},\frac{k_{x,n}+1}{2^n}\right)}\langle V^u_{x_0},B^s_{t_{x,n}}(z_{x,n},\eps)\rangle,$$
	and similarly for $y$. Going forward, we will often omit the dependence on $n$ from the notation since we will primarily work with a fixed choice of $n$. Now, let $K_n\subset V_{x_0}^u$ be as in Proposition \ref{prop: Creation of compact} for $c$ with
	$$\mu\left(g_{\left(\frac{k_{x,n}}{2^n},\frac{k_{x,n} + 1}{2^n}\right)}\langle K_n,B^s_{t_{x,n}}(z_{x,n},\eps)\rangle\right)\geq (1-\beta^2\mu(\xi_n(x))\mu\left(g_{\left(\frac{k_{x,n}}{2^n},\frac{k_{x,n} + 1}{2^n}\right)}\langle V_{x_0}^u,B^s_{t_{x,n}}(z_{x,n},\eps)\rangle\right),$$
	and similarly replacing every $x$ with $y$. Then, use Proposition \ref{prop: partition construction} and Corollary \ref{cor: fill out partition} to define a partition $\omega_n := \{A_0,A_1,\cdots, A_m\}$ of $V^u_{x_0}$, such that:
	\begin{itemize}
		\item For each $A_i\in \omega_n$, $i\neq 0$, there exists $w_i\in A_i$ and $t_i > t_R + \frac{2}{n_0}$ such that $B_{t_i}^u(w_i,\frac{\eps}{2}) \subset A_i\subset B_{t_i}^u(w_i,2\eps)$
		\item For $i\neq 0$, $\lambda(g_{t_i + r} w_i) \geq c$ for all $r\in \left[-\frac{4}{n_0},\frac{4}{n_0}\right]$
		\item For $i\neq 0$, $\max\{\diam A_i \mid A_i\in\omega_n\} \leq \min\{\alpha,d_{\GS}(K_n,\partial \mathcal{B}(n_0,\eps))\}$
		\item $A_0\subset V^u_{x_0}\setminus K_n$
	\end{itemize}
	The second condition follows from choosing $n_0$ large enough so that $\frac{4}{n_0} \leq \kappa$ in the statement of Proposition \ref{prop: partition construction}, which in turn relies only on the choice of $\kappa$ in Proposition \ref{prop: Creation of compact}. The condition on the diameter of $\omega_n$ follows from Corollary \ref{cor: small diameter partition}, as $K_n\subset \Reg$, while $\NE(\eps)\subset \Sing$ due to \cite[Lemma 2.16]{CCESW}.

Take $n$ large enough so that $\frac{1}{2^n} \leq \frac{\alpha}{2}$ and so that both $\diam B^s_{t_{n,x}}(z_{n,x},\eps) \leq \alpha$ and $\diam B_{t_{n,y}}^s(z_{n,y},\eps) \leq \alpha$. Then for $A_i\in \omega_{n}$ with $i\neq 0$, we have
$$\diam g_{\left(\frac{k_{n,x}}{2^n},\frac{k_{n,x} + 1}{2^n}\right)}\langle A_i,B_{t_{n,x}}^s(z_{n,x},\eps)\rangle \leq \frac{2}{2^n} + \diam A_i + \diam B_{t_{n,z}}^s(z_{n,x},\eps) \leq 3\alpha,$$
and similarly replacing $x$ with $y$. Now, we will find an upper bound on $\int_{V_x^u}\psi\circ\pi_{x,y}\,d\mu_{x}^u$. Using the fact that $\psi(\pi_y z) = \psi(\pi_y g_tz)$ for $z\in R(\eps)$ and $|t|\leq \frac{1}{n_0}$, it follows that for all such $n$, we have
 	\begin{align*}
		\int_{\xi_n(x)}\psi\circ \pi_y\,d\mu &\leq \int_{g_{\left(\frac{k_x}{2^n},\frac{k_x+1}{2^n}\right)}\langle V^u_{x_0},B_{t_x}^s(z_x,\eps)\rangle}\psi\circ\pi_y\,d\mu
		\\
		&= \sum_{i=0}^{m}\int_{g_{\left(\frac{k_x}{2^n},\frac{k_x+1}{2^n}\right)}\langle A_i,B_{t_x}^s(z_x,\eps)\rangle}\psi\circ\pi_y\,d\mu
		\\
		&\leq \beta^2\lVert\psi\circ\pi_y\rVert_{\infty}{\mu(\xi_n(x))} + \sum_{i=1}^{m}\int_{g_{\left(\frac{k_x}{2^n},\frac{k_x+1}{2^n}\right)}\langle A_i,B_{t_x}^s(z_x,\eps)\rangle}\psi\circ\pi_y\,d\mu
		\\
		&\leq \beta{\mu(\xi_n(x))} +  \sum_{i=1}^{m}\int_{g_{\left(\frac{k_x}{2^n},\frac{k_x+1}{2^n}\right)}\langle A_i,B_{t_x}^s(z_x,\eps)\rangle}2\psi(\pi_y\langle w_i,z_x\rangle)\,d\mu
		\\
		&= \beta{\mu(\xi_n(x))} + \sum_{i=1}^{m}2\psi(\pi_y\langle w_i,z_x\rangle)\mu\left(g_{\left(\frac{k_x}{2^n},\frac{k_x+1}{2^n}\right)}\langle A_i,B_{t_x}^s(z_x,\eps)\rangle\right)
		\\
		&\leq \beta{\mu(\xi_n(x))} + \sum_{i=1}^{m}2\psi(\pi_y\langle w_i,z_x\rangle)\mu\left(g_{\left(\frac{k_x}{2^n},\frac{k_x+1}{2^n}\right)}\langle B_{t_i}^u(w_i,2\eps),B_{t_x}^s(z_x,\eps)\rangle\right)
		\\
		&\leq \beta{\mu(\xi_n(x))} + \sum_{i=1}^{m}2\psi(\pi_y\langle w_i,z_x\rangle)Q_1n_02^{-(n+1)}e^{-(t_i+t_x)P(\phi) + \int_{-t_x}^{t_i}\phi(g_r\langle w_i,z_x\rangle)\,dr}.
	\end{align*}
The final step uses Lemma \ref{lemma: upper Gibbs estimate}. We will also need the following lower bound:
\begin{align*}
\mu(\xi_n(x))&\geq \mu\left(g_{\left(\frac{k_{x}}{2^n}, \frac{k_{x}+1}{2^n}\right)}\left\langle V_{x_0}^u, B^s_{t_{x}}\left(z_{x},\frac{\varepsilon}{2}\right)\right\rangle\right)\\&\geq\sum\limits_{i=1}^m\mu\left(g_{\left(\frac{k_{x}}{2^n}, \frac{k_{x}+1}{2^n}\right)}\left\langle A_i, B^s_{t_{x}}\left(z_{x},\frac{\varepsilon}{2}\right)\right\rangle\right)\\&\geq \sum\limits_{i=1}^m\mu\left(g_{\left(\frac{k_{x}}{2^n}, \frac{k_{x}+1}{2^n}\right)}\left\langle B_{t_i}^u\left(w_i,\frac{\varepsilon}{2}\right), B^s_{t_{x}}\left(z_{x},\frac{\varepsilon}{2}\right)\right\rangle\right).
\end{align*}

From this, using Lemma \ref{lemma: lower Gibbs estimate}, it follows that

$$\mu\left(\xi_n(x)\right) \geq \frac{n_02^{-n}}{n_0\eps+2}\sum\limits_{i=1}^m Q_2\left(\frac{\eps}{2},\lambda(g_{-t_{x}}z_x),\lambda(g_{t_i} w_i)\right)e^{-(t_i+t_x)P(\phi) + \int_{-t_x}^{t_i}\phi(g_s\langle w_i,z_x\rangle)\,ds}.$$

	Finally, we will need the following computations using the Bowen property to make use of these results. Recall that $\phi$ has the Bowen property at scale $\eps_0$, and that $\diam \mathcal{B}(n_0,\eps) < \eps_0$. Now, observe that $d_{GS}(g_r\langle w,z\rangle, g_r z)$ decreases monotonically to $0$ as $r\to-\infty$ for all $w,z\in R(\eps)$, while $d_{GS}(g_r\langle w,z\rangle, g_rw)$ decreases monotonically to $0$ as $r\to\infty$. Therefore, we have $\langle w,z\rangle \in B_{[-n,0]}(z,\eps_0)$ and $\langle w,z\rangle \in B_{[0,n]}(w,\eps_0)$ for all $n\geq 0$. Applying the Bowen property via Proposition \ref{proposition: global Bowen property}, there exists $L > 0$ such that for all $n_1,n_2 \geq 0$,
$$\left|\int_{-n_1}^{n_2}\phi(g_r\langle w,z\rangle)\,dr - \int_{-n_1}^{0}\phi(g_rz)\,dr - \int_0^{n_2}\phi(g_rw)\,dr\right| \leq 2L.$$

Therefore, recalling that $\lambda(g_{-t_x}z_x)\geq c$ by construction and writing $Q_2 := Q_2(\frac{\eps}{2},c,c)$, we have that: 

\begin{align*}
		\frac{\int_{\xi_n(x)}\psi\circ \pi_y\,d\mu}{\mu(\xi_n(x))} &\leq \beta+\frac{\sum_{i=1}^{m}2\psi(\pi_y\langle w_i,z_x\rangle)Q_1n_02^{-(n+1)}e^{-(t_i+t_x)P(\phi) + \int_{-t_x}^{t_i}\phi(g_r\langle w_i,z_x\rangle)\,dr}}{\frac{n_02^{-n}}{n_0\eps+2}\sum\limits_{i=1}^m Q_2\left(\frac{\eps}{2},\lambda(g_{-t_{x}}z_x),\lambda(g_{t_i} w_i)\right)e^{-(t_i+t_x)P(\phi) + \int_{-t_x}^{t_i}\phi(g_s\langle w_i,z_x\rangle)\,ds}}
  \\
  &\leq \beta+\frac{Q_1(n_0\eps+2)}{Q_2}\cdot\frac{e^{2L+\int_{-t_x}^0\phi(g_rz_x)\,dr}\sum\limits_{i=1}^m \psi(\pi_y\langle w_i,z_x\rangle)e^{-t_iP(\phi)+\int_0^{t_i}\phi(g_rw_i)\,dr}}{e^{-2L+\int_{-t_x}^0\phi(g_rz_x)\,dr}\sum\limits_{i=1}^m e^{-t_iP(\phi)+\int_0^{t_i}\phi(g_s w_i)\,ds}}
		\\
  &\leq \beta +\frac{Q_1e^{4L}(n_0\eps+2)}{Q_2}\cdot\frac{\sum\limits_{i=1}^m \psi(\pi_y \langle w_i,z_x\rangle)e^{-t_iP(\phi)+\int_0^{t_i}\phi(g_r w_i)\,dr}}{\sum\limits_{i=1}^m e^{-t_iP(\phi)+\int_0^{t_i}\phi(g_s w_i)\,ds}}.
	\end{align*}
Note that $\beta$ does not depend on our choice of $n$, although the rest of the computation does, as our collection of $z_i$ and $t_i$ depends on our choice of $K_n$.

Now, we will carry out similar computations to find a lower bound for the corresponding estimates of $\int_{V^u_y}\psi\,d\mu_y^u$. These mirror those carried out above, except using the lower Gibbs bound instead of the upper Gibbs bound and vice versa.

	Recalling that we work with $n$ large enough so that $\diam g_{\left[\frac{k_y}{2^n},\frac{k_y+1}{2^n}\right]}\langle A_i,B_{t_y}^s(z_y,\eps)\rangle \leq 3\alpha$, and that $t_y$ is chosen so that $\lambda(g_{-t_y}z_y) \geq c$, we obtain a lower bound as follows. First, observe that $V^u_y\subset g_{\left(\frac{k_y}{2^n},\frac{k_y+1}{2^n}\right)}\langle V^u_{x_0},B^s_{t_y}(z_y,\eps)\rangle$. Therefore, for every $w_i\in V^u_{x_0}$ we have
 $\pi_y\langle w,z_y\rangle \in g_{\left(\frac{k_y}{2^n},\frac{k_y+1}{2^n}\right)}\langle V^u_{x_0},B^s_{t_y}(z_y,\eps)\rangle$, and consequently,
 $$\psi|_{g_{\left(\frac{k_y}{2^n},\frac{k_y+1}{2^n}\right)}\langle V^u_{x_0},B^s_{t_y}(z_y,\eps)\rangle} \geq \frac{\psi(\pi_y\langle w_i,z_y\rangle)}{2}.$$
 Using this, we see
\begin{align*}
\int_{\xi_n(y)}\psi\,d\mu &\geq \int_{g_{\left(\frac{k_y}{2^n},\frac{k_y+1}{2^n}\right)}\left\langle V^u_{x_0},B_{t_y}^s\left(z_y,\frac{\eps}{2}\right)\right\rangle}\psi\,d\mu
\\
&= \sum_{i=0}^m \int_{g_{\left(\frac{k_y}{2^n},\frac{k_y+1}{2^n}\right)}\left\langle A_i,B_{t_y}^s\left(z_y,\frac{\eps}{2}\right)\right\rangle}\psi\,d\mu
\\
&\geq \sum_{i=1}^{m}\int_{g_{\left(\frac{k_y}{2^n},\frac{k_y+1}{2^n}\right)}\left\langle B_{t_i}^u\left(w_i,\frac{\eps}{2}\right),B_{t_y}^s\left(z_y,\frac{\eps}{2}\right)\right\rangle}\psi\,d\mu
\\
&\geq \sum_{i=1}^{m}\int_{g_{\left(\frac{k_y}{2^n},\frac{k_y+1}{2^n}\right)}\left\langle B_{t_i}^u\left(w_i,\frac{\eps}{2}\right),B_{t_y}^s\left(z_y,\frac{\eps}{2}\right)\right\rangle}\frac{\psi(\pi_y\langle w_i,z_y\rangle)}{2}\,d\mu
\\
&= \sum_{i=1}^{m}\frac{\psi(\pi_y\langle w_i,z_y\rangle)}{2}\mu\left(g_{\left(\frac{k_y}{2^n},\frac{k_y+1}{2^n}\right)}\left\langle B_{t_i}^u\left(w_i,\frac{\eps}{2}\right),B_{t_y}^s\left(z_y,\frac{\eps}{2}\right)\right\rangle\right)
\\
&\geq \frac{n_02^{-n}}{2\left(n_0\frac{\eps}{2} + 1\right)}\sum_{i=1}^mQ_2\left(\frac{\eps}{2},\lambda(g_{-t_y}z_y),\lambda(g_{t_i}w_i)\right)e^{-(t_i+t_y)P(\phi) + \int_{-t_y}^{t_i}\phi(g_r\langle w_i,z_y\rangle)\,dr}\frac{\psi(\pi_y\langle w_i,z_y\rangle)}{2}
\\
&\geq Q_2\frac{n_02^{-n}}{n_0\eps + 2}e^{-L}e^{\int_{-t_y}^{0}\phi(g_rz_y)\,dr}e^{-t_yP(\phi)}\sum_{i=1}^m \frac{\psi(\pi_y\langle w_i,z_y\rangle)}{2}e^{-t_iP(\phi) + \int_{0}^{t_i}\phi(g_rw_i)\,dr}.
\end{align*}
	where the last inequality utilizes the Bowen property as we did in our previous estimates.

We also have the following upper bound:
	\begin{align*}
	\mu(\xi_n(y)) &\leq \mu(g_{\left[\frac{k_y}{2^n},\frac{k_y+1}{2^{n}}\right]}\langle V^u_{x_0},B_{t_y}^s(z_y,\eps)\rangle)
	\\
	&\leq \frac{1}{1-\beta^2}\sum_{i=1}^{m} \mu(g_{\left[\frac{k_y}{2^n},\frac{k_y+1}{2^{n}}\right]}\langle A_i,B_{t_y}^s(z_y,\eps)\rangle)
	\\
	&\leq \frac{1}{1-\beta^2} \sum_{i=1}^{m}\mu(g_{\left[\frac{k_y}{2^n},\frac{k_y+1}{2^{n}}\right]}\langle B_{t_i}^u(w_i,2\eps),B_{t_y}^s(z_y,\eps)\rangle)
	\\
	&\leq \frac{1}{1-\beta^2}\sum_{i=1}^{m}Q_1n_02^{-(n+1)}e^{-(t_i+t_y)P(\phi) + \int_{-t_y}^{t_i}\phi(g_r\langle w_i,z_y\rangle)}\,dr
	\\
	&\leq\frac{1}{1-\beta^2}Q_1n_02^{-(n+1)}e^{2L}e^{\int_{-t_y}^0\phi(g_rz_y)\,dr}\sum_{i=1}^{m}e^{-(t_i+t_y)P(\phi) + \int_{0}^{t_i}\phi(g_r w_i)\,dr}
	\end{align*}

Combining these two estimates, we have that for sufficiently large $n$,
	
\begin{align*}
\frac{\int_{\xi_n(y)}\psi\,d\mu}{\mu(\xi_n(y))} &\geq (1-\beta^2)\frac{Q_2\frac{n_02^{-n}}{n_0\eps + 2}e^{-L}e^{\int_{-t_y}^{0}\phi(g_rz_y)\,dr}e^{-t_yP(\phi)}\sum_{i=1}^m \frac{\psi(\pi_y\langle w_i,z_y\rangle)}{2}e^{-t_iP(\phi) + \int_{0}^{t_i}\phi(g_rw_i)\,dr}}{Q_1n_02^{-(n+1)}e^{2L}e^{\int_{-t_y}^0\phi(g_rz_y)\,dr}\sum_{i=1}^{m}e^{-(t_i+t_y)P(\phi) + \int_{0}^{t_i}\phi(g_r w_i)\,dr}}
\\
&\geq (1-\beta^2)\frac{Q_2}{Q_1(n_0\eps+2)e^{4L}}\frac{\sum_{i=1}^{m}\psi(\pi_y\langle w_i,z_y\rangle) e^{-t_iP(\phi) + \int_0^{t_i}\phi(g_rw_i)\,dr}}{\sum_{i=1}^{m}e^{-t_iP(\phi) + \int_0^{t_i}\phi(g_rw_i)\,dr}}.
\end{align*}
	
Therefore, for all large $n$ we can directly compare our estimates for $\int_{V_x^u}\psi\circ\pi_{x,y}\,d\mu_x^u$ and a lower bound for $\int_{V_y^u}\psi\,d\mu_y^u$. Recall that we have already shown that 

$$\frac{\sum\limits_{i=1}^m \psi(\pi_y \langle w_i,z_x\rangle)e^{-t_iP(\phi)+\int_0^{t_i}\phi(g_r w_i)\,dr}}{\sum\limits_{i=1}^m e^{-t_iP(\phi)+\int_0^{t_i}\phi(g_s w_i)\,ds}} \geq \frac{Q_2}{Q_1e^{4L}(n_0\eps+2)}\left(\frac{\int_{\xi_n(x)}\psi\circ\pi_y\,d\mu}{\mu(\xi_n(x))} - \beta\right).$$	
Further, $\pi_y\langle w_i,z_x\rangle = \pi_y\langle w_i,z_y\rangle$. Consequently, we have:

\begin{align*}
\frac{\int_{\xi_n(y)}\psi\,d\mu}{\mu(\xi_n(y))} &\geq (1-\beta^2)\frac{Q_2}{Q_1(n_0\eps+2)e^{4L}}\frac{\sum_{i=1}^{m}\psi(\pi_y\langle w_i,z_y\rangle) e^{-t_iP(\phi) + \int_0^{t_i}\phi(g_rw_i)\,dr}}{\sum_{i=1}^{m}e^{-t_iP(\phi) + \int_0^{t_i}\phi(g_rw_i)\,dr}}
\\
&\geq (1-\beta^2)\left(\frac{Q_2}{Q_1(n_0\eps+2)e^{4L}}\right)\left(\frac{Q_2}{Q_1e^{4L}(n_0\eps+2)}\right)\left(\frac{\int_{\xi_n(x)}\psi\circ\pi_y\,d\mu}{\mu(\xi_n(x))} - \beta\right).
\end{align*}
We can choose $\beta$ arbitrarily small, independent of $n$. Thus, we have shown that
	
	\begin{align*}
		\frac{\int_{V_x^u}\psi\circ\pi_{x,y}\,d\mu_x^u}{\int_{V_y^u}\psi\,d\mu_y^u} &= \lim_{n\to\infty} \frac{\frac{\int_{\xi_n(x)}\psi\circ \pi_y\,d\mu}{\mu(\xi_n(x))}}{\frac{\int_{\xi_n(y)}\psi\,d\mu}{\mu(\xi_n(y))}}
		\\
		&\leq \left(\frac{e^{4L}(n_0\eps+2)Q_1}{Q_2}\right)^2.
	\end{align*}

\end{proof}

Using Proposition \ref{prop: main prop}, we now complete our proof of local product structure. This proof is standard, and follows \cite[\S3]{ClimenhagaGibbs}.

\begin{proposition}\label{prop: LPS}
Given $\mathcal{B}(n_0,\eps) \subset \Reg(c)$ for some $c > 0$, there exists $K := K(c) > 0$ such that for almost every $x,y\in \mathcal{B}(n_0,\eps)$,
$$K^{-1} \leq \frac{d(\pi_{x,y})_*\mu_x^u}{d\mu_y^u} \leq K.$$
\end{proposition}

\begin{proof}
Let $U\subset V_y^u$ be an arbitrary open set, and observe that there exists a sequence of continuous functions $\psi_n : V_x^u\to (0,1]$ which converges pointwise to $\mathbf{1}_U$. For each such $\psi_n$ we have by Proposition \ref{prop: main prop} that
$$\int_{V^u_y}\psi_n\,d(\pi_{x,y})_*\mu_x^u \leq K\int_{V_y^u}\psi_n\,d\mu_y^u.$$
Then, by dominated convergence, this implies that $(\pi_{x,y})_*\mu_x^u(U) \leq K\mu_y^u(U)$. Letting $E\subset V_y^u$ be an arbitrary measurable set, by outer regularity, we have that
$$\mu_x^u(\pi_{x,y}^{-1}E) = \inf\{\mu_x^u(\pi_{x,y}^{-1}U\mid U\supset E \text{ is open }\} \leq \inf\{K\mu_y^u(U\mid U\supset E \text{ is open }\} = K\mu_y^u(E).$$
As $x$ and $y$ are interchangeable, we also have the reverse inequality, and so our proof is complete.
\end{proof}

\begin{corollary}\label{cor: ae lps}
For almost every $\gamma\in GS$, there exists a flow box $\mathcal{B}(n_0,\eps) \ni \gamma$ and some $K := K(\gamma) > 0$ such that for $\mu$-almost every $x,y\in \mathcal{B}(n_0,\eps)$,
$$K^{-1} \leq \frac{d(\pi_{x,y})_*\mu_x^u}{d\mu_y^u} \leq K.$$
\end{corollary}

\begin{proof}
Recall that if $\lambda(\gamma) = 0$, then $\lambda(g_t\gamma) = 0$ for all $t\geq 0$ or all $t\leq 0$ by \cite[Proposition 3.4]{CCESW}. As $\mu(\lambda^{-1}(0)) = 0$ by Corollary \ref{cor: zero measure Sing}, this implies that for $\mu$-a.e. $\gamma\in GS$, $\lambda(g_t\gamma) > 0$ for all $t\in\mathbb{R}$. Consequently, we can construct a rectangle at $\mu$-a.e. $\gamma$ via Lemma \ref{lemma: bracket_su geometry}, and in turn, a flow box $\mathcal{B}(n_0,\eps)$ for some appropriate choice of $n_0$ and $\eps$. Applying Proposition \ref{prop: LPS} completes the proof.
\end{proof}

%
\section{The Bernoulli property}\label{sec:Bernoulli}

The proof of the Bernoulli property (Theorem \ref{mainthm:Bernoulli}) follows an argument due to Ornstein and Weiss \cite{OW}. A presentation of their argument with somewhat more detail can be found in \cite{CH}; we follow this presentation most closely.

After some preliminary definitions and lemmas, we provide an outline of the argument before beginning the proof.

%
\subsection{Definitions and lemmas}\label{subsec:defs2}

We begin by recording a few definitions and lemmas needed in this section alone.

\begin{definition}
Given a measurable set $A\subset G\SSS$ and $x\in A$, $V_x^{u,A}$ is the connected component of $W^u(x) \cap A$ containing $x$.
\end{definition}

\begin{lemma}\label{lem:agree}
Let $c>0$ and $B>0$ be given. Then there exists some $\delta$, depending only on $c$ and $B$, such that the following holds:

Let $q \in G\SSS$ be such that $\lambda(g_{\hat t}q)>c$ for some $\hat t \in \left[\frac{\theta_0}{2c}+2,B\right]$. If $x,y\in B(q,\delta)$ and $y\in V^{u,B(q,\delta)}_x$, then $x(-\infty,t_0]=y(-\infty,t_0]$ for some $t_0\geq0$.
\end{lemma}

\begin{proof}
Suppose $\lambda(g_{\hat t}q)>c>0$ for some $\hat t \in \left[\frac{\theta_0}{2 c}+2,B\right]$. By \cite[Prop 3.9]{CCESW}, there exists $t_1 \in \left[\hat t-\frac{\theta_0}{2c}, \hat t+\frac{\theta_0}{2c}\right] \subset \left[2,B+\frac{\theta_0}{2c}\right]$ such that $q(t_1)$ is a cone point at which $q$ turns by angle at least $s_0c$ away from $\pm \pi$.

Using the uniform continuity of the geodesic flow over $t\in\left[0,B+\frac{\theta_0}{2c}+2\right]$, choose $\delta$ such that
\begin{itemize}
	\item $\delta < \frac{s_0c}{16}$, and
	\item if $d_{GX}(z,z') < \delta$, then $d_{GX}(g_t z, g_t z')< \frac{s_0 c}{16}$ for all $t\in \left[0,B+\frac{\theta_0}{2c}+2\right].$
\end{itemize}
Note that $\delta$ depends only on $c$ and $B$.

By Lemma \ref{lem:closeness in S and GS}, since $d_{G\SSS}(x,q)<\delta<\frac{s_0 c}{16}$ and $d_{G\SSS}(g_{t_1+2}x,g_{t_1+2}q)<\delta<\frac{s_0 c}{16}$, $d_\SSS(x(0), q(0))<\frac{s c}{8}$ and $d_\SSS(x(t_1+2), q(t_1+2))<\frac{s c}{8}$. By the CAT(0) condition and some simple Euclidean trigonometry, the angle at $q(t_1)$ between $q(0,t_1)$ and the segment $[x(0),q(t_1)]$ is $<\frac{s_0 c}{2}$. Similarly, the angle at $q(t_1)$ between $q(t_1, t_1+1)$ and the segment $[q(t_1), x(t_1+2)]$ is also $<\frac{s_0 c}{2}$.

The angle that $q$ turns with at $q(t_1)$ is at least $s_0 c$ away from $\pm \pi$. Therefore, the angle between the segments $[x(0),q(t_1)]$ and $[q(t_1),x(t_1+2)]$ is $>\pi$ on both sides, and so the concatenation of these geodesic segments is itself a geodesic segment. Therefore, the cone point $q(t_1)$ belongs to the geodesic $x$. Say it is $x(t_0)$. By choosing  $\delta<<1$, we can ensure that $t_0>0$.

The argument above holds for $y$ as well, so $q(t_1)$ belongs to $y$. Lifting $x$ and $y$ to nearby geodesics in $\tilde \SSS$, since $y\in V^{u,B(q,\delta)}_x$, $\tilde x(-\infty)=\tilde y(-\infty)$. Then $\tilde x(t_0)=\tilde y(t_0)$, and the Lemma follows immediately.
\end{proof}

The following definitions relate to partitions, the main objects of study in this section. For the remainder of this section, let $(X,\mu)$ and $(Y,\nu)$ be general probability measure spaces.

\begin{definition}\label{defn:eps almost every}
A property $\mathcal{P}$ holds for $\eps$-almost every set in a partition $\{ C_j\}$ of $(X,\mu)$ if 
\[ \mu\left(\bigcup_{j'\in J'} C_{j'}\right)<\eps, \  \mbox{ where } J' = \{ j\in J: \mathcal{P} \mbox{ does not hold for } C_j\}. \] 
\end{definition}

A distance between finite partitions can be defined as follows. (See \cite{PV} or \cite {CH}.) Let $J(\mu,\nu)$ be the set of joinings of $\mu$ and $\nu$. That is, $J(\mu, \nu)$ is the set of probability measures $\lambda$ on $X\times Y$ such that $(\pi_1)_*\lambda = \mu$ and $(\pi_2)_*\lambda = \nu$, where $\pi_i$ are the usual projections onto the factors or $X\times Y$.  Given measurable partitions $\xi=\{C_1, \ldots, C_k\}$ of $X$ and $\eta = \{ D_1, \cdots, D_k\}$ of $Y$, for each $x\in X$, we write $\xi(x)$ for the index $i$ such that $x\in C_i$. Similarly $\eta(y)$ is the index $j$ such that $y\in D_j$.

\begin{definition}\label{defn:dbar1}
The $\bar d$-distance between $\xi$ and $\eta$ is
\[ \bar d(\xi,\eta):= \inf_{\lambda\in J(\mu,\nu)} \lambda(\{ (x,y): \xi(x)\neq \eta(y) \}).\]
\end{definition}

This distance can be extended to sequences of partitions $\{\xi_i\}_1^n$ and $\{ \eta_i\}_i^n$ as well.

\begin{definition}\label{defn:h}
Given two sequences of partitions $\{\xi_i\}_1^n$ and $\{\eta_i\}_{1}^n$, we define
\[h(x,y) := \frac{1}{n}\sum_{i=1}^n\delta\{\xi_i(x) \neq \eta_i(y)\},\]
where $\delta$ is the indicator function.
\end{definition}

\begin{definition}\label{defn:dbar sequence}
The $\bar d$-distance between sequences of partitions $\{\xi_i\}_1^n$ and $\{\eta_i\}_1^n$ is 
\[ \bar d(\{\xi_i\}_1^n, \{\eta_i\}_1^n):= \inf_{\lambda\in J(\mu,\nu)} \int h(x,y) d\lambda(x,y).\]
\end{definition}

\begin{definition}\label{defn:conditional}
For a measurable $E\subset X$ with $\mu(E)>0$, the conditional measure induced by $\mu$ on $E$ is 
\[ \mu|_E(S) := \frac{\mu(S\cap E)}{\mu(E)}.\]
If $\xi=\{C_1, \ldots, C_k\}$ is a partition of $(X,\mu)$, it induces the partition
\[\xi|_E :=\{ C_1\cap E, \ldots, C_k\cap E\}\]
on $(E, \mu|_E)$.
\end{definition}

A central tool in the argument below will be Very Weak Bernoulli partitions, introduced by Ornstein in \cite{Ornstein_imbedding}. Let $f$ be a measure-preserving, invertible map on $(X,\mu).$

\begin{definition}\label{defn:VWB}
A finite partition $\xi$ is \emph{Very Weak Bernoulli (VWB)} if for any $\epsilon > 0$, there exists $N\in \mathbb{N}$ such that for all $N_1\geq N_0 \geq N$, for all $n > 0$, and for $\epsilon$-almost every $A\in\bigvee_{k=N_0}^{N_1}f^k\xi$, we have
\[\bar{d}( \{f^{-i}\xi|_A\}_1^n, \{f^{-i}\xi\}_1^n) \leq \epsilon\]
with respect to $(A,\mu|_A)$ and $(X,\mu)$.
\end{definition}

%
\subsection{An outline of the argument}\label{subsec:outline}

Our main task in this section is to construct a sequence of VWB partitions with arbitrarily small diameter. To do so, we want to show that we can satisfy the conditions of the following:

\begin{lemma}[\cite{OW} Lemma 1.3, or \cite{CH} Lemma 4.3]\label{lemma:psi}
    Let $(X,\mu)$ and $(Y,\nu)$ be two non-atomic, Lebesgue probability spaces. Let $\{\alpha_i\}, \{\beta_i\}$ for $1\leq i\leq n$, be two sequences of partitions of $X,Y$, respectively. Suppose there is a map $\psi:X\to Y$ such that
    \begin{enumerate}
        \item There is a set $E_1\subset X$ with $\mu(E_1)<\eps$ such that for all $x\notin E_1$,
        \[h(x,\psi x)<\eps;\]
        \item There is a set $E_2\subset X$ with $\mu(E_2)<\eps$ such that for any measurable $B\subset X\setminus E_2$, 
        \[\left| \frac{\mu(B)}{\nu(\psi B)}-1\right|<\eps.\]
    \end{enumerate}
    Then $\bar d(\{\alpha_i\}, \{\beta_i\})<c\eps$ (where $c$ is some constant depending only on our underlying system.)
\end{lemma}

We now revert from the general notation above to our specific situation as outlined in Theorem \ref{mainthm:LPS}, replacing the general probability measure space $(X,\mu)$ with $(\GS, \mu)$, and letting $f=g_1$ be the time-one map for the geodesic flow.

Specifically, we will apply Lemma \ref{lemma:psi} with:
\begin{itemize}
    \item $(X,\mu,\{\alpha_i\}) = (A, \mu|_A, \{f^{-i}\xi|_A\})$
    \item $(Y,\nu,\{\beta_i\}) = (\GS, \mu, \{f^{-i}\xi\})$
\end{itemize}
for $i=1+m, \ldots, n+m$, for some $m$ to be chosen below. $G\SSS$ is a complete, separable metric space, so these are indeed Lebesgue spaces. Since $\mu$ is flow invariant, $\mu$ and $\mu|_A$ can have no atoms. Note that to satisfy Definition \ref{defn:VWB} we only need to apply this for $\eps$-almost every $A\in\bigvee_{k=N_0}^{N_1}f^k\xi$. The result will be that 
\[\bar d\left( \{f^{-i}\xi|_A\}_{i=1+m}^{n+m} ,\{f^{-i}\xi\}_{i=1+m}^{n+m} \right)<c\eps\]
where $c$ depends only on our space $\GS$. That is, $f^{-m}\xi$ is VWB.

A rough outline of the argument for the construction of $\xi$ and $\psi$ as in Proposition \ref{lemma:psi} is as follows.

\begin{enumerate}
    \item Let $\xi$ be a partition of $\GS$ with a small amount of measure concentrated near the boundaries of its sets. 
    \item Define an auxiliary ``almost partition'' of $\GS$ into dynamical rectangles, where on each rectangle, $\mu$ is equivalent or absolutely continuous with respect to a product measure supported on the local unstable and center-stable manifolds.
    \item Show that we can choose $M$ so that for $m_2\geq m_1 \geq M$, ``most'' partition elements $A\in \bigvee_{m_1}^{m_2} f^i\xi$ have a subset of relatively large measure $S_A\subset A$ which stretches completely across the rectangle.
    \item Given a set $S_A$ which ``stretches across'' a rectangle $R$ in the unstable direction, show that we can define a map $\psi_R : S_A\cap R\to R$ which is product measure-preserving, and maps points along their local center stable sets.
    \item After identifying a small-measure collection of `bad' atoms from $\bigvee_{k=N_0}^{N_1}f^k\xi$, given a `good' atom $A$, glue together $\psi : A\to \GS$ from the $\psi_R$'s, and then show that it meets the necessary criteria from Lemma \ref{lemma:psi}.
\end{enumerate}

Of these steps, only (2), (3), and (5) require hyperbolicity of some kind. Step (1) is an entirely general construction for which we quote a Lemma from \cite{OW_some_hyperbolic}. The main work for Step (2) was accomplished in Section \ref{subsec:lps}. We discuss the first two steps in Section \ref{subsec:1-2} below. In Section \ref{subsec:3} we give a precise definition and prove the desired `layerwise' intersection property of Step (3). The argument is completed with Steps (4) and (5) in Sections \ref{subsec:4} and \ref{subsec:5}, primarily by recalling the argument of \cite{OW}.

Once the argument above is complete, we finish the proof of Theorem \ref{mainthm:Bernoulli} by invoking a pair of theorems from \cite{OW}:

\begin{theorem}[\cite{OW}, Theorems A and B]\label{thm:vwB to B}\
\begin{itemize}
    \item[(A)] If $\xi$ is a VWB partition of the system $(X,\mu,f)$, then $(X, \bigvee_{i=-\infty}^\infty f^i\xi, \mu, f)$ is Bernoulli.
    \item[(B)] If $\mathcal{A}_1 \subseteq \mathcal{A}_2 \subseteq \cdots$ is an increasing sequence of $f$-invariant $\sigma$-algebras with $\mathcal{B} = \bigvee_{n=1}^\infty \mathcal{A}_n$, and each $(X,\mathcal{A}_n,\mu,f)$ is Bernoulli, then $(X,\mathcal{B},\mu,f)$ is Bernoulli.
\end{itemize}    
\end{theorem}

\begin{proof}[Proof of Theorem \ref{mainthm:Bernoulli}]
    We apply Theorem \ref{thm:vwB to B} to $(\GS, \mathcal{B}, \mu, f)$ with $f=g_1$ and $\mathcal{B}$ the Borel $\sigma$-algebra. Our argument in the following sections will show that for any partition $\xi$ with a small amount of measure concentrated near its boundary, $f^{-m}\xi$ is VWB. Proposition \ref{prop:good boundary} below ensures that such $\xi$ exist. Indeed, it ensures that such $\xi$ with arbitrarily small diameter can be constructed. Applying this Proposition inductively, we can find a sequence $\xi_1 \subseteq \xi_2 \subseteq \cdots$ with diameters going to zero. $\bigvee_{n=1}^\infty \xi_n$ certainly generates $\mathcal{B}$.

    Applying \ref{thm:vwB to B}(A), we find that $(\GS,\bigvee_{i=-\infty}^\infty f^i\xi_n,\mu,f)$ is Bernoulli. Since $\mathcal A_n:=\bigvee_{i=-\infty}^\infty f^i\xi_n$ are $f$-invariant, and $\mathcal{B} = \bigvee_{n=1}^\infty \mathcal{A}_n$, by Theorem \ref{thm:vwB to B}(B), $(\GS,\mathcal{B},\mu,f)$ is Bernoulli.
\end{proof}

%
\subsection{Steps 1 \& 2: Good partitions}\label{subsec:1-2}

The first step toward using Proposition \ref{lemma:psi} is to note the existence of a partition of $G\SSS$ such that the measure near its boundary is controlled. This is given by the following general fact.

\begin{proposition}[See \cite{OW_some_hyperbolic}, Lemma 4.1 and following remarks]\label{prop:good boundary}
Let $(X,\mu)$ be a compact metric space, $\mu$ a probability measure. For all $\epsilon > 0$, there exists a partition $\xi$ of diameter less than $\epsilon$ such that there exists $C > 0$ such that for all $r > 0$,
$$\mu(B(\partial \xi,r)) \leq Cr$$
where $B(\partial \xi,r)$ is the $r$-neighborhood of the union of the boundaries of the sets in $\xi$.
\end{proposition}

The second step is to find an ``almost partition" of $\GS$ using rectangular subsets of the flow boxes constructed in Section \ref{subsec:lps}. We will denote these rectangles by $R$ or $R_i$. This matches the presentation of \cite{CH}, for example, but these $R$'s should not be confused with the rectangle transversals considered earlier.

Towards that end, note that we may assume (by decreasing $\eps$ and increasing $n$ as necessary) that the bracket structure provided by $[-,-]$ is defined at scale $\bar\delta$ which is at least twice the diameter of any such $\mathcal{B}(n_0,\eps).$

There are product coordinates on $\mathcal{B}(n,\eps)$, defined as follows.  Let $z\in \mathcal{B}(n,\eps)$. We assign it coordinates
\[ z \mapsto (z^u, z^s, z^c) \in W^u(x,\eps) \times W^s(x,\eps) \times \left[-\frac{1}{n},\frac{1}{n}\right]\]
where 
\[z^u = W^u(x,\eps) \cap W^{cs}(z,\bar\delta)\]
\[z^s = W^s(x,\eps) \cap W^{cu}(z,\bar\delta)\]
\[g_{z^c}\langle z^u, z^s\rangle = z.\]
(See Figure \ref{fig:product coords}.)

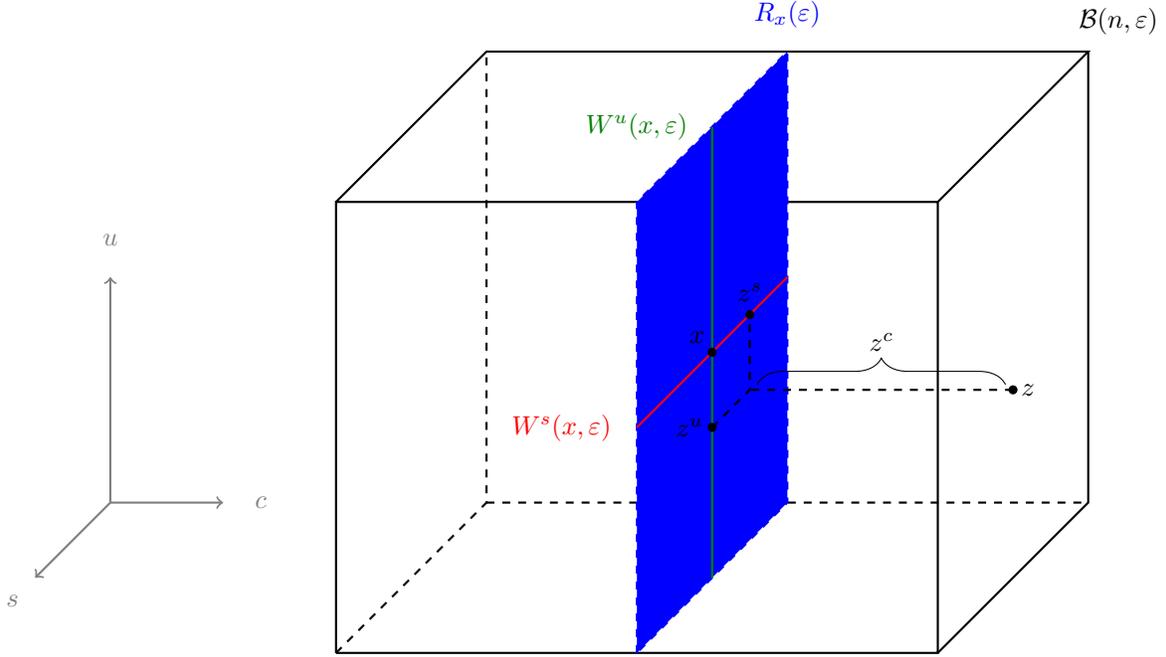
\begin{figure}[h]
\centering
\begin{tikzpicture}[scale=1]

\draw[thick] (-2,-2) -- (6,-2) -- (6,4) -- (-2,4) -- (-2,-2);
\draw[thick] (-2,4) -- (0,6) -- (8,6) -- (6,4);
\draw[thick] (8,6) -- (8,0) -- (6,-2);
\draw[thick, dashed] (-2,-2) -- (0,0) -- (0,6);
\draw[thick, dashed] (0,0) -- (8,0);

\fill[blue, opacity=.3] (2,-2) -- (2,4) -- (4,6) -- (4,0);
\draw[thick, dashed, blue] (2,-2) -- (2,4) -- (4,6) -- (4,0) -- (2,-2);

\draw[thick, green!50!black] (3,-1) -- (3,5);
\draw[thick, red] (2,1) -- (4,3);

\fill (3,2) circle(.06);
\fill (7,1.5) circle(.06);
\fill (3,1) circle(.06);
\fill (3.5,2.5) circle(.06);

\draw[thick, dashed] (7,1.5) -- (3.5,1.5) ;
\draw[thick, dashed] (3,1) -- (3.5,1.5) -- (3.5,2.5) ;

\node at (2.8,2.2) {$x$};
\node[blue] at (4,6.5) {$R_x(\eps)$};
\node[red] at (1,1) {$W^s(x,\eps)$};
\node[green!50!black] at (2,5) {$W^u(x,\eps)$};
\node at (8.4,6.4) {$\mathcal{B}(n,\eps)$};
\node at (7.2,1.5) {$z$};
\node at (3.5,2.8) {$z^s$};
\node at (2.7,1) {$z^u$};

\draw [decorate,decoration={brace,amplitude=10pt},yshift=2pt] (3.6,1.5) -- (6.9,1.5) node [black,midway,yshift=16pt] {$z^c$};

\draw[thick, gray, ->] (-5,0) -- (-5,3) ;
\draw[thick, gray, ->] (-5,0) -- (-3.5,0) ;
\draw[thick, gray, ->] (-5,0) -- (-6,-1) ;

\node[gray] at (-5, 3.5) {$u$} ;
\node[gray] at (-6.3, -1.3) {$s$} ;
\node[gray] at (-3, 0) {$c$} ;

\end{tikzpicture}
\caption{The product coordinates on a flow box.}\label{fig:product coords}
\end{figure}

We note that topologically, $W^u(x,\eps)$ and $W^s(x,\eps)$ are intervals -- they can be identified via projecting to the forward (resp. backward) endpoints at infinity with subsets of the circle $\partial_\infty \tilde \SSS.$

\begin{definition}\label{defn: rectangle definition}
A subset $A\subset \mathcal{B}(n,\eps)$ is a \emph{rectangle} if there are intervals (open, closed, or half open) $I^u \subset W^u(x,\eps)$, $I^s\subset W^s(x,\eps)$, and $I^c\subset [-\frac{1}{n}, \frac{1}{n}]$ such that
\[A= \{z: z^u \in I^u, z^s\in I^s, z^c\in I^c\}\]
or, equivalently, under the map to the coordinates,
\[A \mapsto I^u \times I^s \times I^c.\]
\end{definition}

Suppose that $B_1:=\mathcal{B}(n_1, \eps_1)$ centered at $x_1$ and $B_2:=\mathcal{B}(n_2,\eps_2)$ centered at $x_2$ intersect. Since the diameters of $B_1$ and $B_2$ are less than half the scale on which the bracket $[-,-]$ is defined, we can relate the coordinates for $B_1$ and $B_2$.

Let $z\in B_1\cap B_2$ with coordinates $(z_1^u, z_1^s, z_1^c)$ for $B_1$ and $(z_2^u, z_2^s, z_2^c)$ for $B_2$.
\begin{itemize}
    \item \textbf{Unstable coordinates:} $z_i^u$ both lie on $W^{cs}(z,\bar\delta)$. Hence,
    \[z^u_1 = W^u(x_1,\eps_1) \cap W^{cs}(z_2^u, \bar\delta).\]
    \item \textbf{Stable coordinates:} $z_i^s$ both lie on $W^{cu}(z,\bar\delta)$. Hence,
    \[z_1^s = W^s(x_1,\eps_1)\cap W^{cu}(z_2^s,\bar\delta).\]
    \item \textbf{Flow coordinates:} $z_1^c-z_2^c = t^*$ where $g_{t^*}\langle z_1^u, z_1^s\rangle = \langle z_2^u, z_2^s\rangle$. Note that $t^*$ is the distance along the flow through $z$ between $R_1(\eps_1)$ and $R_2(\eps_2)$. This distance is constant in $z$ as $R_i(\eps_i)$ are connected and foliated by stable and unstable leaves, and the flow takes (un)stable leaves to (un)stable leaves. Hence,
    \[z_1^c = t^*+z_2^c.\]
\end{itemize}

The transformation of the flow coordinate obviously maps intervals to intervals. The same is true for the other foliations. To see that these maps preserve the ordering of points on the (un)stable leaves they map between, one can either draw a geometric picture of the maps in $\tilde \SSS$ or note that these maps are given by holonomies along the leaves of the center (un)stable foliation. Since distinct leaves cannot intersect, ordering must be preserved.

The work above allows us to prove the following.

\begin{lemma}\label{lem:rectangle intersection}
If $A$ is a rectangle in $B_2$, then $A\cap B_1$ is a rectangle in $B_1$.
\end{lemma}

\begin{proof}
Suppose that $A$ has coordinates $I^u_2 \times I^s_2 \times I^c_2$ in $B_2$. Then in $B_1$ it has the coordinates $I^u_1 \times I^s_1 \times I^c_1$ with
\[I^u_1 = \{ z_1^u: z_2^u\in I^u_2\} = \{W^u(x_1,\eps_1) \cap W^{cs}(z_2^u,\bar\delta): z_2^u\in I_2^u \}\]
\[I^s_1 = \{ z_1^s: z_2^s\in I^s_2\} = \{W^s(x_1,\eps_1) \cap W^{cu}(z_2^s,\bar\delta): z_2^s\in I_2^s \}\]
\[I^c_1 = \{ z_1^c: z_s^c \in I^c_2 \}.\]
Since the maps $z_2^*\to z_1^*$ map intervals to intervals, $A$ is a rectangle in $B_1$.
\end{proof}

\begin{proposition}\label{prop:rectangle intersection and complement}
If $A_1$ and $A_2$ are rectangles in $B_1$ and $B_2$, then $A_1\cap A_2$ is a rectangle in $B_1$ and $A_1\setminus A_2$ is a finite union of rectangles in $B_1$. 
\end{proposition}

\begin{proof}
By Lemma \ref{lem:rectangle intersection}, $A_2\cap B_1$ is a rectangle in $B_1$. $A_1 \cap A_2 = A_1 \cap (A_2 \cap B_1)$ and $A_1 \setminus A_2 = A_1 \setminus (A_2\cap B_1)$ so we can without loss of generality assume $A_1$ and $A_2$ are both rectangles in the same flow box $B_1$.

In the stable/unstable/flow coordinates $A_1$ and $A_2$ are rectangles in $W^u(x_1,\eps_1)\times W^s(x_2,\eps_2)\times \left[-\frac{1}{n_1},\frac{1}{n_1}\right]$, itself a rectangle in $\mathbb{R}^3$ after identifying these (un)stable leaves with intervals. The result now follows from the analogous fact about rectangles in $\mathbb{R}^3$.
\end{proof}

\begin{lemma}\label{lem: rectangles have lps}
For any rectangle $A$ with $\mu(A)>0$ in a flow box $\mathcal{B}(n_0,\eps)$ on which $\mu$ has local product structure, the conditional measure $\mu|_A$ has local product structure.
\end{lemma}

\begin{proof}
As noted at the start of Section \ref{sec:lps}, local product structure can be deduced from the behavior of conditional measures under holonomy maps described in Proposition \ref{prop: LPS} and Corollary \ref{cor: ae lps}. When $\mu$ is conditioned on $A$ (using the fact that $\mu(A)>0$), the resulting conditional measures are renormalized versions of the measures for the full flow box and the holonomy maps are the same. Hence, these results hold for $A$, and $\mu|_A$ has local product structure. 
\end{proof}

The conditions on the partition we want are described in the following definition.

\begin{definition}\label{defn: good rectangles}
For a given $\delta>0$, a partition $\pi = \{ R_0, R_1, \ldots, R_k\}$ is a \emph{good almost partition of $\GS$ into rectangles} if
\begin{itemize}
    \item Each $R_i$ for $i=1, \ldots, k$ is a rectangle as in Definition \ref{defn: rectangle definition} with diameter $<\delta/C$ where $C$ is from Proposition \ref{prop:good boundary}.
    \item $R_0$ is some measurable set with $\mu(R_0)<\delta$.
    \item If $x,y\in R_i$ ($i=1, \ldots, k$) are on the same local unstable, then there is a curve in that local unstable of length $<8\diam(R_i)<8\delta/C$ joining them.
    \item For every $R_i$ $(i=1, \ldots, k)$
    \[\left| \frac{\mu^p_{R_i}(R_i)}{\mu(R_i)} -1\right| < \delta\]
    where $\mu^p_{R_i}$ is the local product structure measure on $R_i$.
    Each $R_i$ ($i=1, \ldots, k$) contains a subset $G_i$ with $\mu(G_i)>(1-\delta)\mu(R_i)$ such that at all points $x\in G_i$,
    \[ \left| \frac{d\mu^p_{R_i}}{d\mu}(x) -1 \right|<\delta.\]
    \item If $x,y\in R_i$ $(i=1, \ldots, k)$ are on the same local stable leaf, the distance between them contracts exponentially fast under the geodesic flow.
    \end{itemize}
\end{definition}

\begin{proposition}\label{prop: good rectangles exist}
Given any $\delta>0$ there exists a good almost partition of $\GS$ into dynamical rectangles.
\end{proposition}

\begin{proof}
In Corollary \ref{cor: ae lps} we proved that almost every $\gamma \in \GS$ -- specifically, any $\gamma$ with $\lambda(g_t \gamma)>0$ for all $t\in \mathbb{R}$ -- belongs to a flow box $\mathcal{B}(n,\eps)$ on which $\mu$ has local product structure. In Lemma \ref{lem: rectangles have lps} we showed that on any rectangle $R$ with $\mu(R)>0$ in such a flow box, $\mu$ also has local product structure. We will create our good almost partition from such rectangles. They are clearly measurable, and hence $R_0$ will be measurable.

First, note that we can choose the rectangles $R$ with diameter smaller than $\delta/C$ where $C$ is from Proposition \ref{prop:good boundary}..

Second, it is clear from the construction of the flow boxes (see Definitions \ref{definition: su-rectangle} and \ref{definition: flow box}, and take $n$ sufficiently large) that the unstable segments in each flow box, and hence each rectangle within it are of length $<8\diam(R)$. Also by construction (see Lemma \ref{lemma: bracket_su geometry}) if $x,y \in R$ are on the same stable leaf, then $x(t)=y(t)$ for all $t \geq 0$. Lemma \ref{lemma: boundedness cs} gives the desired exponential contraction.

We know from Lemma \ref{lem: rectangles have lps} that $\frac{d\mu}{d\mu^p_{R}} \in [\bar K^{-1}, \bar K]$ on each such rectangle. Hence, $\mu^p_{R} \ll \mu$ as well. $\frac{d\mu^p_{R}}{d\mu}$ is measurable, so at almost every point, on a sufficiently small neighborhood, it is nearly constant off a small measure set. Therefore, shrinking $R$ as necessary, and rescaling $\mu^p_R$ as necessary, we can ensure that $|\frac{d\mu^p_R}{d\mu}-1|<\delta$ on a set $G\subset R$ with $\mu(G)>(1-\delta)\mu(R)$. Combining this with our bounds on $\frac{d\mu}{d\mu^p_{R}}$ we can show that $|\frac{\mu^p_R(R)}{\mu(R)}-1|<\delta$.

Therefore, around almost every point in $\GS$ we can construct a rectangle answering the requirements of Definition \ref{defn: good rectangles} of arbitrarily small diameter. Let $E$ be the full measure set of such points. The fact that we can make our rectangles of arbitrarily small diameter and the results of Proposition \ref{prop:rectangle intersection and complement}, show that these rectangles form a generating semi-ring for the restriction of the $\sigma$-algebra of measurable sets to $E$. Therefore, every measurable set in this $\sigma$-algebra, including $E$ itself, can be approximated up to measure $<\delta$ by a finite, disjoint, union of our rectangles. These rectangles, together with their complement of measure $<\delta$, is the partition $\pi$.
\end{proof}

%
\subsection{Step 3: Layerwise intersection}\label{subsec:3}

The next step is to prove (Proposition \ref{prop:layerwise}, below) that given a partition $\xi$ as provided in Proposition \ref{prop:good boundary} and a collection of small rectangles, there is a time $M$ such that, after pushing $\xi$ forward by times $\geq M$, almost all of its subsets intersect each rectangle in our collection `layerwise.'

\begin{definition}\label{defn:layerwise}
A measurable set $S$ intersects a rectangle $R$ \emph{layerwise} if for all $x\in S\cap R$, $V_x^{u,S} \supseteq V_x^{u,R}$.
\end{definition}

First, we record a geometric lemma.

\begin{proposition}\label{prop:good points}
Let $\epsilon>0$ and $t_0\geq 0$ be given. Then there exists a set $E\subset GX$ (depending on $\eps$ and $t_0$), an $\eta>0$ (depending only on $\eps$), and $A,B\geq t_0$ (depending on $\eps$ and $t_0$) such that:
\begin{itemize}
	\item $\mu(E)<\frac{\eps^2}{2}$
	\item for each $\gamma \in E^c$, there exists a $t'_\gamma \in [A,B]$ such that $|\theta(\gamma, t'_\gamma)-\pi| \geq c$ for some $c>0$.
\end{itemize} 
\end{proposition}

\begin{proof}
Recall that $\Sing = \bigcap_{t\in \mathbb{R}} g_t(\{ \lambda = 0\})$. (Here, $\{\lambda=0\}$ is shorthand for $\{\gamma \in GX : \lambda(\gamma) = 0\}$.) Thus, 
\[ \Sing = \bigcap_{c'>0} \bigcap_{t\in\mathbb{R}} g_t\left(\left\{ \lambda< c' \right\} \right).\]
By Corollary \ref{cor: zero measure Sing}, $\mu(\Sing)=0$ for our equilibrium states, so given $\eps>0$, there exists $c^*>0$ such that 
\[ \mu\left(  \bigcap_{t\in \mathbb{R}} g_t(\{ \lambda<c^*\}) \right) < \frac{\eps^2}{4}.\]

Now,
\[ \bigcap_{t\in \mathbb{R}} g_t(\{ \lambda<c^*\}) = \bigcap_{n\in\mathbb{N}} \bigcap_{t\in [-n,n]} g_t\left(\left\{ \lambda<c^*\right\}\right).\]
Therefore, there exists $N_1\in \mathbb{N}$ such that 
\[ \mu \left( \bigcap_{t\in[-N_1,N_1]} g_t\left(\left\{ \lambda< c^* \right \}\right)\right) < \frac{\eps^2}{2}.\]
Let 
\[ F= \bigcap_{t\in[-N_1,N_1]} g_t\left(\left\{ \lambda< c^* \right \}\right).\]
Then
\[ F^c = \left\{ \gamma \in GX : \exists t \mbox{ with } |t|\leq N_1 \mbox{ such that } \lambda(g_t\gamma)\geq c^*)\right\}\]
and $\mu(F^c)\geq 1-\frac{\eps^2}{2}$.

Let $E=g_{-N_1-t_0-\frac{\theta_0}{2\eta}} F$; $E^c=g_{-N_1-t_0-\frac{\theta_0}{2\eta}} F^c$. We still have $\mu(E^c) \geq 1-\frac{\eps^2}{2}$. If $\gamma \in E^c$, then there exists a time $t_\gamma \in \left[t_0+\frac{\theta_0}{2c^*}, t_0+\frac{\theta_0}{2c^*}+2N_1\right]$ such that $\lambda(g_t \gamma) \geq c^*$. Let $c = \frac{c^*}{2s}$. Then for each $\gamma \in E^c$, there exists a $t'_\gamma \in [t_0, t_0+2N_1]=:[A,B]$ such that $|\theta(\gamma, t'_\gamma)-\pi| \geq c$.
\end{proof}

\begin{definition}\label{defn:good rectangle}
Let $E$ be as provided by Proposition \ref{prop:good points}. Given a collection $\mathcal{R}$ of rectangles, we decompose it as $\mathcal{R}= \mathcal{R}_E \cup \mathcal{R}_{E^c}$, where $\mathcal{R}_E = \{ R\in \mathcal{R}: R\subset E\}$ and $\mathcal{R}_{E^c} = \{ R\in\mathcal{R}: R\cap E^c\neq \emptyset\}$.
\end{definition}

Directly from Proposition \ref{prop:good points} we have

\begin{corollary}\label{cor:bad small}
For any collection $\mathcal{R}$ of rectangles,
\[ \mu\left(\bigcup_{R\in \mathcal{R}_E} R \right)<\frac{\eps^2}{2}.\]
\end{corollary}

In all that follows, $\xi = \{ F_j\}$ is a measurable partition of $G\SSS$. $B(\partial \xi, \eps)$ is the $\eps$-neighborhood of the union of the boundaries of the sets in $\xi$. We assume, using Proposition \ref{prop:good boundary}, that there exists $C>0$ so that for all $r>0$, $\mu(B(\partial\xi,r))<Cr$.

Fix $\eps>0$ and set $t_0=2$. Let $c>0$, $E\subset G\SSS$, and $A,B \geq t_0$ be as described in Proposition \ref{prop:good points} (for our choice of $\epsilon$ and $t_0$). Recall that each $\gamma\in E^c$ has a time $t'_\gamma\in [A,B]$ associated to it for which $\theta(\gamma, t'_\gamma)\geq \eta$.

Given $c$ and $B$, let $\delta>0$ be as given in Lemma \ref{lem:agree}. By Lemma \ref{lem:agree}, if $q\in E^c$ and $x,y\in B(q,\delta)$ with $y\in V^{u,B(q,\delta)}_x$, then $x$ and $y$ agree over times in $(-\infty, t'_q]$. Therefore, Lemma \ref{lemma: expansion u} applies for such $x,y$.

We are now ready for the main proposition of Step 4:

\begin{proposition}\label{prop:layerwise}
Let $\epsilon>0$ be given. Let $\mathcal{R}$ be a finite, pairwise disjoint collection of rectangles with diameters $<\delta$.

If $\xi$ is a partition as described above, then there exists $M$ (depending on the $C$ mentioned in the description of $\xi$ as well as on $\eps$) such that for all $m_2 \geq m_1 \geq M$ and for $\eps$-almost every $F \in \bigvee_{k=m_1}^{m_2} f^k \xi$ there is a set $S_F\subset F$ with $\mu|_F(S_F) > 1-\eps$ which intersects each $R\in \mathcal{R}$ layerwise.
\end{proposition}

\begin{proof}
Let $\eps>0$ be given; set $t_0=2$. Let $E$ be as provided by Proposition \ref{prop:good points}; it depends on $\epsilon$ and $t_0$.

Let $\mathcal{R}$ be a collection of rectangles with diameters $<\delta$. Decompose $\mathcal{R} = \mathcal{R}_E \cup \mathcal{R}_{E^c}$ as per Definition \ref{defn:good rectangle}. Let $q_R\in R \in \mathcal{R}_{E^c}$ for all $R \in \mathcal{R}_{E^c}$.

Recall that $\xi$ is a partition such that $\mu(B(\partial\xi,r'))<Cr'$ for all $r'>0$. Let $M$ be so large that $C\delta \sum_{k=M}^{\infty} e^{-k}<\frac{\eps^2}{2}$.

Let $F \in \bigvee_{k=m_1}^{m_2} f^k \xi.$ Then if $\xi = \{ F_j\}$, 
\[ F = \bigcap_{k=m_1}^{m_2} f^k F_{j_k}.\]
A point $x$ is in $F$ if and only if for all $k \in [m_1, m_2]$, $f^{-k}x \in F_{j_k}$.

Let $S_1= F \setminus (\cup_{R\in\mathcal{R}} R)$. Let 
\[ S_2 = \bigcup_{R \in \mathcal{R}_{E^c}} \{ x\in F \cap R : V_x^{u,R} \subseteq V_x^{u,F}\}.\] 
Let $S_F = S_1 \cup S_2$. Our goal is to prove that $\mu|_F(S_F^c)\leq \eps.$

Suppose that $x\notin S_F$. Then $x$ belongs to some rectangle $R\in\mathcal{R}$. If that rectangle is in $\mathcal{R}_E$, then $x\in S^c_F\cap (\cup_{R\in \mathcal{R}_{E}}R)$. By Corollary \ref{cor:bad small}, $\mu(S^c_F\cap (\cup_{R\in \mathcal{R}_{E}}R))<\frac{\eps^2}{2}$. 

If, on the other hand, that rectangle is in $R_{E^c}$, then there exists some $y \in V_x^{u,R}$ such that $y\notin F$. Since $y\notin F$, there exists some $k\in [m_1,m_2]$ such that $y\notin f^k F_{j_k}$. As they are both in $R$, $d_{GX}(x,y)<\delta$. Since $y\in V_x^{u,R}$ and $R$ is a good rectangle, by Lemma \ref{lemma: expansion u},
\[ d_{GX}(g_{-k}x,g_{-k}y) < \delta e^{-k}. \]
Therefore, $g_{-k}x \in B(\partial \xi, \delta e^{-k})$.

Let 
\[X_k = g_kB(\partial \xi, \delta e^{-k}).\]
By the choice of $\xi$, $\mu(X_k) < C\delta e^{-k}$.

We have shown that for each $F \in \bigvee_{k=m_1}^{m_2} f^k \xi$, 
\[F\setminus S_F \subset \left(\bigcup_{R\in \mathcal{R}_{E}}R\right)\cup \bigcup_{k=m_1}^{m_2} X_k.\] 
Therefore, letting $Z=\bigcup_{F\in \xi} (F\setminus S_F)$,
\[\mu(Z) \leq   \mu\left(\bigcup_{R\in \mathcal{R}_{E}}R\right)+\sum_{k=m_1}^{m_2} \mu(X_k) < \frac{\eps^2}{2}+C\delta \sum_{k=M}^\infty e^{-k}  < \frac{\eps^2}{2}+\frac{\eps^2}{2} = \eps^2.\]

Writing $\bigvee_{k=m_1}^{m_2} f^k \xi = \{ F_j: j\in J\}$, let $J' = \{ j\in J : \mu|_{F_j}(F_j\setminus S_{F_j})\geq \eps \}.$ Suppose that 
\[ \mu(\bigcup_{j\in J'} F_j) = \sum_{j\in J'} \mu(F_j) \geq \eps.\]
Then
\begin{align}
	\mu(Z) & = \sum_{j\in J} \mu(Z \cap F_j) = \sum_{j\in J} \mu(F_j) \mu|_{F_j}(Z) \nonumber \\
			& \geq \sum_{j\in J'} \mu(F_j) \mu|_{F_j}(Z) \geq \sum_{j\in J'} \mu(F_j) \eps  \geq \eps^2. \nonumber
\end{align}
This contradicts the fact that $\mu(Z)<\eps^2$. Therefore, for $\eps$-almost every $F\in \bigvee_{k=m_1}^{m_2} f^k \xi$, namely for $\{F_j: j\notin J'\}$, $\mu|_F(F\setminus S_F)< \eps$. The sets $S_F$ are the ones we want.
\end{proof}

%
\subsection{Step 4: Local definition of \texorpdfstring{$\psi$}{psi}}\label{subsec:4}

We now turn to the construction of $\psi$ satisfying the conditions of Lemma \ref{lemma:psi}. At this point, the argument entirely follows \cite[\S3]{OW} or \cite[\S6.2]{CH}, so we summarize the main points here, clarifying how their argument unfolds in our setting.

We begin with Step (4), defining $\psi_i: S_A \cap R_i \to R_i$ for each rectangle from our auxiliary partition for which $\mu(S_A \cap R)>0$. These $\psi_i$ will be combined to form $\psi$ in Step (5).

Consider those rectangles $R_i$ in $\pi$ ($i=1, \ldots, n$), such that $S_A\cap R_i$ has positive measure, where $S_A$ is given by Proposition \ref{prop:layerwise}. Recall that for any $y\in S_A\cap R_i$, $V_{x_i}^{u,R_i} \subset V_{x_i}^{u, A}$; that is, the connected component of $y$'s unstable leaf in $A$ stretches fully across $R_i$.

For each such $i$, pick a point $x_i \in S_A\cap R_i$.  Consider $V_{x_i}^{cs,S_A\cap R_i}$ and $V_{x_i}^{cs, R_i}$. These can each be equipped with factor measures $\mu^{cs}_{x_i}|_{V_{x_i}^{cs,S_A\cap R_i}}$ and $\mu^{cs}_{x_i}|_{V_{x_i}^{cs, R_i}}$ from the local product structure on $R_i$. Each of these is a non-atomic, Lebesgue, probability measure space (see the Remark just before Proposition \ref{prop: integral_limit}), so there is a bijective, measure-preserving map 
\[ \psi_i: \left(V_{x_i}^{cs,S_A\cap R_i}, \mu^{cs}_{x_i}|_{V_{x_i}^{cs,S_A\cap R_i}}\right) \to \left(V_{x_i}^{cs, R_i},\mu^{cs}_{x_i}|_{V_{x_i}^{cs, R_i}}\right).\]
Define $\psi$ on $V_{x_i}^{cs,S_A\cap R_i}$ by $\psi_i$.

Let $y\in S_A\cap R_i$. Define $\psi_i(y)$ with this composition:
\[y \mapsto  W_y^u \cap W_{x_i}^{cs} \mapsto \psi_i(W_y^u \cap W_{x_i}^{cs}) \mapsto W^u_{\psi_i(W_y^u \cap W_{x_i}^{cs})} \cap W_y^{cs}.\]
That is, map $y$ along the strong unstable to the weak stable of $x_i$. Since $R_i$ is a rectangle, $W_y^u \cap W_{x_i}^{cs}$ is still in $R_i$, and since $y$ is in $S_A$, $W_y^u \cap W_{x_i}^{cs}$ is still in $S_A$. Hence, $\psi_i(W_y^u \cap W_{x_i}^{cs})$ is defined. Then, using the rectangular structure of $R_i$ once again, we push $\psi_i(W_y^u \cap W_{x_i}^{cs})$ back to the weak stable of $y$. (See Figure \ref{fig:psi i}.)

\begin{figure}[h]
\centering
\begin{tikzpicture}[scale=1]

\draw[thick] (-2,-2) -- (6,-2) -- (6,4) -- (-2,4) -- (-2,-2);
\draw[thick] (-2,4) -- (0,6) -- (8,6) -- (6,4);
\draw[thick] (8,6) -- (8,0) -- (6,-2);
\draw[thick, dashed] (-2,-2) -- (0,0) -- (0,6);
\draw[thick, dashed] (0,0) -- (8,0);

\fill[blue, opacity=.3] (-2,1) -- (0,3) -- (8,3) -- (6,1);
\draw[thick, dashed, blue] (-2,1) -- (0,3) -- (8,3) -- (6,1) -- (-2,1);

\fill[cyan, opacity=.3] (-2,2.7) -- (0,4.7) -- (8,4.7) -- (6,2.7);
\draw[thick, dashed, cyan] (-2,2.7) -- (0,4.7) -- (8,4.7) -- (6,2.7) -- (-2,2.7);

\fill[blue, opacity=.6] (3,2) ellipse (1.5cm and .5cm);
\draw[thick, red] (3,2) ellipse (1.5cm and .5cm);
\draw[thick, red] (3,5) ellipse (1.5cm and .5cm);
\draw[thick, red] (3,-1) ellipse (1.5cm and .5cm);
\draw[thick, red] (1.5,-2.5) -- (1.5,6.5) ;
\draw[thick, red] (4.5,-2.5) -- (4.5,6.5) ;

\draw[thick, green!80!black, dashed] (3,-1) -- (3,5) ;
\draw[thick, green!80!black, dashed] (2,-1.2) -- (2,4.8) ;
\draw[thick, green!80!black, dashed] (-.5,-1.7) -- (-.5,4.3) ;

\fill (3,2) circle(.06);
\fill (2,1.8) circle(.06);
\fill (-.5,1.3) circle(.06);
\fill (-.5,3) circle(.06);
\fill (2,3.5) circle(.06);

\node at (3.3,2.1) {$x_i$};
\node[red] at (4.8,6.4) {$S_A$};
\node[cyan] at (8.4,4.5) {$W_y^{cs}$};
\node[blue] at (8.4,3) {$W_{x_i}^{cs}$};
\node at (8.4,6.4) {$R_i$};
\node[green!80!black] at (3.4,-1) {$W_{x_i}^{u}$};
\node[green!80!black] at (2.4,-1.2) {$W_{y}^{u}$};
\node[green!80!black] at (-.1,-1.5) {$W_{\psi y}^{u}$};
\node at (2.2,3.6) {$y$};
\node at (2.2,2) {$y'$};
\node at (-.2,1.5) {$y''$};
\node at (-.2,3.3) {$\psi y$};

\draw[thick, gray, ->] (-5,0) -- (-5,3) ;
\draw[thick, gray, ->] (-5,0) -- (-3.5,0) ;
\draw[thick, gray, ->] (-5,0) -- (-6,-1) ;

\node[gray] at (-5, 3.5) {$u$} ;
\node[gray] at (-6.3, -1.3) {$s$} ;
\node[gray] at (-3, 0) {$c$} ;

\end{tikzpicture}
\caption{The definition of $\psi_i$. $y' = W_y^u \cap W_{x_i}^{cs}$, $y''=\psi_i(W_y^u \cap W_{x_i}^{cs})$. $\psi_i$ spreads out the dark blue region to the light blue region in a (conditional) measure-preserving way.}\label{fig:psi i}
\end{figure}
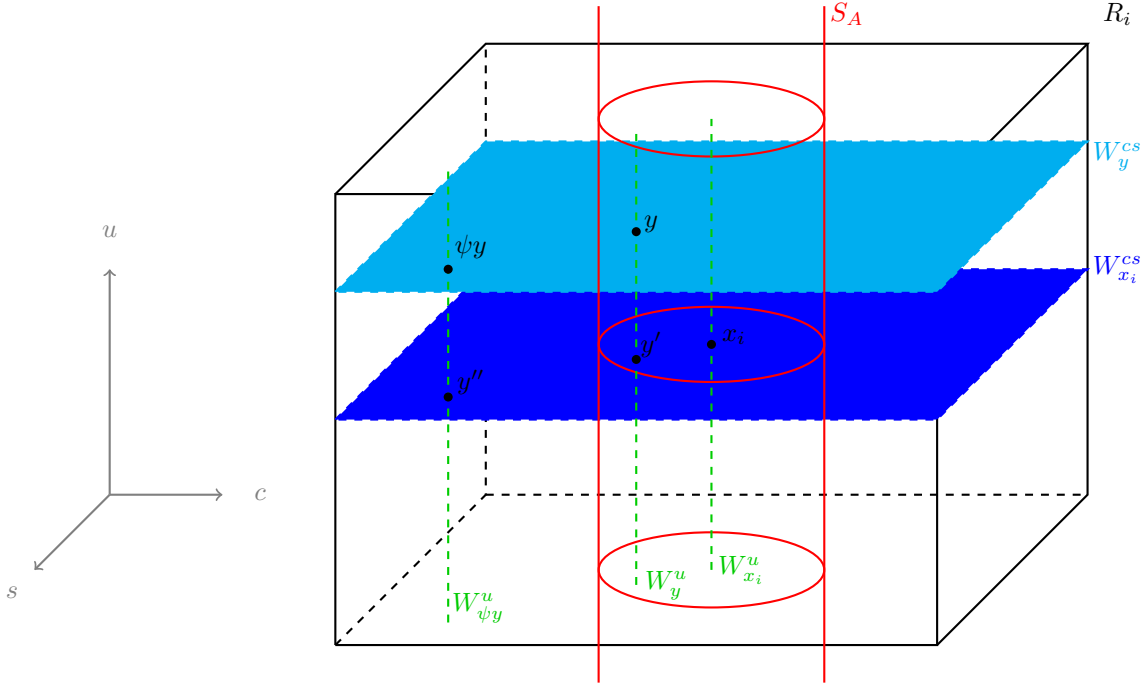

The construction of $\psi_i$ involves translations along the strong unstable factor of the product composed with a measure preserving map in the weak stable leaf. Therefore, $\psi_i:(S_A\cap R_i, \mu^p_{R_i}|_{S_A}) \to (R_i, \mu^p_{R_i})$ is measure-preserving for the (correctly conditioned) product measures. This completes Step (4).

%

\subsection{Step 5: Completing the argument}\label{subsec:5}

Turning to Step (5), recall that by Step (1) (i.e., Proposition \ref{prop:good boundary}) $\xi$ is a partition of $(\GS,\mu)$ with $\mu(B(\partial \alpha, r))< Cr$ for any $r>0$. Let $\eps$ be given, and set $\delta=\eps^4$. Recall the partition $\pi=\{R_0, R_1, \ldots, R_k\}$ where $R_i$ ($i=1, \ldots, k$) are the rectangles provided by Step (2) and $\mu(R_0)<\delta$. For $i=1, \ldots, k$, let $G_i\subset R_i$ be the set described in Step (2).

By the $K$-property for $(\GS, \mu, g_t)$ (\cite[Theorem A]{CCESW}), there exists an integer $m$ such that, setting $N=2m$, for any integers $N_1> N_0>N$, $\delta$-a.e. atom $A$ of $\bigvee_{N_0-m}^{N_1-m} f^i \alpha$ satisfies, for every $R\in \pi$,
\begin{equation}\label{eqn:rectangle measure}
    \left| \frac{\mu|_A(R)}{\mu(R)} - 1 \right| < \delta.   
\end{equation}

By Proposition \ref{prop:layerwise} (with $\epsilon$ in that Proposition taken to be our current $\delta$), we can also take $N=2m$ so large that whenever $N_1>N_0>N$, $\delta$-a.e. atom $A$ of $\bigvee_{N_0-m}^{N_1-m} f^i \alpha$ contains a set $S_A\subset A$ with $\mu|_A(S_A)>1-\delta$ intersecting $R$ layerwise.

Let $m, N, N_0, N_1$ as above and $n>0$ be given. Let $\omega = \bigvee_{i=N_0-m}^{N_1-m} f^i\alpha$. We want to prove, for $c\eps$-a.e. atom $A$ of $\omega$ that 
\[\bar d\left( \{ f^{-i}\alpha\}_{i=1+m}^{n+m}, \{ f^{-i}\alpha|_A\}_{i=1+m}^{n+m}\right)<c\eps\]
holds.

At this point, we can state the definition of $\psi$. For $A\in \omega$, if $y\in S_A\cap R_i$ ($i=1, \ldots, n$) where $\mu(S_A\cap R_i)>0$, set $\psi(y) = \psi_i(y)$. For all other $y$, set $\psi(y) = y$.

It remains to verify the conditions of Lemma \ref{lemma:psi} for $c\eps$-almost every $A\in\bigvee_{k=N_0}^{N_1}f^k\xi$. We begin by identifying the `bad' atoms in $\bigvee_{k=N_0}^{N_1}f^k\xi$, and ensuring that their total measure is less than $c\eps$.

Let 
\[\hat F_1 = \bigcup_{A\in \omega : A \mbox{ \emph{\scriptsize{fails}} }   \eqref{eqn:rectangle measure}} A.\]
By our choice of $N$, $\mu(\hat F_1)<\delta.$

Let $F_2 = \bigcup_{i=1}^k (R_i \setminus G_i).$ By Step (2), $\mu(F_2)<\delta$. In addition, a brief computation using facts from Step (2) implies
\begin{equation}\label{eqn:sum bound}
\sum_{i=1}^k \mu^p_{R_i}(F_2\cap R_i)  < c \delta
\end{equation}
for some universal $c$. Let 
\[\hat F_2 = \bigcup_{A\in \omega: \mbox{\emph{\scriptsize{ either }}} \mu|_A(F_2)>\delta^{1/2} \mbox{\emph{\scriptsize{ or }}} \sum_{i=1}^k\frac{\mu^p_{R_i}(A\cap F_2)}{\mu(A)}>\delta^{1/2}}A.\]
Some computation shows that $\mu(\hat F_2)<c\delta^{1/2}$ for some $c$. We provide part of the argument here, as it is illustrative of a general strategy for bounding the measures $\mu(\hat F_i)$ throughout (see also \cite{CH}). Recall that $\mu(F_2)<\delta.$ Then
$\mu|_A(F_2):=\frac{\mu(F_2\cap A)}{\mu(A)}>\delta^{1/2}$
implies that for such $A$, $\mu(A)<\delta^{-1/2}\mu(F_2\cap A).$
Hence,
\[ \mu\left( \bigcup_{A\in \omega: \mu|_A(F_2)>\delta^{1/2}} A \right) < \sum_{\mbox{\scriptsize{\emph{all }}} A\in \omega} \delta^{-1/2}\mu(F_2\cap A) = \delta^{-1/2} \mu(F_2) < \delta^{1/2}.
\]
A similar computation using \eqref{eqn:sum bound} bounds the other part of the union for $\hat F_2$.

Let 
\[F_3 = \{ x\in \GS\setminus R_0 :  x\notin S_A\}\]
where $S_A$ is provided by applying Proposition \ref{prop:layerwise} to $A$ and the collection $\mathcal{R}=\{R_1, \ldots, R_k \}$.
Let
\[\hat F_3 = \bigcup_{A\in \omega: \mu|_A(F_3)>\delta^{1/2}} A.\]
By similar arguments, using the fact that $\mu|_A(S_A)>1-\delta$, we find that $\mu(\hat F_3)<c \delta^{1/2}$ for some $c$. 

Finally, let 
\[F_4 = \{ x\in \GS\setminus R_0 : \exists y\in V_x^{cs,\pi(x)} \mbox{ such that } h(x,y) > \delta^{1/2} \}.\]
If $x\in F_4$, then $x$ and $y$, two points on the same weak stable which are $<\diam \pi(x)< \delta/C$ apart are, with relative frequency $>\delta^{1/2}$ over the integers $1+m \leq i \leq n+m$, in different atoms of $f^{-i}\alpha$. Hence, for all $x\in F_4$,
\[\#\{ 1+m \leq i \leq n+m : x\in B((\partial f^{-i}\alpha,\delta/C)) \} > n\delta^{1/2}.\]
That is, points in $F_4$ are in at least $n\delta^{1/2}$ of the sets $\{ f^{-i}B(\partial \alpha, \delta/C)\}_{i=1+m}^{n+m}$, a sequence of $n$ sets each with measure at most $C \delta/C = \delta$ by our condition on the measure of neighborhoods of the boundary of $\alpha$. Therefore, some $f^{-i^*}B(\partial \alpha, \delta/C)$ satisfies
\[ \delta > \mu(f^{-i^*}B(\partial \alpha, \delta/C)) \geq \mu(F_4 \cap f^{-i^*}B(\partial \alpha, \delta/C)) \geq \frac{n \delta^{1/2}}{n} \mu(F_4) = \delta^{1/2} \mu(F_4).\]
Hence, $\mu(F_4)<\delta^{1/2}.$ Let 
\[ \hat F_4 =\bigcup_{A\in \omega: \mu|_A(F_4)>\delta^{1/4}} A.\]
Again, by a similar argument to those above, $\mu(\hat F_4)<\delta^{1/4}$.

Assembling the full collection $\hat F= \hat F_1 \cup \hat F_2 \cup \hat F_3 \cup \hat F_4$ of `bad' atoms, for some constant $c$ we have $\mu(\hat F)< c\delta^{1/4}=c\eps$, as desired.

We now describe the construction of the sets $E_1$ and $E_2$ in Lemma \ref{lemma:psi}. Consider an atom A of $\omega$ in the complement of $\hat F$. Recall, by the definition of $E_1$ that for $x\in E_1$, $h(x,\psi x) > \delta^{1/4} = \varepsilon$, so $x \in F_4$, and therefore $E_1 \subset F_4$. Thus, since $A \in \hat F_4^c$, $\mu|_A(E_1) < \delta^{1/4} = \varepsilon
$ as required.

To construct $E_2$, we will consider the union of various `bad' subsets of the atom $A$. We first consider $A \cap R_0$; since $A \in \hat F_1^c$ and $\mu(R_0) < \delta$, we obtain $\mu|_A(R_0) < c\delta$. We next consider the sets $A \cap F_2$ and $A \cap F_3$; as $A$ is in the complement of both $\hat F_2$ and $\hat F_3$ we readily obtain that $\mu|_A(F_2) < \delta^{1/2}$ and $\mu|_A(F_3) < \delta^{1/2}$.

The next sets consist of all the rectangles in which $A$ contains too large a proportion of either $F_2$ or $F_3$. We denote by $D_2$ the set that contains too large a proportion of $F_2$, measured either with respect to the product measure or the invariant measure.

$$D_2 = \bigcup_{R_i : \text{ either } \mu|_{A\cap R_i}(F_2) > \delta^{1/4} \text{ or } \frac{\mu^p_{R_i}(A\cap F_2)}{\mu(A \cap R_i)} > \delta^{1/4}}R_i.$$
Using the fact that $A$ is in the complement of $\hat F_2$, some computation, similar to the arguments given for bounding the measures $\mu(\hat F_i)$, shows that $\mu|_A(D_2) < c\delta^{1/4}$. Similarly we define

    $$D_3 = \bigcup_{R_i : \mu|_{A\cap R_i}(F_3) > \delta^{1/4}}R_i,$$
and again, using the fact that $A$ is in the complement of $\hat F_3$, some computation shows that $\mu|_A(D_3) < c\delta^{1/4}$.

The final bad subset is $\psi^{-1}(F_2)$. As we have already considered the intersection of $A$ with $R_0$, $F_2$, $F_3$, and $D_2$ we will only need to estimate the measure of the part of the subset that lives in the complement of these sets. Following the argument given in \cite[\S6.2]{CH}, which estimates this measure for a particular $R_i$ in the complement of $R_0 \cup D_2$ and the sums over all rectangles, we find that
$$\mu|_A(\psi^{-1}(F_2) \cap (R_0 \cup F_2 \cup F_3 \cup D_2)^c) \leq c\delta.$$
We then define
$$E_2 = A\cap(R_0 \cup F_2 \cup F_3 \cup D_2 \cup D_3 \cup \psi^{-1}F_2).$$

Combining all the above estimates we see $\mu|_A(E_2) < c\delta^{1/4} = c\varepsilon$.

We need the following fact, proven in section 6 of \cite{CH}, to complete the proof that $g_t$ is very weak Bernoulli. 
\begin{proposition}[Section 6, \cite{CH}]\label{prop:good measures}
    If $A$ is an atom with $A\subset \hat{F_2}^c\cap\hat{F_3}^c$ and $B\subset (A\cap R_i\cap E_2)^c$ for some rectangle $R_i$, then
    $$\left\lvert \frac{\mu|_{A\cap R_i}(B)}{\mu|_{R_i}(\psi B)} -1 \right\rvert <c\delta^{1/4}. $$
\end{proposition}

We now show how to use the above proposition to satisfy the conditions of Lemma \ref{lemma:psi}. Consider a set $B\subset A\setminus E_2$. We will use Proposition \ref{prop:good measures} to bound 
$\left\lvert \frac{\mu|_A(B)}{\mu(\psi B)} -1\right\rvert.$

First, we overestimate the term by breaking the sum up over rectangles contained in $(R_0\cup D_2 \cup D_3)^c$:

\begin{align*}
    \left\lvert \frac{\mu|_A(B)}{\mu(\psi B)} -1\right\rvert & = \frac{1}{\mu(\psi B)}\big\lvert \mu|_A (B) - \mu (\psi B) \big\rvert \\
    & \leq \sum_{R_i \subset (R_0\cup D_2 \cup D_3)^c} \mu|_{\psi B }(R_i)
    \left\lvert 
   \underbrace{ \frac{\mu|_{A\cap R_i} (B)}{\mu|_{R_i}(\psi B )}}_{1st} \cdot \underbrace{\frac{\mu|_A (R_i)}{\mu(R_i)}}_{2nd} - 1
    \right\rvert
\end{align*}
We know that for the second term, 
\begin{align*}\label{ineq1_step5}
    \left\lvert \frac{\mu|_{A} (R_i)}{\mu(R_i)} - 1\right\rvert <\delta
\end{align*}
since $\hat{F_1}$ is defined as the set where (\ref{eqn:rectangle measure}) fails and $A$ is in the complement of $\hat{F_1}$.
As for the first term, directly from Proposition \ref{prop:good measures}, we know that 
$$\left\lvert\frac{\mu|_{A \cap R_i} (B)}{\mu|_{R_i}(\psi B)} - 1\right\rvert < c\delta^{1/4}.$$
Combining these,
\[ 1-(2c+1)\delta^{1/4}<(1-c\delta^{1/4})(1-\delta) <
   \frac{\mu|_{A\cap R_i} (B)}{\mu|_{R_i}(\psi B )} \cdot \frac{\mu|_A (R_i)}{\mu(R_i)} < (1+c\delta^{1/4})(1+\delta) < 1-(2c+1)\delta^{1/4},
\]
since $\delta$ is small. Finally,
\begin{align*}
    \left\lvert \frac{\mu|_A(B)}{\mu(\psi B)} -1\right\rvert & 
    \leq \sum_{R_i \subset (R_0\cup D_2 \cup D_3)^c} \mu|_{\psi{B}}(R_i)
    \left\lvert 
   \frac{\mu|_{A\cap R_i} (B)}{\mu|_{R_i}(\psi(B))} \cdot \frac{\mu|_A (R_i)}{\mu(R_i)} - 1
    \right\rvert\\
    & < (2c+1)\delta^{1/4}\sum_{R_i \subset (R_0\cup D_2 \cup D_3)^c} \mu|_{\psi{B}}(R_i)
    \\
    & \leq (2c+1) \delta^\frac{1}{4}.
\end{align*}
Recalling that $\varepsilon=\delta^\frac{1}{4}$, we have completed our verification of condition (2) of Lemma \ref{lemma:psi}. Having satisfied the conditions of Lemma \ref{lemma:psi}, we now have that $$\bar d\left( \{ f^{-i}\alpha\}_{i=1+m}^{n+m}, \{ f^{-i}\alpha|_A\}_{i=1+m}^{n+m}\right)<c'\eps,$$
for some $c'$ depending only on $\GS$, as desired.

\appendix
\section{Local product structure of equlibrium states for geodesic flows on rank 1 nonpositively curved manifolds}\label{section: Appendix}

Consider a closed, connected, $C^\infty$ Riemannian manifold $(M,\sigma)$ of rank 1 with nonpositive sectional curvature and dimension $m_0$. Let $(g_t)_{t\in\mathbb R}$ be the geodesic flow on the unit tangent bundle $T^1M$. We recall that there are two continuous, invariant subbundles $E^s$ and $E^u$ of $TT^1M$, each of dimension $m_0-1$, which are orthogonal to the flow direction $E^c$ in the natural Sasaki metric. In negative curvature, the geodesic flow is Anosov with Anosov splitting $TT^1M=E^s\oplus E^c\oplus E^u$ with uniform contraction and expansion on $E^s$ and $E^u$, respectively. In nonpositive curvature, $E^s_v$ and $E^u_v$ may intersect nontrivially, and the rank of a vector $v\in T^1M$ is defined to be $1+\dim (E^s_v\cap E^u_v)$. Then, $M$ has rank $1$ if the minimum rank over all vectors in $T^1M$ is equal to $1$. The set of vectors whose rank is larger than $1$ will be called the singular set, and denoted by $\Sing$, while the set of rank $1$ vectors will be denoted $\Reg$.

It was shown in \cite[Theorem A and B]{BCFT} that if a H\"older continuous function $\phi\colon T^1M\rightarrow \mathbb R$ is locally constant on a neighborhood of $\Sing$ then there exists a unique equilibrium state $\mu$ for the potential $\phi$. We will refine this result by showing that $\mu$ has a local product structure (see Definition~\ref{defn:lps} in combination with Definition~\ref{defn: flow box nonpositive}). This, in conjunction with the results of \cite{CT19}, give an alternate proof of the Bernoulli property for these equilibrium states.

%
\subsection{Geometric preliminaries}\label{sec:geometric preliminaries}

Let $(M,\sigma)$ be as above. For each tangent vector $v\in T^1M$, there exists a unique constant speed geodesic denoted by $\gamma_v$ with $\dot\gamma_v(0)=v$. The geodesic flow is defined by $g_t(v)=\dot\gamma_v(t)$. 

Let $d$ be the distance function on $M$. The Knieper metric, $d_K$, on $T^1M$ is defined as follows \cite{Knieper}:
$$d_K(v,w)=\max\{d(\gamma_v(t),\gamma_w(t))\,|\, t\in[0,1]\}.$$ The Knieper metric is uniformly equivalent to the Sasaki metric on $TM$ which is the lift of the Riemannian metric on $M$.

The geodesic flow generates natural invariant foliations that we now define. First, we describe precisely the $(g_t)_{t \in\mathbb R}$-invariant subbundles $E^c$, $E^s$, and $E^u$ of $TT^1M$, where the bundle $E^c$ is spanned by the vector field that generates the geodesic flow.

Let $\mathcal J(\gamma)$ be the space of orthogonal Jacobi fields for a geodesic $\gamma$. Then, $J\in \mathcal J(\gamma)$ is called stable if $\|J(t)\|$ is bounded for $t\geq 0$, and unstable if it is bounded for $t\leq 0$. For each $v\in T^1M$, there is a natural isomorphism between $TT^1M$ and $\mathcal J(\gamma_v)$ so that if $\xi\mapsto J_{\xi}$ then $$\|dg_t(\xi)\|^2=\|J_\xi(t)\|^2+\|J'_\xi(t)\|^2.$$
The stable and unstable subbundles of $TT^1M$ are defined as 
$$E^s(v)=\left\{\xi\in T_vT^1M\,|\, J_\xi\text{ is stable}\right\}\text{ and }E^u(v)=\left\{\xi\in T_vT^1M\,|\, J_\xi\text{ is unstable}\right\},$$
respectively. We denote $E^{cs}=E^c\oplus E^s$. By \cite{Eberlein}, $E^u$, $E^s$, and $E^{cs}$ are integrable to the foliations $W^u$ (unstable), $W^s$ (stable), and $W^{cs}$ (center stable), respectively. For more details and properties of the bundles and foliations, see \cite{Eberlein} and \cite{Ballmann82}.

For each $v\in T^1M$, we can consider the intrinsic metrics on the leafs of the defined foliations. The intrinsic metric on $W^s(v)$ is given by
\begin{equation*}
    d^s(u,w)=\inf\{\ell(\pi\zeta)\,|\, \zeta\colon [0,1]\rightarrow W^s(v)\text{ is } C^1\text{ map}, \, \zeta(0)=u,\,\zeta(1)=w\},
\end{equation*}
where $\pi\colon T^1M\rightarrow M$ is a natural projection and $\ell$ is the the length of the curve in $M$. Similarly, we define the intrinsic metric $d^u$ on $W^u(v)$. The intrinsic metric on $W^{cs}(v)$ is given locally by $$d^{cs}(u,w)=|t|+d^s(g_tu,w),$$
where $t$ is the unique value such that $g_tu\in W^s(w)$. The local definition extends to the whole leaf $W^{cs}(v)$. The local stable leaf through $v$ of size $\eps$ is 
$$W^s(v,\eps)=\left\{w\in W^s(v)\,|\, d^s(v,w)\leq\eps\right\}.$$
Similarly, we can define the local unstable leaf $W^u(v,\eps)$ and the local central stable leaf $W^{cs}(v,\eps)$.

%
\subsection{Bowen Balls, rectangles, and Bowen brackets}

As before, we introduce the notion of stable and unstable Bowen balls.

\begin{definition}
Given $x\in T^1M$ and $\eps, t > 0$, the unstable Bowen ball is 
\begin{align}
    B^u_t(x,\eps) :&= \{y\in W^u(x,\eps) \mid d^u(g_rx,g_ry) \leq \eps \mbox{ for } 0\leq r \leq t\} \nonumber \\
        &= \{y\in W^u(x,\eps) \mid d^u(g_rx,g_ry) \leq \eps \mbox{ for } -\infty < r \leq t\}.\nonumber 
\end{align}
The stable Bowen ball is defined similarly, with $W^s(x,\eps)$ replacing $W^u(x,\eps)$ and taking $-t \leq r < \infty$.
\end{definition}

We will also use two-sided Bowen balls, defined as follows

\begin{definition}
Given $x\in T^1M$ and $\eps, t_1,t_2 > 0$, the two-sided Bowen ball is given by
$$B_{[-t_1,t_2]}(x,\eps) := \{y\in T^1M \mid d_K(g_rx,g_ry) \leq \eps \text{ for } -t_1 \leq r \leq t_2\}.$$
We will also sometimes write $B_{[0,t]}(x,\eps) = B_t(x,\eps)$ in the interest of simplification of notation.
\end{definition}

As is already evident, we will often be forced to work with multiple metrics at once. Luckily, we have direct relationships between them.
\begin{lemma}\cite[Lemma 2.13]{BCFT}\label{lemma: metric comparison}
There exists $\Lambda > 0$ such that for all $v\in T^1M$, $w\in W^{cs}(v)$ and $z\in W^{u}(v)$, we have
$$d_K(v,w) \leq d^{cs}(v,w) \quad \text{ and }\quad d_K(v,z) \leq e^{\Lambda}d^u(v,z).$$
\end{lemma}

In \cite[Definition 2.10]{BCFT}, the authors defined a continuous function $\lambda: T^1M \to [0,\infty)$ which played an essential role in studying the dynamics of the geodesic flow. For the details of the definition of $\lambda$ and its properties, see \cite[Sections 2.4.2 and 3.2]{BCFT}. 

Everywhere below, we will use that the complement of $\Sing$ is $\Reg$. Similarly to Definitions~\ref{defn: regular} and \ref{defn: good start and end}, we have

\begin{definition}\label{defn: regular non-positive}
Let $c > 0$. We define $$\operatorname{Reg}(c) := \{x\in T^1M\mid \lambda(x) \geq c\}.$$    
and
$$\CCC(c) = \{(x,t) \in T^1M \times \mathbb{R} \mid \lambda(x) \geq c \text{ and } \lambda(g_tx) \geq c \}.$$
\end{definition}

In order to show product structure of the measure, we first recall the topological local product structure.

\begin{lemma}\cite[Lemma 4.4]{BCFT}\label{lemma: topological LPS}
For every $\eta > 0$, there exists $\rho > 0$, $\chi \geq 1$ such that for all $v\in \Reg(\eta)$, all $\eps\in (0,\rho]$, and all $w_1,w_2\in B(v,\eps)$, the intersection $[w_1,w_2] := W^{cs}(w_1,\chi\eps)\cap W^{u}(w_2,\chi\eps)$ is unique. Furthermore, 
$$d^{cs}(w_1,[w_1,w_2]) \leq \chi d_{K}(w_1,w_2),$$
$$d^{u}(w_2,[w_1,w_2]) \leq \chi d_{K}(w_1,w_2).$$
\end{lemma}

Note that from Lemma \ref{lemma: topological LPS} we have
\[d^{cs}(v,w)=d^{cs}(v,[v,w])\leq \chi d_K(v,w) \ \mbox{ and }\  d^u(v,z)=d^u(v,[z,v])\leq \chi d_K(v,z)\]
when $d_K(v,w)$ and $d_K(v,z)<\rho$.

With Lemma \ref{lemma: topological LPS}, we can extend the bracket operation to be defined on sets when appropriate. Consequently, we have the following result comparing the bracket of stable and unstable Bowen balls with two-sided Bowen balls.

\begin{lemma}\label{lemma: NPC Bowen balls containment}
Let $\eta > 0$, and let $\rho,\chi$ be as in Lemma \ref{lemma: topological LPS}. Let $v\in \Reg(\eta)$, $\eps\in (0,\frac{\rho}{4e^{\Lambda}}]$, and $s \in (\eps,2\eps)$. Then, there exists $T_i := T_i(\chi) > 0$ for $i=1,2$ such that for all $x,y\in B(v,\eps)$ and $n,m\geq 0$ with $\lambda(g_{-m}y),\lambda(g_{n}x) \geq \eta$, we have
$$B_{[-m,n]}([x,y],T_1\eps) \subset g_{[-s,s]}[B_{n}^{u}(x,\eps),B_m^{s}(y,\eps)] \subset B_{[-m,n]}([x,y],T_2\eps + (e^\Lambda\chi^2+1) s).$$
\end{lemma}

\begin{proof}
We begin with the second inclusion. First, as $\eps \leq \frac{\rho}{4}$, we know that $[B_{n}^{u}(x,\eps),B_m^{s}(y,\eps)]$ is well-defined. Let $w_1\in B_n^u(x,\eps)$, $w_2\in B_m^{s}(y,\eps)$, and $r\in [-s,s]$. 

For $0\leq t\leq n,$
\begin{align}
    d_K(g_tg_r[w_1,w_2],g_t[x,y]) & \leq d_K(g_t[w_1,g_rw_2],g_tw_1)+d_K(g_tw_1,g_tx)+d_K(g_tx,g_t[x,y]) \nonumber \\
    & \leq d_K([w_1,g_rw_2],w_1)+e^\Lambda d^u(g_tw_1,g_tx)+d_K(x,[x,y]) \nonumber \\
    & \leq d^{cs}([w_1,g_rw_2],w_1)+e^\Lambda d^u(g_tw_1,g_tx)+d^{cs}(x,[x,y]) \nonumber \\
    & \leq \chi d_K(w_1,g_rw_2)+e^\Lambda\eps+\chi d_K(x,y) \nonumber \\
    & \leq \chi [d_K(w_1, w_2)+d_K(w_2, g_rw_2)] + e^\Lambda\eps + \chi 2 \eps \nonumber \\
    & \leq \chi [d_K(w_1, x) + d_K(x,y) + d_K(y,w_2)] + \chi s + e^\Lambda\eps + \chi 2 \eps \nonumber \\
    & \leq (5+e^\Lambda + e^\Lambda \chi^{-1})\chi \eps + \chi s. \nonumber
\end{align}
For $-m\leq t \leq 0,$
\begin{align}
    d_K(g_tg_r[w_1,w_2],g_t[x,y]) & \leq d_K(g_t[w_1,g_rw_2],g_tg_rw_2)+ d_K(g_tg_rw_2,g_ty)+d_K(g_ty,g_t[x,y]) \nonumber \\
    & \leq d_K([w_1,g_rw_2],g_rw_2)+d_K(g_tg_rw_2,g_tw_2)+d_K(g_tw_2,g_ty)+d_K(y,[x,y]) \nonumber \\
    & \leq d_K([w_1,g_rw_2],g_rw_2)+s+d_K(w_2,y)+ e^\Lambda\chi d_K(x,y) \nonumber \\
    & \leq e^\Lambda \chi d_K(w_1,g_rw_2) + s+ e^\Lambda \eps + e^\Lambda \chi 2 \eps \nonumber \\
    & \leq e^\Lambda \chi [\eps(3+e^\Lambda)+\chi s] + s + e^\Lambda\eps + e^\Lambda \chi 2\eps \nonumber \\
    & = [5e^\Lambda\chi+e^{2\Lambda}\chi + e^\Lambda] \eps + (e^\Lambda\chi^2+1)s. \nonumber
\end{align}
Letting $T_2 = (5e^\Lambda+e^{2\Lambda}+e^\Lambda\chi^{-1})\chi$ we get
\[g_{[-s,s]}[B_{n}^{u}(x,\eps),B_m^{s}(y,\eps)] \subset B_{[-m,n]}([x,y],T_2\eps + (e^\Lambda\chi^2+1) s).\]
Note that if $r=0$, the terms involving $s$ are unnecessary. That is, 
\[ [B^u_n(x,\eps),B^s_m(y,\eps)] \subseteq B_{[-m,n]}([x,y],T_2\eps). \]

We now turn to the first inclusion. Let $w \in B_{[-m,n]}([x,y],T_1\eps)$ with $T_1$ to be determined. We claim that $[y,w]\in B_m^{cs}(y,\eps)$ and $[w,x]\in B_n^u(x,\eps).$ Given this claim, $w=[[w,x],[y,w]] \in g_{[-s,s]}[B_n^u(x,\eps),B_m^s(y,\eps)]$ as desired.

To prove the claim, we start with $[w,x].$ For any $0\leq t\leq n,$ $d_K(g_t w, g_t[x,y])\leq T_1\eps.$ Consider $[g_nw, g_ny] = [g_nw, g_n[x,y]].$ Then
    \[d^u(g_n[x,y],[g_nw,g_n[x,y]]) \leq \chi d_K(g_nw, g_n[x,y]) \leq \chi T_1\eps\]
since $w\in B_{[-m,n]}([x,y],T_1\eps)$ for the second inequality. The first inequality uses the fact that $\lambda(g_nx)\geq \eta$ and
\[d_K(g_nx,g_nw) \leq T_1\eps + d_K(g_n[x,y],g_nx) \leq T_1\eps + d_K([x,y],x) \leq T_1\eps + \chi d_K(x,y) \leq \rho.\]
Therefore,
\[[w,y] \in B_n^u([x,y], \chi T_1 \eps).\]
Of course, $x\in B_m^s(x,\eps/T_2)$. Therefore, by the second inclusion, as long as $T_1 < \frac{1}{\chi^2T_2}$,
\[[w,x] = [[w,y],x] \in [B_n^u([x,y],\chi T_1 \eps),B^s_m(x,\eps/T_2)] \subseteq B_{[-m,n]}(x,\eps/\chi).\]
Therefore, $d_K(g_n[w,x],g_nx) < \frac{\eps}{\chi}$, and so $d^u(g_n[w,x],g_nx) < \eps$. Thus, for all $0 \leq t \leq n$, $d^u(g_t[w,x],g_tx)<\eps$. Hence, $[w,x]\in B^u_n(x,\eps)$ as desired.

Finally, we turn to $[y,w]$. For all $-m\leq t \leq 0$, $d_K(g_tw, g_t[x,y])\leq T_1\eps$, for $T_1$ to be determined.  Consider $[g_tx, g_tw] = [g_t[x,y],g_tw].$ Using the fact that $\lambda(g_{-m}y)\geq \eta$,
\[d^{cs}(g_{-m}[x,w],g_{-m}[x,y]) = d^{cs}([g_{-m}[x,y],g_{-m}w],g_{-m}[x,y]) \leq \chi d_K(g_t[x,y],g_tw) \leq \chi T_1\eps.\]
Therefore, for some $r<\chi T_1\eps$,
\[g_r[x,w] \in B^s_m([x,y],\chi T_1 \eps).\]
Applying the second inclusion again,
\begin{align}
    g_r[y,w]=[y,g_r[x,w]] & \in \left[B^u_n\left(y,\frac{\eps}{2T_2}\right),B^s_m([x,y],\chi T_1\eps)\right] \nonumber \\
    &\subseteq B_{[-m,n]}([y,[x,y]],T_2\chi T_1\eps) \nonumber \\
    &\subseteq B_{[-m,n]}\left(y,\frac{\eps}{2\chi}\right) \nonumber
\end{align}
provided $T_1<\frac{1}{2\chi^2 T_2}$.  Note also that with this choice, $r<\frac{\eps}{2}$. Therefore, for all $-m\leq t\leq 0$,
\[d_K(g_tg_r[y,w],g_ty)<\frac{\eps}{2\chi}\] and hence
\[d^{cs}(g_t[y,w],g_ty)<\eps.\]
Therefore $[y,w]\in B^{cs}_m(y,\eps)$ as desired.
\end{proof}

\begin{definition}
Let $x\in T^1M$ with $\lambda(x) = c > 0$, and  take $\rho := \rho(c)$ and $\chi := \chi(c)$ as in Lemma \ref{lemma: topological LPS}. For all $\eps \in \left(0,\frac{\rho}{4\chi(1+\chi)(1+e^{\Lambda})}\right)$, we define the rectangle $R(x,\eps)$ centered at $x$ of radius $\eps$
$$R(x,\eps) := [W^u_{\eps}(x), W^{s}_{\eps}(x)].$$
\end{definition}

This is well-defined, since given $w\in W^u_{\eps}(x)$ and $v\in W^s_{\eps}(x)$, we have
$$d_K(w,x) \leq e^{\Lambda}d^u(w,x) < \frac{\rho}{2},$$
and
$$d_K(v,x) \leq d^{cs}(v,x) = d^s(v,x) < \frac{\rho}{2}.$$

The $d_K$ diameter of $R(x,\eps)$ scales well with $\eps$.

\begin{lemma}\label{lemma: NPC rectangle diameter}
Given $y\in R(x,\eps)$, $d_K(y,x) \leq (1 + \chi)\eps$.
\end{lemma}
\begin{proof}
Let $y = [v_1,v_2]$. Then we have
$$d_K([v_1,v_2],x) \leq d_K([v_1,v_2],v_1) + d_K(v_1,x) \leq d^{cs}([v_1,v_2],v_1) + d_K(v_1,x) < (1+\chi)\eps.$$
\end{proof}

\begin{proposition}\label{prop: NPC disjoint rectangles}
Let $s_0$ be the minimum length of a closed geodesic. Given $R=R(x,\eps)$ as above, with $\lambda(x) = c > 0$, then for all distinct $t,s\in \left(-\frac{s_0}{2},\frac{s_0}{2}\right)$, $g_tR\cap g_sR = \emptyset$.
\end{proposition}

\begin{proof}
It suffices to show that for all $|t|\in (0,s_0)$, $g_tR\cap R = \emptyset$. Assume for contradiction that $y\in g_tR\cap R$. Then $y = [v_1,v_2] = g_t[w_1,w_2]$ for $v_1,w_1\in W^u_{\eps}(x)$ and $v_2,w_2\in W^s_{\eps}(x)$. Thus, we see $\tilde{y}(-\infty) = \tilde{w_2}(-\infty) = \tilde{v_2}(-\infty)$, and $\tilde{v_1}(\infty) = \tilde{w_1}(\infty)$. As $v_2,w_2\in W^s_{\eps}(x)$, this implies that $v_2=w_2$, and similarly, $v_1=w_1$, since $v_1,w_1\in W^u_{\eps}(x)$. However, this implies that $g_{-t}y = y$, a contradiction.
\end{proof}

Using this, we can now define a flow box.

\begin{definition}\label{defn: flow box nonpositive}
Given a rectangle $R(x,\eps)$ as above and $n > 2(s_0)^{-1}$, define a flow box $\mathcal{B}_x(n,\eps) := g_{[-\frac{1}{n},\frac{1}{n}]}R(x,\eps)$.
\end{definition}

By Proposition \ref{prop: NPC disjoint rectangles}, we see that we can uniquely express every $y\in \mathcal{B}_x(n,\eps)$ in the form $g_t[v,w]$.

\begin{proposition}
If we take $\eps$ as above and $\frac{1}{n} < \min\left\{\frac{s_0}{2},\frac{\rho(c)}{2}\right\}$, then $[\cdot,\cdot]$ is defined on $\mathcal{B}_x(n,\eps)$.
\end{proposition}

\begin{proof}
Let $y_1,y_2\in\mathcal{B}_x(n,\eps)$. Then $y_i = g_{t_i}[v_i,w_i]$ for $v_i\in W^u_{\eps}(x)$ and $w_i\in W^s_{\eps}(x)$. Then
$$d_K(y_i,x) \leq |t_i| + d_K([v_i,w_i],x) \leq \frac{\rho}{2} +  (1+\chi)\eps < \rho$$
using Lemma \ref{lemma: NPC rectangle diameter}. Therefore, by Lemma \ref{lemma: topological LPS}, $[y_1,y_2]$ is defined and consists of a unique point.
\end{proof}

Henceforth, we will write $\mathcal{B}(n_0,\eps)$, omitting the basepoint $x$ and fixing a time scale $n_0 > \max\{\frac{2}{s_0},\frac{2}{\rho(\lambda(x))}\}$. We will also fix $\eta > 0$ such that $\lambda(x)\geq \eta$ and assume that $\eps < \frac{\rho(\eta)}{4\chi(\eta)(1+\chi(\eta))(1 + e^{\Lambda})}$, so that we can use all of the previous results. Additionally, we will use the same notations for the local center stable and unstable leaves in $\mathcal B(n_0,\eps)$ as in Definition~\ref{defn: local leaves}.

Going forward, we will sometimes need variations of results that we have already shown, such as Lemma~\ref{lem:ratio of flow boxes}. Unless otherwise noted, the proofs provided for the geodesic flows on flat surfaces work for the geodesic flows on non-positively curved manifolds with minor appropriate changes, e.g., $\GS$ being replaced by $T^1M$ and $\langle\cdot,\cdot\rangle$ by $[\cdot, \cdot]$.

\begin{lemma}\label{lem: NPC ratio of flow boxes}(cf. Lemma~\ref{lem:ratio of flow boxes})
For all $-\frac{s_0}{2} \leq a<b\leq \frac{s_0}{2}$ and  for any $n > 2(s_0)^{-1}$, for any measurable $A\subset V^u_x$ and $B\subset V^s_x$,
$$\frac{\mu\left(g_{\left[-\frac{1}{n},\frac{1}{n}\right]}\left[ A,B\right]\right)}{\mu\left(g_{[a,b]}\left[ A,B\right]\right)} = \frac{2}{n(b-a)}.$$
\end{lemma}


\subsection{The Bowen property}\label{sec: NPC Bowen}

In this section we establish the Bowen property (see Definition~\ref{defn:global bowen} with $\GS$ replaced by $T^1M$ and $d_{GS}$ by $d_K$) for H\"older continuous potentials that are locally constant on $\Sing$. 

\begin{theorem}\label{thm: NPC global Bowen}
Let $M$ be a closed connected $C^\infty$ Riemannian manifold of rank 1 with nonpositive sectional curvature. Assume $\phi\colon T^1M\rightarrow \mathbb R$ is H\"older continuous and is locally constant on a $\delta$-neighborhood, $B(\Sing,\delta)$, of $\Sing$ with respect to $d_K$ metric for some $\delta>0$. Then  $\phi$ has the global Bowen property.
\end{theorem}

\begin{proof}
By Proposition \ref{proposition: global Bowen property}, using the fact that the setting of \cite{BCFT} is a $\lambda$-decomposition, it suffices to show that $\phi$ has the Bowen property on $\CCC(\eta) := \{(v,t)\in T^1M \mid \lambda(v),\lambda(g_tv) \geq \eta\}.$
Thus, this theorem will follow from Proposition~\ref{prop: Holder to Bowen} and Lemma~\ref{lem: NPC Bowen} below.
\end{proof}

\begin{corollary}\label{cor: NPC Gibbs}
Let $\eta > 0$. There exists $\bar\delta := \bar\delta(\eta) \in \left(0,\frac{s_0}{4}\right)$ such that the following holds. Let $v\in \Reg(\eta)$, $\eps\in (0,\bar\delta]$ and $x,y\in B(v,\eps)$. Then for all $m_1,m_2\geq 0$ with $\lambda(g_{m_1}x)\geq\eta$, $\lambda(g_{-m_2}y)\geq \eta$, and $-2\eps<a<b<2\eps$, there exists $Q_1 := Q_1(\eps) > 0$ and $Q_2 := Q_2(\eps, \min\{\lambda(g_{m_1}[x,y]),\lambda(g_{-m_2}[x,y])\}) > 0$ such that:
\begin{align}
    \frac{b-a}{3\eps}Q_2e^{-(m_1+m_2)P(\phi) + \int_{-m_2}^{m_1}\phi(g_s[x,y])\,ds} \nonumber & \leq  \mu(g_{[a,b]}[B_{m_1}^u(x,\eps),B_{m_2}^s(y,\eps)]) \nonumber \\
        & \leq \frac{b-a}{3\eps}Q_1e^{-(m_1+m_2)P(\phi) + \int_{-m_2}^{m_1}\phi(g_s[x,y])\,ds}. \nonumber
\end{align}

\end{corollary}

\begin{proof}
First, observe that the upper and lower Gibbs properties described in Corollary \ref{cor: unif Lower Gibbs} and Proposition \ref{prop: Upper Gibbs}, respectively, apply in this setting, because specification is proved in \cite[Theorem 4.1]{BCFT} for the set of orbit segments $(x,t)$ such that $\lambda(x)\geq \eta$ and $\lambda(g_tx)\geq \eta$ for some scale $\eta$, and $\phi$ has the global Bowen property by Theorem~\ref{thm: NPC global Bowen}. As the upper Gibbs property only holds for sufficiently small radii, fix $\delta' > 0$ so that both the upper and lower Gibbs properties hold at all scales less than $\delta'$.

Now, by our choice of $a$,$b$, and the fact that $2\eps < \frac{s_0}{2}$, we can apply Lemma \ref{lem: NPC ratio of flow boxes} to show
$$\mu(g_{[a,b]}[B_{m_1}^u(x,\eps),B_{m_2}^s(y,\eps)]) = \frac{\frac{2}{3\eps}(b-a)}{2}\mu(g_{\left[-\frac{3\eps}{2},\frac{3\eps}{2}\right]}[B_{m_1}^u(x,\eps),B_{m_2}^s(y,\eps)]).$$
The statement now follows immediately from Lemma \ref{lemma: NPC Bowen balls containment}, taking $\bar\delta$ small enough so that $\max\{T_1\bar\delta, (T_2+2e^{\Lambda}\chi^2+2)\bar\delta\} < \min\left\{\delta',\frac{\rho}{4e^\Lambda}\right\}$ where $\Lambda$ as in Lemma~\ref{lemma: metric comparison} and $\chi$ as in Lemma~\ref{lemma: topological LPS}.
\end{proof}

Towards the proofs of Proposition \ref{prop: Holder to Bowen} and Lemma \ref{lem: NPC Bowen}, we recall some definitions for the setting of geodesic flow on nonpositively curved manifolds and results in \cite{BCFT}.

\begin{definition}
A potential $\phi\colon T^1M\to \mathbb{R}$ is H\"older along stable leaves if there exist constants $C,\theta,\eps> 0$ such that for any $v\in T^1M$ and $w\in W^s(v,\eps)$, then $|\phi(v) - \phi(w)| \leq Cd^s(v,w)^{\theta}$. A similar definition can be made for unstable leaves.
\end{definition}

\begin{definition}
A potential $\phi\colon T^1M\to \mathbb{R}$ has the Bowen property along stable leaves with respect to a collection of orbit segments $\CCC$ if there exist $\delta, K > 0$ such that 
$$\sup\left\{\left|\int_0^t\phi(g_rv) - \phi(g_rw)\,dr\right|\mid (v,t)\in \CCC, w\in W^s(v,\delta)\right\}\leq K.$$
\end{definition}

Using $\lambda$, one can show definite contraction and expansion rates on the stable and unstable leaves in the following manner.

\begin{lemma}\cite[Lemma 3.10]{BCFT}\label{lemma: BCFT contraction lemma}
	Given $\eta > 0$, fix $\bar\rho :=\bar\rho(\eta) > 0$ small enough so that if $d_K(v,w)\leq \bar\rho e^{\Lambda}$, then $|\lambda(v) - \lambda(w)|\leq \frac{\eta}{2}$. Now, for any $v\in T^1M$ and $w,w'\in W^s(v,\bar\rho)$, for all $t\geq 0$ we have
	$$d^s(g_tw,g_tw') \leq d^s(w,w')e^{-\int_0^t\max\{\lambda(g_sv) - \frac{\eta}{2},0\}\,ds}.$$
	Similarly, if $w,w'\in W^u(v,\bar\rho)$, for all $t\geq 0$,
	$$d^u(g_{-t}w,g_{-t}w') \leq d^u(w,w')e^{-\int_0^t\max\{\lambda(g_{-s}v) - \frac{\eta}{2},0\}\,ds}.$$
\end{lemma}

In order to use this lemma, we need to prove a technical result about the growth of $\int_0^t\lambda(g_sx)\,ds$ when we know that $\lambda(g_sx)$ is greater than $\eta$ with a set frequency.

\begin{lemma}\label{lemma: lambda estimate}
Let $(x,t)\in T^1M\times [0,\infty)$ and suppose that there exists $\eta, T > 0$ such that $\lambda(x) \geq \eta$ and $\sup_{r\in [s-T,s+T]}\lambda(g_rx) \geq \eta$ for all $s\in [0,t]$. Then there exists $\eta' := \eta'(\eta, T) > 0$ such that for all $r\in [0,t]$,
$\int_{0}^{r}\max\{\lambda(g_sx) - \frac{\eta}{2},0\}\,ds \geq r\eta'$.
\end{lemma}

\begin{proof}
Take $\bar\rho := \bar\rho(\frac{\eta}{2}) < 2T$ as above (using uniform continuity of $\lambda$). Then, if $r\in [-\bar\rho,\bar\rho]$, it follows that $d_K(g_rv,v) \leq \bar\rho \leq \bar\rho e^{\Lambda}$. So, for all $v\in T^1M$, if $\lambda(v) \geq \eta$, then $\lambda(g_rv) - \frac{\eta}{2} \geq \lambda(v) - \frac{\eta}{4} - \frac{\eta}{2} \geq \frac{\eta}{4}$ for $r\in [-\bar\rho,\bar\rho]$. Now, let $r\leq 3T$. Then
$$\int_{0}^{r}\max\left\{\lambda(g_sx) - \frac{\eta}{2},0\right\}\,ds \geq \min\{\bar\rho,r\}\frac{\eta}{4} = r\min\left\{1,\frac{\bar\rho}{r}\right\}\frac{\eta}{4}\geq r\frac{\bar\rho}{3T}\frac{\eta}{4}.$$
On the other hand, if $r \geq 3T > \bar\rho$, we have
\begin{align*}
\int_0^r\max\left\{\lambda(g_sx) - \frac{\eta}{2},0\right\}\,ds &\geq \sum_{k=0}^{\left\lfloor \frac{r}{2T}\right\rfloor - 1} \int_{2Tk}^{2T(k+1)}\max\left\{\lambda(g_sx) - \frac{\eta}{2},0\right\}\,ds
\\
&\geq \left\lfloor \frac{r}{2T}\right\rfloor \bar\rho \frac{\eta}{4}
\\
&\geq \left(\frac{r}{2T} - 1\right)\frac{\eta\bar\rho}{4}
\\
&\geq r\left(\frac{1}{2T} - \frac{1}{3T}\right)\frac{\eta\bar\rho}{4}
\\
&= r\frac{\eta\bar\rho}{24T}.	
\end{align*}
Thus, taking $\eta' = \frac{\eta\bar\rho}{24T}$, we complete our proof.
\end{proof}

We also prove the following proposition, which will be used in place of \cite[Lemma 7.3]{BCFT}. The proof requires that we make use of Lemma \ref{lemma: lambda estimate} to upgrade the result to $\CCC(\eta)$.

\begin{proposition}\label{prop: Holder to Bowen}
Suppose $\phi : T^1M\to \mathbb{R}$ is H\"older along (un)stable leaves and is locally constant on $B(\Sing,2\delta)$. Then $\phi$ has the Bowen property on (un)stable leaves with respect to $\CCC(\eta)$ for any $\eta > 0$.
\end{proposition}

\begin{proof}
Let $\eta > 0$. By \cite[Proposition 3.4]{BCFT}, as well as Lemma \ref{lem: general version of BCFT 3.4}, there exists $T,\eta_0 > 0$ such that for all $v\in T^1M$, if $\lambda(g_sv) \leq \eta_0$ for all $s\in [-T,T]$, then $v\in B(\Sing,\delta)$. As for all $\eta \geq \eta_0$, we have $\CCC(\eta_0) \supset \CCC(\eta)$, we will assume $\eta \leq \eta_0$. Now, given $(v,t)\in \CCC(\eta)$, we define a sequence of times $r_k,s_k$ so that, taking $s_1 = 0$, we have
$$r_k = \inf\{s\in [s_k,t]\mid \lambda(g_rv) > \eta \text{ for some }r\in [s - T, s + T]\},$$
and
$$s_k = \inf\{s\in (r_{k-1},t]\mid \lambda(g_rv)\leq \eta \text{ for all }r\in [s - T, s + T]\}.$$
If $s_k$ is defined and $r_k$ is not (e.g., if it would be the case that $r_k > t$), we take $r_k = t$. Defined in this manner, we have a sequence of intervals $\{[s_k,r_k]\}_{k=1}^n$ for which $g_sv\in B(\Sing,\delta)$ for all $s\in [s_k,r_k]$. Observe that from construction, for $k > 1$, we have $\lambda(g_{r_k+T}v) = \eta$, unless $r_k = t$, and additionally, $s_{k+1} \geq r_k + 2T$. Now, we use Lemma \ref{lemma: lambda estimate} to estimate $\int_{0}^{s}\max\left\{\lambda(g_rv) - \frac{\eta}{2},0\right\}\,dr$ when $s\notin [s_k,r_k]$. For this, observe that for any $s \in [r_i + T,s_{i+1}]$, we have
\begin{align*}
	\int_{0}^{s}\max\left\{\lambda(g_rv) - \frac{\eta}{2},0\right\}\,dr &\geq  \sum_{k=0}^{i-1}\int_{r_k+T}^{s_{k+1}}\max\left\{\lambda(g_rv) - \frac{\eta}{2},0\right\}\,dr + \int_{r_i+T}^{s}\max\left\{\lambda(g_rv) - \frac{\eta}{2},0\right\}\,dr
	\\
	&\geq\sum_{k=0}^{i-1}(s_{k+1} - (r_k + T))\eta' + (s - (r_i + T))\eta'.
\end{align*}
Now, let $\bar\rho :=\bar\rho(\eta)$ be as in Lemma \ref{lemma: BCFT contraction lemma} and $\delta'$ as in the definition of the Bowen property along stable leaves. Assume without loss of generality that $\bar\rho \leq \min\{\delta,\delta'\}$. Then for any $w\in W^s_{\bar\rho}(v)$, by Lemma \ref{lemma: BCFT contraction lemma}, we have that for any $s\in [r_i + T,s_{i+1}]$, we have
$$d^s(g_sw,g_sv) \leq d^s(w,v)e^{-\sum_{k=0}^{i-1}(s_{k+1} - (r_k + T))\eta' - (s - (r_i + T))\eta'}.$$
Therefore, using the fact that $\phi$ is H\"older on stable leaves, it follows that for any such $s$, we have
$$|\phi(g_sw) - \phi(g_sv)| \leq Ce^{-\theta\eta'\left(s -(r_i + T) + \sum_{k=0}^{i-1}(s_{k+1} - (r_k + T))\right)}$$
while for $s\in [r_i,r_i+T]$, using the fact that $s_{k+1} > r_{k} + 2T$ for $k > 1$, we just have
$$|\phi(g_sv) - \phi(g_sw)| \leq Ce^{-\theta\eta' \sum_{k=0}^{i-1}(s_{k+1} - (r_k + T))} \leq Ce^{-(i-1)\theta\eta'T}.$$
Finally, if $s\in [s_i,r_i]$, then $v\in B(\Sing,\delta)$, and so $w\in B(\Sing,2\delta)$. Therefore, $\phi(g_sw) = \phi(g_sv)$. Combining all of these results, we have that

\begin{align*}
\left| \int_{0}^{t}\phi(g_sw) - \phi(g_sv)\,ds\right| &\leq \sum_{i=0}^{n} \int_{r_i}^{r_i + T}\left|\phi(g_sw) - \phi(g_sv)\right|\,ds + \int_{r_i+T}^{s_{i+1}}\left|\phi(g_sw) - \phi(g_sv)\right|\,ds
\\
&\leq \sum_{i=0}^{n} TCe^{-\theta\eta'\sum_{k=0}^{i-1}(s_{k+1}-(r_k+T))} + \int_{r_i+T}^{s_{i+1}}\left|\phi(g_sw) - \phi(g_sv)\right|\,ds
\\
&\leq \sum_{i=0}^{n}TCe^{-(i-1)\theta\eta'T} + \int_{r_i+T}^{s_{i+1}}Ce^{-\theta\eta'\left(s -(r_i + T) + \sum_{k=0}^{i-1}(s_{k+1} - (r_k + T))\right)}\,ds
\\
&\leq  TC\sum_{i=0}^{\infty}e^{-(i-1)\theta\eta'T} + \int_{0}^{\infty}Ce^{-s\theta\eta'}\,ds
\\
&<\infty.
\end{align*}
As this is constant and independent of our choice of $v$, we are done.
\end{proof}

\begin{lemma}\label{lem: NPC Bowen}
Suppose $\phi$ has the Bowen property on $\CCC(\eta/2)$ with respect to stable and unstable leaves for some $\eta > 0$. Then $\phi$ has the Bowen property on $\CCC(\eta)$.
\end{lemma}

\begin{proof}
The proof is almost exactly the same as that of \cite[Lemma 7.4]{BCFT}, which shows the same result, just working with orbit segments in the collection $\GGG(\eta) \subset \CCC(\eta)$. The only change needed is that they use an earlier lemma \cite[Lemma 3.11]{BCFT} to show that there exists $\delta_2 > 0$ such that if $w\in B_t(v,\delta)$ for some $(v,t)\in\GGG(\eta)$, then $(w,T)\in\GGG(\eta/2)$. For us, replacing $\GGG$ with $\CCC$, this follows immediately from uniform continuity of $\lambda$.
\end{proof}





%
\subsection{Proof of the local product structure}\label{sec:proof of LPS rank 1}

As shown at the end of Section \ref{sec:lps} (specifically in Proposition \ref{prop: LPS} and Corollary \ref{cor: ae lps}) to prove Theorem \ref{thm:rank 1}, it is sufficient to prove an analogue of Proposition \ref{prop: main prop}. We accomplish this with Proposition \ref{prop: NPC main prop} below.

Throughout this section we use the following notations. Let $\eps_0$ be the scale at which $\phi$ has the Bowen property. Let $z_0\in T^1M$ be such that $\lambda(z_0) \geq \frac{3}{4}\lVert\lambda\rVert_{\infty}$, and, using continuity of $\lambda$, let $\kappa > 0$ be chosen so that $\lambda|_{B(z_0,3\kappa)} \geq \frac{2}{3}\lambda(z_0)$. Fix $c_0\in \left(0,\frac{\lVert\lambda\rVert_{\infty}}{4}\right)$ and $\bar\rho := \bar\rho(c_0)$ as in Lemma \ref{lemma: BCFT contraction lemma}. Let $x_0\in \Reg\cap G_{\mu}$ and write $\eta_0 := \lambda(x_0)$. Further, take $\rho:= \min\{\rho(\eta_0),\bar\rho\}$ and $\chi := \chi(\eta_0/2)$ as in Lemma~\ref{lemma: topological LPS}. Now, choose $n_0:=\max\{\frac{4}{s_0},\frac{4}{\rho},\frac{4}{\kappa},\frac{4}{\eps_0}\}$ and $\eps < \min\{\bar\delta(\eta_0),\frac{\rho}{4\chi(1+\chi)(1+e^{\Lambda})},\frac{\eps_0}{4(1+\chi)}\}$ with $\mu(\mathcal B(n_0,\eps)\setminus\operatorname{Int}\mathcal B(n_0,\eps)) = 0$ which holds for almost every small $\eps$ by the same argument as in Lemma~\ref{lem:zero msr bdry}. Recall $\mathcal B(n_0,\eps)$ is the corresponding $(n_0,\eps)$-flow box (see Definition~\ref{defn: flow box nonpositive}).

As in Section~\ref{sec:lps}, we first show how to obtain a compact subset of $V_x^u$ which satisfies the necessary properties to apply the results of \S\ref{sec: Partitions}. In particular, we need the appropriate versions of Lemma ~\ref{lem: generic times} and Propositions~\ref{prop: Creation of compact} and \ref{prop: integral_limit}. The proofs provided for the geodesic flows on flat surfaces work for the geodesic flows on non-positively curved manifolds with minor changes. If a more significant change is needed, we will discuss it next to the statement.

Before we proceed, we note that for any $y\in \mathcal{B}(n_0,\eps)$ let $\pi_y: \mathcal{B}(n_0,\eps) \to V_y^u$ and $\pi_{x,y}: V^u_x \to V^u_y$ be as in Definition~\ref{defn:pi maps}. Moreover, those maps are continuous as the foliations that define them are (see, e.g., \cite{Eberlein}).

One complication that arises in this setting is the fact that $\lambda(g_t[w,z])$ is no longer equal to $\lambda(g_tw)$ or $\lambda(g_tz)$ for sufficiently large $t$. In this setting, we leverage uniform continuity of $\lambda$ and the contraction along local strong (un)stable manifolds to obtain a statement similar to Lemma \ref{lem: bracket and flow interaction}.

\begin{lemma}\label{lemma: NPC contraction of stables}\label{lemma: NPC unstable contraction}
Let $c\in (0,\lVert\lambda\rVert_{\infty})$, and let $\bar\rho := \bar\rho(c) > 0$ be as in Lemma \ref{lemma: BCFT contraction lemma}. Then for all $x\in T^1M$ and $t\geq 0$, if $\lambda(g_tx)\geq c$ and $y\in W^{cs}(x,\bar\rho)$, then $\lambda(g_ty)\geq \frac{c}{2}$. Alternatively, $\lambda(g_{-t}x) \geq c$, then for all $z\in W^u(x,\bar\rho)$, we have $\lambda(g_{-t}z)\geq \frac{c}{2}$.
\end{lemma}

\begin{proof}
Suppose $\lambda(g_tx)\geq c$ and $y\in W^{cs}(x,\bar\rho)$. Then $d_K(g_ty,g_tx) \leq d_K(x,y) \leq d^{cs}(x,y) \leq \bar\rho \leq \bar\rho e^{\Lambda}$. Therefore, by the definition of $\bar\rho$, it follows that
$$\lambda(g_ty) \geq \lambda(g_tx) - |\lambda(g_tx) - \lambda(g_ty)| \geq c - \frac{c}{2} = \frac{c}{2}.$$
For the second half of the statement, observe that $$d_K(g_{-t}z,g_{-t}x) \leq d_K(z,x) \leq e^{\Lambda}d^u(z,x) \leq e^{\Lambda}\bar\rho.$$
The remainder of the proof is identical.
\end{proof}

\begin{lemma}\label{lem: NPC generic times}(cf. Lemma~\ref{lem: generic times})
Let $c_0\in \left(0,\frac{\|\lambda\|_{\infty}}{4}\right)$, and $z_0\in T^1M$ and $\kappa$ be as above. For any $z\in G_{\mu}$, there exists an increasing sequence $\{t_k\}_{k\in\mathbb{Z}}$ which tends to $\pm\infty$ as $k\to\pm\infty$, respectively, such that $\lambda(g_{t_k+s}z)\geq 2c$ for all $|s|\leq \kappa$.
\end{lemma}

\begin{proposition}\label{prop: NPC Creation of compact} (cf. Proposition~\ref{prop: Creation of compact})
Let $c_0 \in \left(0,\frac{\|\lambda\|_{\infty}}{4}\right)$, and $z_0\in T^1M$ and $\kappa$ be as above. Let $x\in\mathcal{B}(n_0,\eps)\cap G_{\mu}$. Then for all $|a|,|b| \leq \frac{1}{n_0}$ and $t\geq 0$, and for all $\delta > 0$, we can find a compact subset $K\subset V^u_x\cap \operatorname{Int}\mathcal{B}(n_0,\eps)$ such that:
 \begin{itemize} 
 \item[(a)] for all $z\in K$, there exists a sequence $t_{k}\to\infty$ such that $\lambda(g_{t_k+s}z)\geq 2c$ for all $|s|\leq \kappa$;
 \item[(b)] $\mu(g_{[a,b]}\left[K, B_{t}^s(x,\eps)\right])\geq (1-\delta)\mu(g_{[a,b]}\left[ V_x^u, B_{t}^s(x,\eps)\right])$.
\end{itemize}
\end{proposition}

\begin{proposition}\label{prop: NPC integral_limit}(cf. Proposition~\ref{prop: integral_limit})
Let $c_0$ and $x_0$ be as above. There exists $R'\subset \mathcal B(n_0,\eps)$ with $\mu(B(n_0,\eps)\setminus R') = 0$ such that for all $x,y\in R'$, there exists a sequence of partitions of $\mathcal{B}(n_0,\eps)$, denoted $\{\xi_n\}_{n\in\mathbb{N}}$, such that for every continuous $\psi:\mathcal{B}(n_0,\eps) \to \mathbb{R}$,
$$\int_{V^u_x} \psi\,d\mu^u_x = \lim_{n\to\infty}\frac{1}{\mu(\xi_n(x))}\int_{\xi_n(x)}\psi\,d\mu,$$
and similarly for $y$. Furthermore, for all sufficiently large $n$, we have that
$$g_{\left(\frac{k}{2^n},\frac{k+1}{2^n}\right)}\left[V^u_{x_0},B^s_{t_z}\left(z,\frac{\eps}{2}\right)\right]\subset \xi_n(x) \subset g_{\left(\frac{k}{2^n},\frac{k+1}{2^n}\right)}\left[ V^u_{x_0},B^s_{t_z}(z,\eps)\right]$$
for some $z\in V^s_{x_0}$, $\frac{-2^n}{n_0}\leq k \leq \frac{2^n}{n_0}-1$, and $t_z\geq 0$ such that for all $z'\in W^u(z,\rho)$, $\lambda(g_{-t_z}z')\geq \frac{c_0}{2}$.
We have a similar statement for $\xi_n(y)$.
\end{proposition}

\begin{proof}
This proof proceeds in the same manner as Proposition \ref{prop: integral_limit}. The main difference is that instead of simply showing that $\lambda(g_{-t_z}z)\geq c_0$ as we did previously, we use Lemma \ref{lemma: NPC contraction of stables} and the fact that $\rho \leq \bar\rho$. Thus, the additional statement that for all $z'\in W^u(z,\rho)$ we have $\lambda(g_{-t_z}z')\geq \frac{c_0}{2}$ follows.
\end{proof}

\begin{proposition}\label{prop: NPC main prop}(cf. Proposition~\ref{prop: main prop})
There exists $K > 0$ such that for almost every $x,y\in \mathcal B(n_0,\eps)$ and for every uniformly continuous $\psi\colon \mathcal{B}(n_0,\eps) \to (0,\infty)$,
$$\int_{V^u_y}\psi\,d(\pi_{x,y})_*\mu_x^u \leq K\int_{V^u_y}\psi\,d\mu_y^u.$$
\end{proposition}

\begin{proof}

 	Let $\psi : \mathcal{B}(n_0,\eps)\to (0,\infty)$ be uniformly continuous, and let $\alpha > 0$ be small enough so that if $d_K(u,u') \leq 3\alpha$, then both $\frac{1}{2}\psi(u') \leq \psi(u) \leq 2\psi(u')$, and $\frac{1}{2}\psi(\pi_yu') \leq \psi(\pi_y u) \leq 2\psi(\pi_y u')$. Now let $R'\subset \mathcal{B}(n_0,\eps)$ be as in Proposition \ref{prop: NPC integral_limit}, and let $x,y\in R'\cap \operatorname{Int}\mathcal B(n_0,\eps)\cap G_{\mu}$. Observe that this is a full measure set, from applying Proposition \ref{prop: NPC integral_limit} and the fact that $\mu(\mathcal B(n_0,\eps)\setminus\operatorname{Int}\mathcal B(n_0,\eps)) = 0$ by the choice of $\eps$. Now, both $\psi\circ\pi_{x,y}$ and $\psi$ are continuous functions, and so we can apply Proposition \ref{prop: NPC integral_limit} to estimate both
	$$\int_{V^u_x}\psi\circ\pi_{x,y}\,d\mu_x^u = \int_{V^u_y}\psi\,d(\pi_{x,y})_*\mu_x^u \quad \text{ and } \quad \int_{V^u_y}\psi\,d\mu_y^u.$$

 To do so, recall that we have fixed $x_0\in \mathcal B(n_0,\eps)$ and $c_0\in \left(0,\frac{\lVert\lambda\rVert_{\infty}}{4}\right)$, and let $\{\xi_n\}_{n\in\mathbb{N}}$ be the sequence of partitions given by Proposition \ref{prop: NPC integral_limit}. We will now proceed with our proof by using Propositions \ref{prop: NPC Creation of compact} and \ref{prop: partition construction} to partition most of $V_{x_0}^u$ (and then, in turn $V_x^u$ and $V_y^u$) into well-understood sets, for which we have nice upper and lower bounds from the Gibbs property.

 Let $0 < \beta < \min\{1,\lVert\psi\rVert_\infty^{-1}\}$. For each $n\in\mathbb{N}$, recall that by Proposition \ref{prop: NPC integral_limit}, there exists $k_{x,n},k_{y,n}\in [-\frac{2^n}{n_0},\frac{2^n}{n_0} - 1]$, $t_{x,n},t_{y,n}\geq 0$ and $z_{x,n},z_{y,n}\in V^s_{x_0}$ such that
	$$g_{\left(\frac{k_{x,n}}{2^n},\frac{k_{x,n}+1}{2^n}\right)}\left[ V^u_{x_0},B^s_{t_{x,n}}\left(z_{x,n},\frac{\eps}{2}\right)\right]\subset \xi_n(x) \subset g_{\left(\frac{k_{x,n}}{2^n},\frac{k_{x,n}+1}{2^n}\right)}\left[V^u_{x_0},B^s_{t_{x,n}}(z_{x,n},\eps)\right],$$
	and similarly for $y$. Going forward, we will often omit the dependence on $n$ from the notation since we will primarily work with a fixed choice of $n$. Now, let $K_n\subset V_{x_0}^u$ be as in Proposition \ref{prop: NPC Creation of compact} for $c_0$ with
	$$\mu\left(g_{\left(\frac{k_{x,n}}{2^n},\frac{k_{x,n} + 1}{2^n}\right)}\left[ K_n,B^s_{t_{x,n}}(z_{x,n},\eps)\right]\right)\geq (1-\beta^2\mu(\xi_n(x))\mu\left(g_{\left(\frac{k_{x,n}}{2^n},\frac{k_{x,n} + 1}{2^n}\right)}\left[ V_{x_0}^u,B^s_{t_{x,n}}(z_{x,n},\eps)\right]\right),$$
	and similarly replacing every $x$ with $y$. Then, use Proposition \ref{prop: partition construction} and Corollary \ref{cor: fill out partition} along with Lemma \ref{lemma: NPC contraction of stables} to define a partition $\omega_n := \{A_0,A_1,\cdots, A_m\}$ of $V^u_{x_0}$, such that:
	\begin{itemize}
		\item For each $A_i\in \omega_n$, $i\neq 0$, there exists $w_i\in A_i$ and $t_i > t_R + \frac{2}{n_0}$ such that $B_{t_i}^u(w_i,\frac{\eps}{2}) \subset A_i\subset B_{t_i}^u(w_i,2\eps)$
		\item For $i\neq 0$, $\lambda(g_{t_i} w') \geq \frac{c_0}{2}$ for all $w'\in W^{cs}(w_i,\rho)$
            \item $A_0\subset V^u_{x_0}\setminus K_n$
		\item For $i\neq 0$, $\max\{\diam A_i \mid A_i\in\omega_n\} \leq \min\{\hat\delta,d_K(K_n,\partial \mathcal{B}(n_0,\eps))\}$ where $\hat\delta$ is defined as follows.

   Having fixed $x_0$ and $\epsilon$ already, consider the map $V_{x_0}^u \times V_{x_0}^s \to T^1M$ by $(a,b) \mapsto \pi_b(a) = [a,b].$ By continuity of the foliations defining the bracket, this is a continuous map, and over this domain it is uniformly continuous. So, given any $\alpha>0$, there exists a $\hat\delta>0$ such that for all $a_1, a_2\in V_{x_0}^u$ and all $b_1, b_2\in V_{x_0}^s$, $d_K(a_1, a_2)<\hat\delta$ and $d_K(b_1, b_2)<\hat\delta$ implies $d_K([a_1,b_1],[a_2,b_2])<2\alpha$.
 	\end{itemize}
	The second condition follows from the fact that $\rho\leq\bar\rho$ and using Lemma \ref{lemma: NPC contraction of stables}.

Going forward, take $n$ so large that $\diam B^s_{t_{n,x}}(z_{n,x},\eps), \diam B^s_{t_{n,y}}(z_{n,y},\eps)< \hat \delta$. Below, still taking $n$ large enough that $\frac{1}{2^n}\leq \alpha$, this will guarantee that $\diam g_{\left(\frac{k_{n,x}}{2^n},\frac{k_{n,x} + 1}{2^n}\right)}\left[A_i,B_{t_{n,x}}^s(z_{n,x},\eps)\right]\leq 3\alpha$, and similarly with $y$ in place of $x$.

For $t\leq 3\alpha$, $\frac{1}{2}\psi(\pi_y g_tz)\leq \psi(\pi_y z) \leq 2\psi(\pi_y g_tz)$ for $z\in \mathcal B(n_0,\eps)$. For all such $n$, using the fact that $\pi_y = \pi_y\circ g_{\tau}$ for $\tau\in \left[-\frac{1}{n_0},\frac{1}{n_0}\right]$, we have
 	\begin{align*}
		\int_{\xi_n(x)}\psi\circ \pi_y\,d\mu &\leq \int_{g_{\left(\frac{k_x}{2^n},\frac{k_x+1}{2^n}\right)}\left[V^u_{x_0},B_{t_x}^s(z_x,\eps)\right]}\psi\circ\pi_y\,d\mu
		\\
		&\leq \beta^2\lVert\psi\circ\pi_y\rVert_{\infty}{\mu(\xi_n(x))} + \sum_{i=1}^{m}\int_{g_{\left(\frac{k_x}{2^n},\frac{k_x+1}{2^n}\right)}\left[ A_i,B_{t_x}^s(z_x,\eps)\right]}\psi\circ\pi_y\,d\mu
		\\
		&\leq \beta{\mu(\xi_n(x))} + \sum_{i=1}^{m}2\psi(\pi_y[ w_i,z_x])\frac{Q_1}{6\eps\cdot 2^n}e^{-(t_i+t_x)P(\phi) + \int_{-t_x}^{t_i}\phi(g_r[ w_i,z_x])\,dr}.
	\end{align*}

	The final step uses Corollary \ref{cor: NPC Gibbs}, taking $Q_1 := Q_1(2\eps)$. 
 We will also need the following lower bound:
\begin{align*}
    \mu(\xi_n(x))&\geq \mu\left(g_{\left(\frac{k_{x}}{2^n}, \frac{k_{x}+1}{2^n}\right)}\left[ V_{x_0}^u, B^s_{t_{x}}\left(z_{x},\frac{\varepsilon}{2}\right)\right]\right)\\
    &\geq \sum\limits_{i=1}^m\mu\left(g_{\left(\frac{k_{x}}{2^n}, \frac{k_{x}+1}{2^n}\right)}\left[ B_{t_i}^u\left(w_i,\frac{\varepsilon}{2}\right), B^s_{t_{x}}\left(z_{x},\frac{\varepsilon}{2}\right)\right]\right).
\end{align*}

From this, using Corollary \ref{cor: NPC Gibbs}, it follows that

$$\mu\left(\xi_n(x)\right) \geq \frac{2}{3\eps\cdot 2^n}\sum\limits_{i=1}^m Q_2\left(\frac{\eps}{2},\lambda(g_{-t_{x}}[w_i,z_x]),\lambda(g_{t_i} [w_i,z_x])\right)e^{-(t_i+t_x)P(\phi) + \int_{-t_x}^{t_i}\phi(g_s[ w_i,z_x])\,ds}.$$
Note that, since $[w_i,z_x] = W^{cs}(w_i,\rho)\cap W^{u}(z_x,\rho)$, we have by construction that $\lambda(g_{-t_x}[w_i,z_x])\geq \frac{c}{2}$ and $\lambda(g_{t_i}[w_i,z_x])\geq \frac{c}{2}$.

Finally, we will need the following computations using the Bowen property to make use of these results. Recall that $\phi$ has the Bowen property at scale $\eps_0$, and by the choice of $\eps$ and $n_0$ that $\diam \mathcal{B}(n_0,\eps) < \eps_0$. Now, observe that $d_{K}(g_r[w,z], g_r z)\leq d_K([w,z],z) \leq \eps_0$ for $r\leq 0$ for all $w,z\in \mathcal B(n_0,\eps)$ using the fact that $[w,z]\in \mathcal{B}(n_0,\eps)$, while $d_{K}(g_r[w,z], g_rw)\leq d_K([w,z],w) \leq\eps_0$  for $r\geq 0$. Therefore, we have $[w,z] \in B_{[-n,0]}(z,\eps_0)$ and $[ w,z] \in B_{[0,n]}(w,\eps_0)$ for all $n\geq 0$. By Theorem~\ref{thm: NPC global Bowen}, there exists $L > 0$ such that for all $n_1,n_2 \geq 0$,
$$\left|\int_{-n_1}^{n_2}\phi(g_r[ w,z])\,dr - \int_{-n_1}^{0}\phi(g_rz)\,dr - \int_0^{n_2}\phi(g_rw)\,dr\right| \leq 2L.$$

Therefore, recalling that $\lambda(g_{-t_x}z_x)\geq c$ by construction and writing $Q_2 := Q_2\left(\frac{\eps}{2},\frac{c}{2},\frac{c}{2}\right)$, we have that: 
\[\frac{\int_{\xi_n(x)}\psi\circ \pi_y\,d\mu}{\mu(\xi_n(x))} \leq \beta +\frac{Q_1e^{4L}}{2Q_2}\cdot\frac{\sum\limits_{i=1}^m \psi(\pi_y [w_i,z_x])e^{-t_iP(\phi)+\int_0^{t_i}\phi(g_r  w_i)\,dr}}{\sum\limits_{i=1}^m e^{-t_iP(\phi)+\int_0^{t_i}\phi(g_s w_i)\,ds}}.\]
Note that $\beta$ does not depend on our choice of $n$, although the rest of the computation does, as our collection of $z_i$ and $t_i$ depends on our choice of $K_n$.

Now, we will carry out similar computations to find a lower bound for the corresponding estimates of $\int_{V^u_y}\psi\,d\mu_y^u$. These mirror those carried out above, except using the lower Gibbs bound instead of the upper Gibbs bound and vice versa. 	

	Recalling that we work with $n$ large enough so that $\diam g_{\left[\frac{k_y}{2^n},\frac{k_y+1}{2^n}\right]}[ A_i,B_{t_y}^s(z_y,\eps)] \leq 3\alpha$, and that $t_y$ is chosen so that $\lambda(g_{-t_y}z_y) \geq c$, we obtain a lower bound as follows. First, observe that $V^u_y\subset g_{\left(\frac{k_y}{2^n},\frac{k_y+1}{2^n}\right)}\left[ V^u_{x_0},B^s_{t_y}(z_y,\eps)\right]$. Therefore, for every $w_i\in V^u_{x_0}$ we have
 $\pi_y[w_i,z_y] \in g_{\left(\frac{k_y}{2^n},\frac{k_y+1}{2^n}\right)}\left[ V^u_{x_0},B^s_{t_y}(z_y,\eps)\right]$, and consequently,
 $$\psi|_{g_{\left(\frac{k_y}{2^n},\frac{k_y+1}{2^n}\right)}\left[V^u_{x_0},B^s_{t_y}(z_y,\eps)\right]} \geq \frac{\psi(\pi_y[ w_i,z_y])}{2}.$$
 Using this, we see

\begin{align*}
\int_{\xi_n(y)}\psi\,d\mu &\geq \int_{g_{\left(\frac{k_y}{2^n},\frac{k_y+1}{2^n}\right)}\left[ V^u_{x_0},B_{t_y}^s(z_y,\frac{\eps}{2})\right]}\psi\,d\mu
\\
&\geq \sum_{i=1}^{m}\int_{g_{\left(\frac{k_y}{2^n},\frac{k_y+1}{2^n}\right)}\left[ B_{t_i}^u(w_i,\frac{\eps}{2}),B_{t_y}^s(z_y,\frac{\eps}{2})\right]}\psi\,d\mu
\\
&\geq \frac{2Q_2}{3\eps\cdot 2^n}e^{-L}e^{\int_{-t_y}^{0}\phi(g_rz_y)\,dr}e^{-t_yP(\phi)}\sum_{i=1}^m \frac{\psi(\pi_y[w_i,z_y])}{2}e^{-t_iP(\phi) + \int_{0}^{t_i}\phi(g_rw_i)\,dr}.
\end{align*}
where the last inequality utilizes the Bowen property as we did in our previous estimates.

We also have the following upper bound:
	\begin{align*}
	\mu(\xi_n(y)) &\leq \mu\left(g_{\left[\frac{k_y}{2^n},\frac{k_y+1}{2^{n}}\right]}\left [ V^u_{x_0},B_{t_y}^s(z_y,\eps)\right]\right)
	\\
	&\leq \frac{1}{1-\beta^2}\sum_{i=1}^{m} \mu\left(g_{\left[\frac{k_y}{2^n},\frac{k_y+1}{2^{n}}\right]}\left[ A_i,B_{t_y}^s(z_y,\eps)\right]\right)
	\\
	&\leq\frac{1}{1-\beta^2}Q_1(6\eps\cdot 2^n)^{-1}e^{2L}e^{\int_{-t_y}^0\phi(g_rz_y)\,dr}\sum_{i=1}^{m}e^{-(t_i+t_y)P(\phi) + \int_{0}^{t_i}\phi(g_r w_i)\,dr}
	\end{align*}

Combining these two estimates, we have that for sufficiently large $n$,
	
\begin{align*}
\frac{\int_{\xi_n(y)}\psi\,d\mu}{\mu(\xi_n(y))} &\geq (1-\beta^2)\frac{2Q_2(3\eps\cdot 2^n)^{-1}e^{-L}e^{\int_{-t_y}^{0}\phi(g_rz_y)\,dr}e^{-t_yP(\phi)}\sum_{i=1}^m \frac{\psi(\pi_y[ w_i,z_y])}{2}e^{-t_iP(\phi) + \int_{0}^{t_i}\phi(g_rw_i)\,dr}}{Q_1(6\eps\cdot 2^n)^{-1}e^{2L}e^{\int_{-t_y}^0\phi(g_rz_y)\,dr}\sum_{i=1}^{m}e^{-(t_i+t_y)P(\phi) + \int_{0}^{t_i}\phi(g_r w_i)\,dr}}
\\
&\geq (1-\beta^2)\frac{2Q_2}{Q_1e^{4L}}\frac{\sum_{i=1}^{m}\psi(\pi_y[w_i,z_y]) e^{-t_iP(\phi) + \int_0^{t_i}\phi(g_rw_i)\,dr}}{\sum_{i=1}^{m}e^{-t_iP(\phi) + \int_0^{t_i}\phi(g_rw_i)\,dr}}.
\end{align*}
	
Therefore, for all large $n$ we can directly compare our estimates for $\int_{V_x^u}\psi\circ\pi_{x,y}\,d\mu_x^u$ and a lower bound for $\int_{V_y^u}\psi\,d\mu_y^u$. Recall that we have already shown that

$$\frac{\sum\limits_{i=1}^m \psi(\pi_y [ w_i,z_x])e^{-t_iP(\phi)+\int_0^{t_i}\phi(g_r w_i)\,dr}}{\sum\limits_{i=1}^m e^{-t_iP(\phi)+\int_0^{t_i}\phi(g_s w_i)\,ds}} \geq \frac{2Q_2}{Q_1e^{4L}}\left(\frac{\int_{\xi_n(x)}\psi\circ\pi_y\,d\mu}{\mu(\xi_n(x))} - \beta\right).$$
	
Further, $\pi_y[w_i,z_x]=[w_i,y] = \pi_y[w_i,z_y]$. Consequently, we have:
\begin{align*}
\frac{\int_{\xi_n(y)}\psi\,d\mu}{\mu(\xi_n(y))} &\geq (1-\beta^2)\frac{2Q_2}{Q_1e^{4L}}\frac{\sum_{i=1}^{m}\psi(\pi_y[ w_i,z_y]) e^{-t_iP(\phi) + \int_0^{t_i}\phi(g_rw_i)\,dr}}{\sum_{i=1}^{m}e^{-t_iP(\phi) + \int_0^{t_i}\phi(g_rw_i)\,dr}}
\\
&\geq (1-\beta^2)\left(\frac{2Q_2}{Q_1e^{4L}}\right)\left(\frac{2Q_2}{Q_1e^{4L}}\right)\left(\frac{\int_{\xi_n(x)}\psi\circ\pi_y\,d\mu}{\mu(\xi_n(x))} - \beta\right).
\end{align*}
	
We can choose $\beta$ arbitrarily small, independent of $n$. Thus, we have shown that
	
	\begin{align*}
		\frac{\int_{V_x^u}\psi\circ\pi_{x,y}\,d\mu_x^u}{\int_{V_y^u}\psi\,d\mu_y^u} &= \lim_{n\to\infty} \frac{\frac{\int_{\xi_n(x)}\psi\circ \pi_y\,d\mu}{\mu(\xi_n(x))}}{\frac{\int_{\xi_n(y)}\psi\,d\mu}{\mu(\xi_n(y))}}
		\\
		&\leq \left(\frac{e^{4L}Q_1}{2Q_2}\right)^2.
	\end{align*}
\end{proof}

\bibliography{bibliography}
\bibliographystyle{alpha}

\Addresses
\end{document}